\documentclass[12pt,oneside]{amsart}

\usepackage{geometry}   
\usepackage{todonotes}
\usepackage{color}
\usepackage{tikz-cd}
\usepackage{mathrsfs}
\usepackage[mathscr]{euscript}
\usepackage{bm}
\usepackage{bbm}
\usepackage{enumerate}
\usepackage{changepage}
\usepackage{wasysym}
\usepackage[utf8]{inputenc}
\inputencoding{utf8}
\usepackage{comment}
\usepackage{graphicx}
\usepackage{amsthm,amssymb,amsmath}
\usepackage{subfigure}
\allowdisplaybreaks[4]

\newtheorem{theorem}{Theorem}[section]
\newtheorem{lemma}[theorem]{Lemma}
\newtheorem{fact}[theorem]{Fact}
\newtheorem{proposition}[theorem]{Proposition}
\newtheorem{corollary}[theorem]{Corollary}

\theoremstyle{definition}
\newtheorem{definition}[theorem]{Definition}
\newtheorem{remarkx}[theorem]{Remark}
\newenvironment{remark}
  {\pushQED{\qed}\remarkx}
  {\popQED\endremarkx}

\newtheorem*{claimx}{Claim}
\newenvironment{claim}
  {\pushQED{\qed}\claimx}
  {\popQED\endclaimx}
\def\NN{\mathbb{N}}

\def\sgn{\mathrm{sgn}}

\def\RR{\mathbb{R}}
\def\TT{\mathbb{T}}
\def\ZZ{\mathbb{Z}}
\def\d{\,\mathrm{d}}
\def\e{\mathbbm{1}}

\def\id{\mathrm{id}_G}
\def\exp{\mathrm{exp}}
\def\mut{\widetilde{\mu}}
\def\dis{\mathfrak{d}}
\def\tri{\,\triangle\,}
\def\Stab{\mathrm{Stab}}

\title[Minimal and nearly minimal measure expansions]{Minimal and nearly minimal measure expansions in connected unimodular groups}

\author{Yifan Jing}
\address{Department of Mathematics, University of Illinois  Urbana-Champaign, Urbana IL, USA}
\email{yifanjing17@gmail.com}

\author{Chieu-Minh Tran}
\address{Department of Mathematics, University of Notre Dame, Notre Dame IN, USA}
\email{mtran6@nd.edu}

\thanks{YJ was supported by Arnold O. Beckman Research Award (Campus Research Board RB21011), by the Department Fellowship, and by the Trjitzinsky Fellowship from UIUC}

\date{}
\begin{document}

\begin{abstract}
Let $G$ be a connected unimodular group equipped with a (left and hence right) Haar measure $\mu_G$, and suppose $A, B \subseteq G$ are nonempty and compact. An inequality by Kemperman gives us  $\mu_G(AB)\geq\min\{\mu_G(A)+\mu_G(B),\mu_G(G)\}.$

Our first result determines the conditions for the equality to hold, providing a complete answer to a question asked by Kemperman in 1964.  
Our second result characterizes compact and connected $G$, $A$, and $B$ that nearly realize equality, with quantitative bounds having the sharp exponent.
This can be seen up-to-constant as a $(3k-4)$-theorem for this setting, and confirms the connected case of conjectures by Griesmer and by Tao. As an application, we get a measure expansion gap result for connected compact simple Lie groups.

The tools developed in our proof include an analysis of the shape of minimally and nearly minimally expanding pairs of sets,  a bridge from this to the  properties of a certain pseudometric, and a construction of  appropriate continuous group homomorphisms to either $\RR$ or $\TT = \RR/ \ZZ$ from the pseudometric.

\end{abstract}
\maketitle
\tableofcontents

\section{Introduction}

\subsection{Background}

The Cauchy--Davenport theorem asserts that if $X$ and $Y$ are nonempty subsets of the group $\ZZ/p\ZZ$ of prime order $p$, then 
$$|X+Y| \geq \min\{|X|+|Y|-1, p\},$$
where we set $X+Y :=\{ x+y : x\in X, y \in Y\}$. The condition for $X$ and $Y$ to have minimal expansion (i.e.  equality happens in the above inequality) is essentially given by Vosper's theorem~\cite{Vosper}, which states that if 
$$1< |X|, |Y|, \text{ and } |X+Y|= |X|+|Y|-1< p-1$$
then $X$ and $Y$ must be arithmetic progressions with the same common  difference. When $X$ and $Y$ have nearly minimal expansion
(i.e. the equality nearly happens), one might expect that $X$ and $Y$ are instead contained in arithmetic progressions with slightly larger cardinalities.  This was confirmed by Freiman~\cite{Freiman} for small $X=Y$ even though the optimal statement, known as the $(3k-4)$-conjecture for $\ZZ/p\ZZ$, remains wide open. 
Similar results were also obtained for other abelian groups; see e.g.~\cite{Kneser, Kem60, DF03, GreenRuzsa06, T18, G19, Lev20}. 

This paper considers Kemperman's analog of the Cauchy--Davenport theorem to connected unimodular locally compact nonabelian groups. We will determine the conditions for equality to happen  and the condition for equality nearly happen when the group is compact. Our work is inspired by recent progresses in the study of small expansions in the nonabelian settings; see~\cite{BGT}, in particular, for the classification of approximate groups by Breuillard, Green, and Tao; see also~\cite{Helfgott08, BGS10, BV12, Hrushovski, PS16, BM18, Hru20}.

Throughout, let $G$ be a connected locally compact group, and $\mu_G$ a left Haar measure on $G$. We further assume that $G$ is unimodular (i.e., the measure $\mu_G$ is invariant under right translation), so $\mu_G$ behaves like an appropriate notion of size. This assumption holds in  many  situations of interest (e.g, when $G$ is compact, discrete, a nilpotent Lie group, a semisimple Lie group, etc). 
As usual, for $A, B\subseteq G$, we set $AB:=\{ ab: a\in A, b\in B \}$ and let $A^n$ be the $n$-fold product of $A$ for $n \in \NN^{>0}$.   In~\cite{Kemperman},  Kemperman proved that if $A, B \subseteq G$ are nonempty and compact,  then 
$$
\mu_G(AB) \geq \min\{ \mu_G(A)+\mu_G(B), \mu_G(G)\}.
$$
This generalizes earlier results for one-dimensional tori,  $n$-dimensional tori, abelian groups, and second countable compact groups by  Raikov~\cite{Onedimtorus}, Macbeath~\cite{MacbeathTorus}, Kneser~\cite{Kneser}, and Shields~\cite{shields}.

 The problem of determining when equality happens for this inequality was proposed in the same paper~\cite{Kemperman}. After handling a number of easy cases, the problem can be reduced to 
classifying all connected and unimodular group $G$ and  pairs $(A,B)$ of compact subsets of $G$ such that
$$0< \mu_G(A), \mu_G(B),  \text{ and } \mu_G(AB)= \mu_G(A)+\mu_G(B) < \mu_G(G).$$
We call such $(A,B)$ a {\bf minimally expanding} pair on $G$.

It is easy to see that, if $I$ and $J$ are closed intervals in $\TT= \RR/\ZZ$, such that $I$ and $J$ have positive measures and the total of their measures is strictly smaller than $\mu_{\TT}(\TT)$, then $I+J$ is an interval with length the total lengths of $I$ and $J$. Hence, such $(I, J)$ is a minimally expanding pair on $\TT$. 
More generally, when $G$ is a compact group, $\chi: G \to \TT$ is a continuous surjective group homomorphism, $I$ and $J$ are as before,  
$$
A =\chi^{-1}(I) \quad\text{and}\quad  B= \chi^{-1}(J),
$$ 
we can check by using the Fubini theorem that $(A, B)$ is a minimally expanding pair. Note that an arithmetic progression on $\ZZ/ p\ZZ$ is the inverse image under a group homomorphism $\phi: \ZZ/p\ZZ \to \TT$ of an interval on $\TT$, so this example is the counterpart of Vosper's classification. 
Another obvious example is when $G$ is a noncompact group, $\chi: G \to \RR$ is a continuous surjective group homomorphism with compact kernel, $I$ and $J$ are nonempty compact intervals in $\RR$, $A =\chi^{-1}(I)$, and $B= \chi^{-1}(J)$. One might optimistically conjecture, in analogy with Vosper's theorem, that there are no other $G$, $A$, and $B$ such that $(A,B)$ is a minimally expanding pair on $G$.

In view of the earlier discussions, for compact $A, B \subseteq G$, we say that $(A,B)$ is a {\bf $\delta$-nearly minimally expanding pair} on $G$ if 
\[
0< \mu_G(A), \mu_G(B), \text{ and } \mu_G(AB)< \mu_G(A)+\mu_G(B)+ \delta \min\{\mu_G(A), \mu_G(B)\}  < \mu_G(G).
\]
We interpret the problem of determining when equality nearly happens in the Kemperman inequality as classifying all connected and unimodular groups $G$ and $\delta$-nearly minimally expanding pairs $(A,B)$ on $G$.

In analogy with the discussion for the Cauchy--Davenport theorem, we hope for an answer along the following line: 
If $G$ is compact, and $(A,B)$ is a $\delta$-nearly minimally expanding pair on $G$ with small $\delta$, then there is a continuous and surjective group homomorphism $\chi: G \to \TT$, compact interval $I, J \subseteq T$, and small $\varepsilon$, such that 
\begin{equation}\label{eq: main conj}
A \subseteq \chi^{-1}(I), B \subseteq \chi^{-1}(J),  \mu_G( \chi^{-1}(I)\setminus A ) < \varepsilon \mu_G(A),  \mu_G( \chi^{-1}(I)\setminus B ) < \varepsilon \mu_G(B).
\end{equation}
  The optimistic conjecture for noncompact groups is similar, but with $\TT$ replaced by $\RR$ and an extra condition that $\chi$ has compact kernel.

Under the extra assumption that $G$ is abelian, the optimistic conjectures for both classification problems were more or less confirmed before our work. In the same paper~\cite{Kneser} mentioned earlier, Kneser solved the classification problem for equality with the answer we hope for. 

For  the  near equality  problem, when $G=\TT^d$, the desired classification was obtained by Bilu~\cite{Bilu}, and later improved by Candela and De Roton~\cite{circle19} for a special case when $d= 1$. When $G$ is a general abelian group, a classification result was obtained by Tao~\cite{T18} for compact $G$, and by Griesmer~\cite{G19} when $G$ is noncompact. Griesmer also proved more general results for disconnected groups~\cite{G14, G19}. The results by Griesmer~\cite{G14, G19} and by Tao~\cite{T18} used nonstandard analysis methods, and did not provide how $\varepsilon$ quantitative depends on $\delta$ in \eqref{eq: main conj}. A sharp exponent classification result (i.e., $\varepsilon=O(\delta)$) for compact abelian groups was obtained very recently by Christ and Iliopoulou~\cite{ChristIliopoulou}. Results with sharp exponent bounds are likely the best that one can achieve without solving the $(3k-4)$-conjecture for $\ZZ/p\ZZ$.

For nonabelian $G$, not much was known earlier than this paper.  In closest proximity to the current work, Bj\"orklund considered in~\cite{abelcomponent} a variation of Kemperman's inequality and the equality classification problem without  assuming that $G$ is connected while assuming additionally that $G$ is compact, second countable, and has abelian identity component, and the sets $A$ and $B$ are ``spread out'' (i.e., far away from being subgroups). The only common case to this and our current setting happens when $G$ is abelian and connected.

Toward showing that appropriate versions of the optimistic conjectures also hold for the nonabelian classification problem, there is an important new challenge: While the desired conclusions for the abelian setting are mainly about the structure of $(A,B)$, the structure of $G$ is also highly involved for the nonabelian setting. If $G = \text{SO}_3(\RR)$, for example, one would not be able to find a minimally expanding pair according to the optimistic answers because there is no continuous surjective group homomorphism from $\text{SO}_3(\RR)$ to $\TT$. On the other hand, one can always find a continuous and surjective group homomorphism from a compact connected nontrivial abelian group to $\TT$ and use this to construct  minimally expanding pairs.

The above challenge connects our problem to the subject of small expansions in nonabelian groups, a fascinating topic that brings together ideas from different areas of mathematics. The phenomenon that expansion rate encodes structural information about the group can already be seen through the following famous theorem by Gromov~\cite{Gromov} in geometric group theory: If $G$ is a group generated by a finite set $X = X^{-1}$, and  $|X^n|$ grows polynomially as a function of $n$, then $G$ must be virtually nilpotent. A more recent result by Breuillard indicates that some of the analysis go through for locally compact groups~\cite{Breuillardballs}. Even more suggestive is the classifications of approximate groups  in~\cite{BGT} mentioned earlier (see the definition in Section~\ref{sec: 6.2}). 
In our proof, we will use the continuous version of the result in~\cite{BGT}; this was proven in  
the thesis of Carolino~\cite{thesis} and can also be deduced from the result in~\cite{BGT} using a result of Massicot--Wagner~\cite{MW}.  The ideas in the proof of these results can be traced back to the solution of Hilbert's Fifth problem by Montgomery--Zippin~\cite{MZ56}, Gleason~\cite{Gleason}, and Yamabe~\cite{Yamabe}, which we will also use later on. Finally, let us mention that these stories are also closely tied to the study of definable groups in model theory. This is the natural habitat of the aforementioned result by Massicot--Wagner~\cite{MW}, and also of Hrushovski's Lie model theorem~\cite{Hrushovski}, a main ingredient for the proof of the main theorem in~\cite{BGT}.

Before getting to the results, we briefly survey a number of works for nonabelian groups which are thematically relevant but use different techniques. When $A, B \subseteq G$ are finite and nonempty,  DeVos~\cite{Devos} classified  all situations where $|AB|<|A|+|B|$.  In~\cite{Bjorklundandfish}, Bj\"orklund and Fish studied an expansion problem with respect to upper Banach density in amenable nonabelian groups and obtained conclusions with similar flavor. Under the assumption that $A$ is a finite subset of a group $G$ such that the relation $xy \in A$ has finite VC-dimension (or NIP), Terry, Conant, and Pillay~\cite{CPT18, boundedVCdim} shows that $A$ must have a structure which is surprisingly similar to the optimistic conjecture mentioned earlier.

It would also be interesting to study a different minimal and nearly minimal measure expansion problem where we fix a connected unimodular group $G$ instead of letting $G$ range over all connected unimodular group $G$ as we are doing here. When $G$ is $\RR^n$, Kemperman inequality is a consequence of the Brunn--Minkowski inequality
$$\mu_G(AB)^{1/n} \geq \mu_G(A)^{1/n}+ \mu_G(B)^{1/n}. $$
This inequality  also holds for nilpotent $G$~\cite{Brunn1, Gromov2, Brunn3}. The equality holds in the Brunn--Minkowski inequality for $\RR^n$ if and only if $A$ and $B$ are homothetic convex subsets of $\RR^n$. This was a result by Brunn and Minkowski when $A$ and $B$ are further assumed to be convex, and a result by Lyusternik~\cite{Lyusternik}, Henstock and Macbeath~\cite{HenstockMacbeath} in the general case.
 A qualitative answer for the near equality Brunn--Minkowski problem for $\RR^n$ is obtained by Christ \cite{qualitativebrunnminkowski}, and a quantitative version is obtained by Figalli and Jerison \cite{QuantitativeBrunnMinkowski}. We do not pursue this direction further here.

\subsection{Statement of main results}
  Our first main result determines the conditions for equality to happen in the Kemperman inequality answering a question by Kemperman in~\cite{Kemperman}. Scenario (v) and (vi) in the theorem is a classification of the groups $G$ and minimally expanding pairs $(A,B)$ on $G$. 

\begin{theorem}\label{thm:mainequal}
Let $G$ be a connected unimodular group, and let $A,B$ be nonempty compact subsets of $G$.  If
\[
\mu_G(AB)=\min\{\mu_G(A)+\mu_G(B), \mu_G(G)\}. 
\]
 then we have the following:
\begin{enumerate}[\rm (i)]
    \item $\mu_G(A) +\mu_G(B)=0$ implies $\mu_G(AB) =0$; 
    \item $\mu_G(A) +\mu_G(B)\geq \mu_G(G)$ implies   $AB=G$;
    \item   $\mu_G(A)=0$ and $0<\mu_G(B)< \mu_G(G)$ imply there is compact  $H \leq G$ and compact $B_1 \subseteq B$ such that with $B_2= B\setminus B_1$, we have $HB_1=B_1$,  $\mu_G(AB_2)=0$, and $A \subseteq gH$ for some $g \in G$;
    \item  $\mu_G(B)=0$ and $0<\mu_G(A)< \mu_G(G)$ imply there is compact  $H \leq G$ and compact $A_1 \subseteq A$ such that, with $A_2=A \setminus A_1$, we have $A_1H=A_1$, $\mu_G(A_2B)=0$, and $B \subseteq Hg$ for some $g \in G$; 
    \item $0<\min\{\mu_G(A),\mu_G(B),\mu_G(G)-\mu_G(A)-\mu_G(B)\}$, and $G$ is compact together imply that there is a surjective continuous group homomorphism $\chi: G \to \TT$ and compact intervals $I$ and $J$ in $\TT$ with $I+J\neq \TT$ such that $A = \chi^{-1}(I)$ and $B = \chi^{-1}(J)$;
    \item   $0<\min\{\mu_G(A),\mu_G(B)\}$, and $G$ is not compact together imply that there is a surjective  continuous group homomorphism $\chi: G \to \RR$ with compact kernel and compact intervals $I$ and $J$ in $\RR$ such that $A = \chi^{-1}(I)$ and $B = \chi^{-1}(J)$.
\end{enumerate}
Moreover, $\mu_G(AB)=\min\{\mu_G(A)+\mu_G(B), \mu_G(G)\}$ holds if and only if we are in exactly one of the implied scenarios in (i-vi).
\end{theorem}

  Next we obtain a classification of nearly minimally expanding pairs. This answers questions by Griesmer~\cite{G19} and confirms a conjecture by Tao~\cite[Conjecture 5.1]{T18}, under the extra assumption of connectedness.

\begin{theorem}\label{thm:mainapproximate}
 Let $G$ be a connected compact group, and let $A,B$ be compact subsets of $G$ with $$0< \lambda=\min\{\mu_G(A),\mu_G(B),1-\mu_G(A)-\mu_G(B)\}.$$
  There is a constant $K=K(\lambda )$, not depending on $G$, such that for any $0\leq \varepsilon<1$, whenever we have $\delta\leq K\varepsilon$ and
  \[
  \mu_G(AB)<\mu_G(A)+\mu_G(B)+\delta\min\{\mu_G(A),\mu_G(B)\},
  \]
 there is a surjective continuous group homomorphism $\chi: G \to \TT$ together with two compact intervals $I,J\in \mathbb T$ with
 \[
 \mu_\TT(I)<(1+\varepsilon)\mu_G(A),\quad \mu_\TT(J)<(1+\varepsilon)\mu_G(B),
 \]
 and $A\subseteq \chi^{-1}(I)$, $B\subseteq\chi^{-1}(J)$. 
\end{theorem}

It is worth noting that the linear dependence between $\varepsilon$ and $\delta$ is the best possible up to a constant factor. The conclusions in Theorems~\ref{thm:mainequal} and \ref{thm:mainapproximate}, with suitable modifications, hold for arbitrary $A,B \subseteq G$ with inner measures; see Section~\ref{sec: 9}.

The proof Theorem~\ref{thm:mainapproximate} yields a number of auxiliary results. One of them is a short proof of the main result in~\cite{T18} and its sharp-exponent improvement in~\cite{ChristIliopoulou}; see Theorem~\ref{thm: abelian case}.  We also showed quantitatively that measure approximate groups of a compact Lie group cannot be a ``Kakeya set'' with respect to the ``torus directions''; see Theorem~\ref{thm: Torictransversal}.
 Most notably, we obtain the following measure expansion gap result for sets in connected compact simple Lie groups. 
\begin{theorem}[Expansion gaps in compact simple Lie groups]\label{thm: maingap}
There is a constant $\eta\geq 10^{-12}$ such that the following holds. Let $d>0$ be an integer. There is a constant $C>0$ depending only on $d$ such that if $G$ is a connected compact simple Lie group of dimension at most $d$, and $A$ is a compact set of $G$ with $0<\mu_G(A)<C$. Then 
\[
\mu_G(A^2)>(2+\eta)\mu_G(A).
\]
\end{theorem}

  We did not try to optimize the constant $\eta$ of Theorem~\ref{thm: maingap} in this paper. 
One may compare Theorem~\ref{thm: maingap} with expansion gaps for finite sets. The latter problem is well studied, which is initialed by  Helfgott~\cite{Helfgott08} where he proved an expansion gap in $\mathrm{SL}_2(\mathbb Z/p\mathbb Z)$. Results on the expansions for finite sets are one of the main ingredients in proving many of spectral gap results. For example, the result by Helfgott is largely used in the proof by Bourgain and Gamburd~\cite{BG08,BG12}. De Saxc\'e proved in \cite{de15} an expansion gap results in simple Lie groups, which is used in the later proof of  spectral gap results~\cite{Bde16,BAG17}. 
For more background in this direction, we refer the reader to \cite{Breuillard18,Tao-expansion}.

\subsection{Notation and convention} 
Throughout, let $k$ and $l$ range over the set $\ZZ$ of integers, and $m$ and $n$ range over the set $\NN=\{0,1,\ldots\}$ of natural numbers. A constant in this paper is always a positive real number. For real valued quantities $r$ and $s$, we will use the standard notation $r= O(s)$ and $s= \Omega(r)$ to denote the statement that $r< Ks$ for an absolute constant $K$ independent of the choice of the parameters underlying $r$ and $s$.  If we wish to indicate dependence of the constant on an additional
parameter, we will subscript the notation appropriately.

We let $G$ be a locally compact group, and $\mu_G$ a left Haar  measure on $G$. We normalize $\mu_G$, i.e., scaling by a constant to get $\mu_G(G)=1$, when $G$ is compact. For $\mu_G$-measurable $A\subseteq G$ and a constant $\varepsilon$, 
we set
\[  \Stab^{\varepsilon}_{G}(A) =\{ g \in G : \mu_G( A \tri gA) \leq \varepsilon \} \  \text{ and } \  \Stab^{<\varepsilon}_G(A) =\{ g \in G : \mu_G( A \tri gA) < \varepsilon\} .
\] 
Suppose $A$, $B$, and $AB$ are $\mu_G$-measurable sets in $G$. The {\bf discrepancy} of $A,B$ in $G$ is defined by
\[
\dis_G(A,B)=\mu_G(AB)-\mu_G(A)-\mu_G(B). 
\]
When $G$ is connected, we always have $\dis_G(A,B)\geq0$.

Let $H$ range over closed subgroups of $G$. We let  $\mu_H$ denote  a left Haar measure of $H$, and normalize $\mu_H$ when $H$ is compact. Let $G/H$ and $H\backslash G$ be the left coset space and the right coset space with quotient maps
\[ \pi: G \to G/H\  \text{ and }\ \widetilde{\pi}: G \to H\backslash G
\] Given a coset decomposition of $G$, say $G/H$, a left-fiber of a set $A\subseteq G$ refers to $A\cap xH$ for some $xH\in G/H$. We also use $\mu_H$ to denote the fiber lengths in the paper, that is, we sometimes write $\mu_H(A\cap xH)$ to denote $\mu_H(x^{-1}A\cap H)$. By saying that $G$ is Lie group, we mean $G$ is a real Lie group with finite dimension, and we denote $\dim(G)$ the real dimension of $G$.

\section{Outline of the argument}\label{sec: section 2}

In this section, we informally explain some of the main new ideas of the proofs. We decided to write a slightly longer outline as some of the later computations are rather technical.

\subsection{Overview of the strategy}  \label{Sec: Overview} After handling a number of easy cases, the proofs of  Theorem~\ref{thm:mainequal} and Theorem~\ref{thm:mainapproximate} require constructing appropriate continuous and surjective group homomorphisms into either $\RR$ or $\TT$ under the given data of a minimally or nearly minimally expanding pair $(A,B)$ on a connected unimodular group $G$. The key difficulty of the problem is that many methods in the abelian settings (e.g., fourier analysis) has no generalization to the nonabelian settings which is obviously useful for our purpose. We will get around this, essentially using an induction on dimension strategy to reduce to the abelian settings.

Our argument can be broadly divided into three steps and a small preparation. For the preparation, we use a submodularity argument also used in~\cite{T18} and~\cite{ChristIliopoulou} to arrange that $\mu_G(A) =\mu_G(B)$ are rather small  when $G$ is compact.

In the first step, using ideas from the solution of Hilbert's Fifth problem, arithmetic combinatorics, and model theory, we choose a compact and connected normal subgroup $H$ of $G$ such that $G/H$ is a Lie group of bounded dimension and $H$ is ``in roughly the same direction'' as $A$ and $B$ (i.e., the measures of the images of $A$ and $B$ in $G/H$ are not too large compared to  $A$ and $B$). Then, by studying the geometric shape of $A$ and $B$ with respect to $H$, we prove a coarse version of our theorem:  the minimally or nearly minimally expanding pair $(A,B)$ arises from a minimally or nearly minimally expanding pair $(A',B')$ on $G/H$. This also essentially reduces the problem to the case of Lie group, where we are aided by a powerful structure theory. 

The second step reduces the problems of constructing the desired group homomorphism onto $\TT$ or $\RR$ to showing that 
$ d_A(g_1, g_2) := \mu_G(A)  - \mu_G(g_1A \cap g_2 A)  $
satisfies 
\begin{equation} \label{eq: almostlinearlyadditivelyofdA}
    d_A(g_1, g_3) = |d_A(g_1, g_2)\pm d_A(g_2,g_3)|+\varepsilon \ \text{ and } \ d_A(\id, g^2) = 2d_A(\id, g) +\varepsilon,
\end{equation}
for some error $\varepsilon$, and assuming  $d_A(g_i, g_j)<\mu_G(A)/4$ for $g_i, g_j \in \{\id, g, g_1, g_2, g_3 \}.$ 
Note that $d_A$ is a pseudometric (i.e., it satisfies all the properties of a metric except $d_A(g_1,g_2)=0$ implying $g_1=g_2$). We remark that similar pseudometrics are also used in the solution of Hilbert's Fifth problem. Moreover,
$$\Stab^{<\varepsilon}_G(A)=\{g \in G: d_A(\id, g)< \varepsilon/2\}.$$
Hence, looking at $d_A$ can be seen as an alternative way of  considering approximate stabilizers, a recurring theme in the study of definable groups in model theory and approximate groups. 

In the third step, we assume that $G$ is a Lie group.
Using probabilistic and Lie-theoretic arguments, we choose a connected closed subgroup $H$ of $G$ with smaller dimension such that all the cosets of $H$ intersect $A$ and $B$ ``transversally in measure'' (i.e. the intersection of $A$ or $B$ with each coset of $H$ has small measure). When $G$ is compact, $H$ can be chosen to be a one-dimensional torus subgroups, so this can be seen as showing that $A$ and $B$ cannot be ``Kakeya sets'' in ``torus subgroups directions''. We can assume $H$ already satisfy the conclusion of Theorem~\ref{thm:mainequal} and Theorem~\ref{thm:mainapproximate} as an induction hypothesis. We then obtain a description of the geometric shapes of $A$ and $B$ relative to the cosets of $H$. Using that, we show that $d_A$ satisfies \eqref{eq: almostlinearlyadditivelyofdA} which completes the proof.

 \begin{figure}[h]
     \centering
     \includegraphics[width=5.9in]{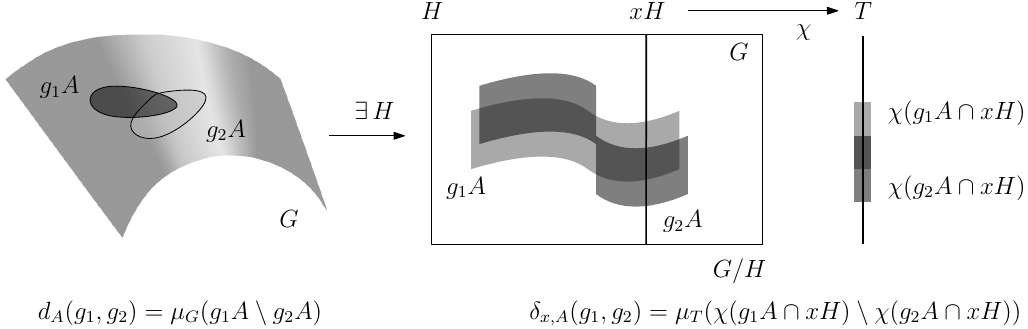}
     \caption{The third step}
     \label{fig:intuition}
 \end{figure}

The remaining part of Section~\ref{Sec: Overview} will further elaborate ideas from  the steps above.  Sections~\ref{Sec: Overviewstep3}, \ref{Sec: Overviewstep2}, and \ref{Sec: Overviewstep1} and  will explain further a number of technical innovations. 

The bulk of the second step is to produce a continuous group homomorphism onto $\RR$ or $\TT$ from the data of a pseudometric as in~\eqref{eq: almostlinearlyadditivelyofdA}.
Let us assume we already had a continuous surjective group homomorphism $\chi: G \to \TT$,  and $d_\TT$ is the Euclidean metric on $\TT$. Set $d(g_1, g_2) = d_\TT(\chi(g_1), \chi(g_2))$. It is easy to see that $d$ is a pseudometric  on $G$ with the ``linear''  property in \eqref{eq: almostlinearlyadditivelyofdA} with $\varepsilon=0$.
In this case $\chi$ can be recovered from $d$ by noticing that
  $$ \ker \chi = \{ g \in G : d(g, \id ) =0  \}. $$
 It turns out that it is also possible for $\varepsilon>0$ as in \eqref{eq: almostlinearlyadditivelyofdA}, but it requires developing some nontrivial machinery, especially when $\varepsilon$ grows linearly on the radius of the pseudometric; see Section~\ref{Sec: Overviewstep2} for details.

Let us next consider the geometrical description of $(A,B)$ in the third step and
how it can be used to show \eqref{eq: almostlinearlyadditivelyofdA}.
We suppose a ``transversally in measure'' $H$ is already obtained.  
We visualize $G$ in two ways: a rectangle  with the horizontal side representing $G/H$ and each of the vertical section representing a left coset of $H$, and a similar dual  picture for $H\backslash G$; see the middle figure of Figure~\ref{fig:intuition}. The main idea is to show that $g_1A$ and $g_2A$ geometrically (under this coset decomposition) look like in the picture with $g_1, g_2 \in G$ in a suitable neighborhood of $\id$, and behave rigidly under translations.  
(For instance, we want the ``fibers'' of $A$ and $B$, i.e. intersections of $A$ or $B$ with coset of $H$, to be preimages of intervals of $T$, and all the nonempty ``fibers'' of $g_1A$ to have similar ``lengths''.)  
By induction hypothesis, we can assume that $H$ satisfies the conclusions in the main theorems.
In particular, the Euclidean metric on $\TT$ induces a pseudometric $\delta_{x,A}$ on a generic fiber $xH$. The key point is that, the nice geometric shape of $A$ can pass some of the properties of $\delta_{x,A}$ to the global pseudometric $d_A$.

The geometric idea in the first step is somewhat similar to what described above for the third step, so to end Section~\ref{Sec: Overview}, let us explain how ideas from arithmetic combinatorics and model theory comes into play.  Using an argument from~\cite{T08}, we can produce from $(A,B)$  a commensurable open approximate groups $S$. Then, by a result from Carolino's thesis~\cite{thesis}, which can be regarded as the continuous version of the Lie model theorem from~\cite{Hrushovski} and~\cite{BGT}, we produce a continuous surjective group homomorphism to a Lie group with bounded dimension. The bound on dimension will play a role in determining what it means for $(A,B)$ to have small measure in the third step, and contribute in the error bound in Theorem~\ref{thm:mainapproximate}.

\subsection{The first step: Nearly minimal expansions and quotients} \label{Sec: Overviewstep3}
Suppose $(A, B)$ is a nearly minimally expanding pair on $G$, $H$ is a {\it connected}, {\it compact} and {\it normal} subgroup of $G$,  $\pi: G \to G/H$ is the quotient map, and 
   $$  \mu_{G/H}(\pi(A))+ \mu_{G/H}(\pi(B)<1.$$
The goal of this step is to show the following {\it quotient domination} result: There is a nearly minimally expanding pair $(A', B')$ on $G/H$ such that $\mu_G( A \tri \pi^{-1}(A'))$ and $\mu_G( B \tri \pi^{-1}(B'))$ are both small.

 \begin{figure}[h] \label{fig: quotient}
    \centering
    \includegraphics{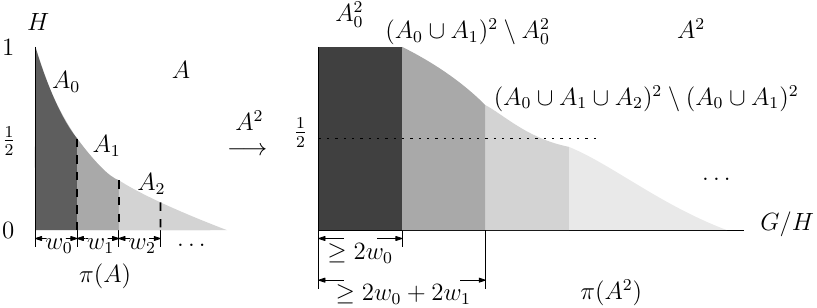}
    \caption{Lower bound for $\mu_G(A^2)$.}
    \label{fig:quotient}
\end{figure}

To illustrate the idea, we focus on the special case with $A=B$. Employing the geometric language in Section~\ref{Sec: Overview}, we call $\mu_{G/H}(\pi(A))$ the \emph{width} of $A$, for each $g$ in $G$, we call $A \cap gH$  a \emph{fiber} of $A$, and refer to $\mu_H(g^{-1}A\cap H)$  as its \emph{length}. We consider a further special case assuming that  $A$ can be partitioned into $N+1$ parts  $A=\bigcup_{i=0}^N A_i$ such that the images under $\pi$ of the $A_i$'s are compact and pairwise disjoint, $A_i$ has width $w_i$,  the fibers in $A_0$ all have length $\geq 1/2$, the fibers in $A_{i}$ all have the same length $l_i \leq  1/2$ for each $i \geq 1$, and $l_i \geq l_{i+1}$ for all $i <N$. This further special case is, in fact, quite representative, as we can reduce the general problem to it using approximation techniques.

The proof of this step can be seen as the following ``spillover'' argument. Applying the Kemperman inequalities for $H$ and $G/H$, we learn that all the fibers in $A^2_0$ has length $1$, and the width of $A_0^2$ is at least $2w_0$. By Fubini's theorem, $\mu_G(A_0^2)$ is at least $2w_0$. Next, consider $(A_0 \cup A_1)^2$. A similar argument gives us that all the fibers in $(A_0 \cup A_1)^2$ has length at least $2l_1$, and the width of  $(A_0 \cup A_1)^2$ in $G/H$ is at least $2w_0+2w_1$. Note that $2l_1(2w_0+2w_1)$ is a weak bound for $\mu_G((A_0 \cup A_1)^2)$ since the fibers in $A_0^2$ are ``exceptionally long''. Taking all of these into account, a stronger lower bound for $\mu_G(A_0(A_0 \cup A_1))$ is 
$$  2w_0 + 4l_1w_1.   $$
Note that $\mu_{G}(A) = l_0w_0+ \ldots +l_Nw_N$. Hence, $\mu_G(A^2)$ is nearly $2\mu_G(A)$ implies that we must nearly have $w_0=1$, $w_1 = \ldots = w_N =0$, and  $l_0=1$.  From this, one can deduce the conclusion that we want for this step.

\subsection{The second step: Pseudometrics and homomorphisms} \label{Sec: Overviewstep2}
We assume in this section that $G$ is a {\it connected} and {\it compact} Lie group, $d$ is a left-invariant continuous pseudometric on $G$ with the following properties:
\begin{enumerate}[\rm (i)]
    \item (Locally almost linearity) There is $\lambda \in \RR^{>0}$ such that with $$N(\lambda):= \{ g\in G : d(\id, g) < \lambda \},$$ there is $\varepsilon < 10^{-6}\mu_G( N(\lambda))$ such that for all $g_1, g_2, g_3 \in N(\lambda)$
      $$  d(g_1, g_3) \in  |d(g_1, g_2)\pm d(g_2,g_3)| + (-\varepsilon,\varepsilon). $$
    \item (Locally almost monotonicity) With the same $\lambda$ in (2), for all $g \in N_\lambda$,
    $$ |d(\id, g^2) - 2d(\id, g)| \in (-\varepsilon,\varepsilon). $$
\end{enumerate}
We now sketch how to construct a continuous and surjective group homomorphism to $\TT$ from these data. 
A crucial argument we will not be able to get into details here is to show that the local monotonicity condition can be deduced from a weaker property of path monotonicity obtained from the third step (Section~\ref{Sec: Overviewstep1}).

When we are in the special case with $\varepsilon=0$ in property (i), there is a relatively easy argument which also works for noncompact Lie groups. Set $$\ker d =\{ g\in G: d( \id, g)=0\}.$$ 
Using the left invariance, continuity, and triangle inequality,  one can  show that $\ker d$ is a closed subgroup of $G$. Moreover, in this case,  $G/\ker d$ must be isomorphic to $\TT$, and  the pseudometric $d$ locally must agree with a constant multiple of the pullback of the Euclidean metric. These are, perhaps, not too surprising as a Lie group equipped with a locally linear pseudometric is, intuitively, a very rigid object which locally looks like a straight line. In fact, property (ii) is not needed, as it is a consequence of property (i) in this case.

The general case is much harder, as we no longer have the same type of rigidity. In particular,  $\ker d$ might not be normal, and $G/\ker d$ might not be $\TT$ even if $\ker d$ is normal. The reader familiar with the proof of Hilbert's Fifth problem would guess that we might try to slightly modify $d$ to get a locally linear pseudometric $d'$ and use the earlier strategy.  This is still true at the conceptual level, but our actual argument is much more explicit, allowing error control.  We will construct a multi-valued function $\Omega$ from $G$ to $\RR$ with additive properties. From this we get a multi-valued almost group homomorphism $\Psi$ from $G$ to $\TT$ by quotienting $\omega\ZZ$ for a value of $\omega \in \RR^{>0}$ to be described later. The desired group homomorphism $\chi$ can be obtained from $\Psi$ using descriptive set theory and Riemannian geometry.

 For an element $g \in G$, we represent it in ``the shortest way'' as a product of elements in $N(\lambda)$. The following notion capture this idea.
\begin{definition}
A sequence $(g_1,\dots,g_m)$  with  $g_i\in N(\lambda)$ is \emph{irreducible} if $g_{i+1}\cdots g_{i+j}\notin N(\lambda)$  for $2\leq j\leq 4$. 
\end{definition} 
Now, we set $\Omega_m(g)$ be the subset of $\RR$ consists of sums of the form
\begin{equation}
\sum_{i=1}^m \sgn(g_i)d(g_i,\id),
\end{equation}
where $(g_1,\dots,g_m)$ is an irreducible sequence with $g=g_1\cdots g_m$, and  $\sgn(g_i)$ is the \emph{relative sign} defined in Section~\ref{section:pseudometric}. Heuristically, $\sgn(g_i)$  specifies the ``direction'' of the translation by $g_i$ relative to a fixed element. Finally, we set $\Omega(g)  =\bigcup_{m=1}^M \Omega_m(g)$, where we will describe how to choose $M$ in the next paragraph. We do not quite have $\Omega(gg') = \Omega(g)+ \Omega(g')$, but a result along this line can be proven.

Call the maximum distance between two elements of $\Omega(g)$ the error of $\Omega(g)$. We want this error
to be small so as to be able to extract $\chi(g)$ from it eventually.
In the construction, each $d(g_i,\id)$ will increase the error of the image of $\Omega(g)$ by at least $\varepsilon$. Since the error propagate very fast, to get a sharp exponent error bound, we cannot choose very large $M$. To show that a moderate value of $M$ suffices, we construct a ``monitoring system'',  an irreducible sequence of bounded length, and every element in $G$ is ``close to'' one of the elements in the sequence. 
The knowledge required for the construction amounts to an understanding of the size and expansion rate of small stabilizers of $A$, which can also be seen as a refined Sanders--Croot--Sisask-type result for nearly minimally expanding sets.

In the aforementioned case where $d$ is linear (i.e. $\varepsilon =0$), the metric induced by $d$ in $G/\ker d$ is a multiple of the standard Euclidean metric in $\RR/\ZZ$ by a constant $\omega$. 
If we apply the machinery of irreducible sequence to this special case that $\varepsilon=0$, $\Omega(g)$ will be a single-valued function, and we will get for each $g\in G$ that $\Omega(g) = \psi(g) + \omega\ZZ$ for $\psi(g) \in [0, \omega)$. In particular, $\Omega(\id)= \omega\ZZ$, hence 
$
\omega=\inf\big|\Omega(\id) \setminus\{0\}\big|.
$

In general case when $\varepsilon>0$, we can define
\[
\omega=\inf\big|\Omega(\id) \setminus(-M\varepsilon,M\varepsilon)\big|.
\]
By further applying properties of irreducible sequences and the monitor lemma, we show  that for each $g$, $\Omega(g) \subseteq \psi(g) \omega\ZZ + (-M\varepsilon,M\varepsilon)$. For each $g\in G$, we then set $\Psi(g) \subseteq \TT$ to be $\Omega(g)/\omega\ZZ$.

Note that $\Psi$ is ``continuous'' from the way we construct it. We obtain $\chi$ from $\Psi$ by first extracting from $\Psi$ a universal measurable single valued almost-homomorphism $\psi$, then modifying $\psi$ to get a universal measurable group homomorphism $\chi$, and show that $\chi$ is automatically continuous. Many elements of our proof also work for noncompact group. However, the modifying step to get a group homomorphism from an almost group homomorphism does not go through due to the existence of quasi-morphism which are not close to group homomorphism in the noncompact case.

\subsection{The third step: Nearly minimal expansions and subgroups} \label{Sec: Overviewstep1}
We focus on the case where $G$ is a {\em connected} and {\em compact} Lie group, and $(A,B)$ is a nearly minimally expanding pair on $G$ with sufficiently small measure.   Recall that  we want to find a closed and connected subgroup $H$ of $G$ such that the ``length'' $\mu_H(g^{-1} A \cap H)$ of each ``left fiber'' $A \cap gH$ of $A$ is small, and a similar condition hold for ``right fibers'' of $B$. In this case, $H$ can be chosen to be a one-dimensional torus subgroup.

For simplicity, we assume $A=B$. Suppose $A$ is a ``Kakeya set'', i.e, it has a long fiber $A \cap gH$ for every choice of ``direction'' $H$ of $G$. In Section~\ref{subsec: toric exp} we will show that under the condition when $\mu_G(A^2)<M\mu_G(A)$ for some constant $M$, for every $H$, $\mu_G(AH)/\mu_G(A)$ is not too large. This motivates us to define the following notion:
\begin{definition}
$A$ is a \emph{toric $K$-expander} if there is a one dimensional torus subgroup $H$ of $G$ such that $\mu_G(AH)>K\mu_G(A)$.
\end{definition}

We will show that every nonempty compact subset of $G$ with sufficiently small measure must be a toric $K$-expander. 
Let us present here a pseudo-argument, which nevertheless illustrate the idea. Assume $A$ is not a toric $K$-expander. Obtain finitely many torus subgroups $H_1, \ldots, H_n$ of $G$ such that $$G=H_1\cdots H_n.$$
 Let us pretend that using the assumption $\mu_G(AH_1)\leq  K \mu_G(A)$ we can cover $AH_1$ with $(K+1)$ right translations of $A$. It can be then shown that $AH_1$ is not a toric $K(K+1)$-expander. Next, we further pretend that $AH_1H_2$ can be covered with $K(K+1)+1$ right translations of $AT_1$ which can then be covered by $K(K+1)^2+ (K+1)$ right translations of $A$. Continuing the procedure, we get $C(K)$ such that $AT_1\cdots T_n=G$ can be covered by $C(K)$ right translations of $A$. Thus, $\mu_G(A)> 1/C(K)$, contradicting the assumption that $\mu_G(A)$ is very small.

 The pseudo-argument in the preceding paragraph does not work in most of the cases. In particular, one cannot deduce from $\mu_G(AH_1)< K \mu_G(A)$  that $AH_1$ can be covered by finitely many right translations of $A$. However, it does contain some truth, and we will be able to use a probabilistic argument to approximate this pseudo-argument.

Now choosing a one-dimensional torus subgroup $H$ of $G$ such that for all $x \in G$ and $y \in G$, the fibers $xH\cap A$ and $B\cap Hy $ are both short. We will show that the set $A$ and $B$ have the shape as described in Section 2.1. Without loss of generality, we can arrange that the width $\mu_G(AH)$ of $A$ in $G/H$ is at least the width $\mu_G(HB)$ of $B$ in $H \backslash G$. Choose uniformly at random $xH \in AH $, and applying the Kemperman inequality for $H$, we have 
\begin{align*} 
 \mu_G(AB)&\geq\mathbb E_{xH\in AH} \mu_G( (A \cap xH) B)\\
 &\geq    \mathbb E_{xH\in AH}\mu_{H} ( A \cap xH ) \mu_{H \backslash G}( HB) + \mu_G(B) \\
 &= \mu_G(A)\frac{\mu_{H\backslash G}(HB)}{\mu_{G/H}(AH)}+\mu_G(B)  \\
& \geq \mu_G(A)+\mu_G(B);
\end{align*}
As $(A,B)$ is nearly minimally expanding, we have $\mu_G(AB)$ is nearly the same as $\mu_G(A)+\mu_G(B)$. The fourth line then gives us that $\mu_G(HB)$ is nearly the same as $\mu_G(AH)$. The second line now gives us that for most of $xH \in AH$, the fiber $(xH \cap A)$ is nearly an interval up to an endomorphism of $H$ by using the induction hypothesis on $H$. From the first line,  $\mu_H(A \cap xH)$ is almost constant as $xH$ ranges through $AH$. 

Using the above geometric properties of $A$ and $B$ together with some probabilistic arguments, in Section~\ref{sec: geometry II} we show that $A$ also behaves rigidly under translations: $A$ look like a horizontal ``strip'' under the coset decomposition, and the intersection of $A$ and $gA$ also behaves like a ``strip'', where $g$ is from a certain neighborhood of identity. Using this, we are able to pass the almost linearity of the pseudometric in $H$ to the pseudometric $d_A$ defined in Section~\ref{Sec: Overview}, as well as the \emph{path monotonicity} which is a necessary ingredient to show assumption (2) in Step 2.

 \subsection{Structure of the paper}

The paper is organized as follows. Section~\ref{sec: Preliminaries} includes some facts about Haar measures and unimodular groups, which will be used in the subsequent part of the paper. A general version of the Kemperman inequality and its inverse theorem in tori are also included in Section~\ref{sec: Preliminaries}. Section~\ref{sec: easy directions of thm 1.1} deals with the more immediate parts of Theorem~\ref{thm:mainequal} and hence sets up the stage for the main part of the argument. Section~\ref{sec: reduction to small} allows us to arrange that in a minimally or a nearly minimally expanding pair $(A,B)$, the sets $A$ and $B$ have small measure (Lemma~\ref{lem:sml}). Sections~\ref{sec: geometry I}, \ref{section:pseudometric}, and \ref{sec: geometry II} contain main new technical ingredients of the proof, which will be put together in Section~\ref{sec: 9} to complete the proofs of Theorem~\ref{thm:mainequal}, Theorem~\ref{thm:mainapproximate}, and Theorem~\ref{thm: maingap}. Steps 1, 2, and 3 discussed in Sections~\ref{Sec: Overviewstep1}, \ref{Sec: Overviewstep2}, and \ref{Sec: Overviewstep3}, corresponds to Sections~\ref{sec: geometry I}, \ref{section:pseudometric}, and \ref{sec: geometry II} respectively. 

In Section~\ref{sec: 6.1}, we proved the quotient domination theorem (Theorem~\ref{thm: criticality transfer}), which allow us to transfer the problem into certain quotient groups. In Section~\ref{sec: 6.2} we obtained a coarse version of the main theorems (Proposition~\ref{prop: coarsetheorem}). These two sections reduce the problem to a bounded dimension Lie group. Section~\ref{sec: 6.3} contains structural results assuming we have an appropriate homomorphism (Proposition~\ref{prop: struc on AB}). We  also give 
a new proof of the inverse Kneser's inequality~\cite{T18} with a sharp exponent bound (Theorem~\ref{thm: abelian case}).
In the next two sections, we will focus on constructing a suitable homomorphism.

In Section~\ref{sec: 7.1}, we showed that a locally linear pseudometric on $G$ would induce a continuous surjective homomorphism to either $\RR$ or $\TT$, with compact kernel (Proposition~\ref{prop: strong linear}). Sections~\ref{sec: 7.2} and~\ref{sec: 7.3} study the locally almost linear pseudometric in compact Lie groups. In particular, we proved that path monotonicity implies monotonicity (Proposition~\ref{prop: localmonotoneimplyglobalmonotone}), and for almost monotone almost linear pseudometric, one can also find a homomorphism mapping to $\TT$ (Theorem~\ref{thm: homfrommeasurecompact}). 

In Section~\ref{subsec: toric exp}, we show that given a small measure expansion set with small measure, one can find a torus subgroup that is transversal in measure (Theorem~\ref{thm: Torictransversal}). We construct the pseudometric from geometric properties of nearly minimal expansion sets in Sections~\ref{subsec: linear fiber} and~\ref{subsec: almost linear fiber}. Section~\ref{subsec: linear fiber} provides a locally linear pseudometric from minimally expansion sets (Proposition~\ref{prop: equal pseudometric linear}). In Section~\ref{subsec: almost linear fiber}, we construct a path monotone locally  almost linear pseudometric (Proposition~\ref{prop: almost linear metric from local}). 

The dependency diagram of the paper is as below.
\begin{figure}[h]
    \centering
    \includegraphics{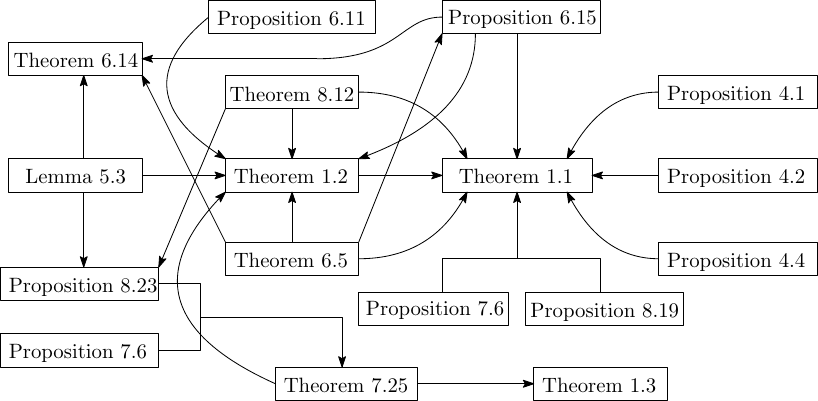}
\end{figure}

\section{Preliminaries}\label{sec: Preliminaries}
 Throughout this section, we assume that $G$ is a \emph{connected} locally compact group (in particular, Hausdorff) equipped with a left Haar measure $\mu_G$, and $A,B \subseteq G$ are nonempty.

\subsection{Locally compact groups and Haar measures}\label{sec: 3.1}

  Below are some basic facts about $\mu_G$ that we will use; see~\cite[Chapter~1]{Harmonicanalysis} for details:
\begin{fact} \label{fact: Haarmeasurenew}
Suppose $\mu_G$ is either a left or a right Haar measure on $G$. Then:
\begin{enumerate}[\rm (i)]
    \item If $A$ is compact, then $A$ is $\mu_G$-measurable and $\mu_G(A)< \infty$.
    \item If $A$ is open, then $A$ is $\mu_G$-measurable and $\mu_G(A)>0$.
    \item \emph{(Outer regularity)} If $A$ is $\mu_G$-measurable, then there is a decreasing sequence $(U_n)$ of open subsets of $G$ with $A \subseteq U_n$ for all $n$, and $\mu_G(A) = \lim_{n \to \infty} \mu_G(U_n). $
    \item \emph{(Inner regularity)} If $A$ is $\mu_G$-measurable, then there is an increasing sequence $(K_n)$ of compact subsets of $A$ such that $\mu_G(A) = \lim_{n \to \infty} \mu_G(K_n). $
     \item \emph{(Measurability characterization)} If there is an increasing sequence $(K_n)$ of compact subsets of $A$, and a decreasing sequence $(U_n)$ of open subsets of $G$ with $A \subseteq U_n$ for all $n$ such that $\lim_{n \to \infty} \mu_G(K_n) = \lim_{n \to \infty} \mu_G(U_n) $, then $A$ is measurable.
     \item \emph{(Uniqueness)} If $\mu'_G$ is another measure on $G$ satisfying the properties (1-5), then there is $C \in \RR^{>0}$ such that $\mu'_G = C\mu_G$.
    \item  \emph{(Continuity of measure under symmetric difference)} Suppose $A \subseteq G$ is measurable, then the function $G \to \RR, g \mapsto \mu_G(A \tri gA)$ is continuous.  
\end{enumerate}
\end{fact}
  We remark that the assumption that $G$ is connected implies that every measurable set is $\sigma$-finite (i.e., countable union of sets with finite $\mu_G$-measure). Without the connected assumption, we only have inner regularity for $\sigma$-finite sets. From Fact~\ref{fact: Haarmeasurenew}(vii), we get the following easy corollary:
\begin{corollary} \label{cor: Stabilizerisclosed}
Suppose $A$ is $\mu_G$-measurable and $\varepsilon$ is a constant. Then $\Stab^{\varepsilon}_G(A)$ is closed in $G$, while $\Stab^{<\varepsilon}_G(A)$ is open in $G$. In particular, $\Stab^{0}_G(A)$ is a closed subgroup of $G$.
\end{corollary}
  We say that $G$ is {\bf unimodular} if $\mu_G$ (and hence every left Haar measure on $G$)  is also a right Haar measure. The following is well known and can be easily verified: 

\begin{fact} \label{fact: measureofinverse}
If $G$ is unimodular, $A$ is $\mu_G$-measurable, then $A^{-1}$ is also $\mu_G$-measurable and $\mu_G(A) = \mu_G(A^{-1})$.
\end{fact}

 We use the following isomorphism theorem of topological groups. 
\begin{fact}\label{fact: first iso thm}
Suppose $G$ is a locally compact group, $H$ is a closed normal subgroup of $G$. Then we have the following.
\begin{enumerate}[\rm (i)]
    \item \emph{(First isomorphism theorem)} Suppose $\phi: G \to Q$ is a continuous surjective group homomorphism with $\ker \phi = H$.  Then the exact sequence of groups
        $$  1 \to H \to G \to Q \to 1 $$ 
    is an exact sequence of topological groups if and only if  $\phi$ is open; the former condition is equivalent to saying that  $Q$ is canonically isomorphic to $G/H$ as topological groups. 
    \item \emph{(Third isomorphism theorem)} Suppose $S \leq  G$ is closed, and $H \leq S$. Then $S/H$ is a closed subgroup of $G/H$. If $S\vartriangleleft G$ is normal, then $S/H$ is a normal subgroup of $G/H$, and we have the exact sequence of topological groups
    $$  1 \to S/H \to G/H \to G/S \to 1; $$
    this is the same as saying that $(G/H)/(S/H)$ is canonically isomorphic to $G/S$ as topological groups. 
\end{enumerate}
\end{fact}

Suppose $H$ is a closed subgroup of $G$. The following fact allows us to link Haar measures on $G$ with the Haar measures on $H$ for unimodular $G$ and $H$:

\begin{fact}[Quotient integral formula]\label{fact: QuotientIF} 
Suppose $H$ is a closed subgroup of $G$ with a left Haar measure $\mu_H$. If $f$ is a continuous function on $G$ with compact support, then
$$ xH \mapsto \int_H f(xh) \d\mu_H(x). $$
defines a function $f^H: G/H \to \RR$ which is continuous and has compact support. 
If both $G$ and $H$ are unimodular, then there  is unique invariant Radon measures $\mu_{G/H}$ on $G \slash H$ such that for all continuous function $f: G \to \RR$ with compact support, the following integral formula holds
    $$ \int_G f(x) \d\mu_G(x) = \int_{G/H} \int_H f(xh) \d\mu_H(h) \d\mu_{G /H}(xH). $$
A similar statement applies replacing the left homogeneous space $G/H$ with the right homogeneous space $H \backslash G$.
\end{fact}

  We can extend Fact~\ref{fact: QuotientIF} to measurable functions on $G$, but the function $f^H$ in the statement can be only be defined and is $\mu_{G/H}$-measurable $\mu_G$-almost everywhere. So, in particular, this problem applies to indicator function $\e_A$ of a measurable set $A$. This causes problem in our later proof and prompts us to sometimes restrict our attention to a better behaved subcollection of measurable subsets of $G$. We say that a subset of $G$ is {\bf $\sigma$-compact} if it is a countable union of compact subsets of $G$. 
 \begin{lemma} We have the following:
 \label{lem: mesurability}
\begin{enumerate}[\rm (i)]
    \item $\sigma$-compact sets are measurable.
    \item the collection of $\sigma$-compact sets is closed under taking countable union, taking finite intersection, and taking product set.
    \item For all $\mu_G$-measurable $A$, we can find a $\sigma$-compact subset $A'$ of $A$ such that $\mu_G(A') =\mu_G(A)$.
    \item Suppose $G$ is unimodular, $H$ is a closed subgroup of $G$ with a left Haar measure $\mu_H$, $A \subseteq G$ is $\sigma$-compact, and $\e_A$ is the indicator function of $A$. Then
 $aH \mapsto \mu_H(A \cap aH)$
defines a measurable function $\e^{H}_A: G/H \to \RR$. If $H$ is unimodular and, $\mu_{G/H}$ is the Radon measure given in Fact~\ref{fact: QuotientIF}, then  $$ \mu_G(A)=  \int_{G/H} \int_H \mu_H(A \cap aH) \d\mu_H(h) \d\mu_{G /H}(xH). $$
A similar statement applies replacing the left homogeneous space $G/H$ with the right homogeneous space $H \backslash G$.
\end{enumerate}

\end{lemma}
\begin{proof}
The verification of (i-iii) is straightforward. We now prove (iv). First consider the case where $A$ is compact. By Baire's Theorem, $\e_A$ is the pointwise limit of a monotone nondecreasing sequence of continuous function of compact support. If $f: G \to \RR$ is a continuous function of compact support, then the function $$f^H: G/H \to R, aH \mapsto \int_{H} f(ax) dx $$ is continuous with compact support, and hence measurable; see, for example, \cite[Lemma~1.5.1]{Harmonicanalysis}.  Noting that $\mu_H(A \cap aH) = \int_H \e_A(ax) dx$, and applying monotone convergence theorem, we get that $\e^H_A$
is the pointwise limit of a monotone nondecreasing sequence of continuous function of compact support. Using monotone convergence theorem again, we get $\e^H_A$ is integrable, and hence measurable. Also, by monotone convergence theorem, we get the quotient integral formula in the statement. 

Finally, the general case where $A$ is only $\sigma$-compact can be handled similarly, noting that $\e_A$ is then the pointwise limit of a monotone nondecreasing sequence of indicator functions of compact sets.
\end{proof}

  Suppose $H$ is a closed subgroup of $G$. Then $H$ is locally compact, but not necessarily unimodular. We use the following fact in order to apply induction arguments in the later proofs. 

\begin{fact}\label{fact: unimodular}
Let $G$ be a unimodular group. If $H$ is a closed normal subgroup of $G$, then $H$ is unimodular. Moreover, if $H$ is compact, then $G/H$ is unimodular.  
\end{fact}

  \subsection{More on Kemperman's inequality and the inverse problems} We will need a version of Kemperman's inequality for arbitrary sets.   Recall that the \emph{inner Haar measure} $\mut_{G}$ associated to $\mu_G$ is given by
  $$ \mut_G(A) = \sup \{ \mu_G(K) : K \subseteq A  \text{ is compact}.  \}$$
  The following is well known and can be easily verified:
 \begin{fact} \label{fact: Innermeasure}
 Suppose $\mut_G$ is the inner Haar measure associated to $\mu_G$. Then we have the following:
 \begin{enumerate}[\rm (i)]
     \item \emph{(Agreement with $\mu_G$)} If $A$ is measurable, then $\mut_G(A)=\mu(A)$. 
     \item \emph{(Inner regularity)} There is $\sigma$-compact $A' \subseteq A$ such that $$\mut_G(A)=\mut_G(A') = \mu_G(A).$$
     \item \emph{(Superadditivity)} If $A$ and $B$ are disjoint, then 
     $$\mut_G(A\cup B) \geq \mut_G(A)+\mut_G(B).$$
     \item \emph{(Left invariance)} For all $g\in G$, $\mut_G(gA)=\mut(A)$.
     \item \emph{(Right invariance)} If $G$ is unimodular, then for all $g\in G$, $\mut_G(Ag)=\mut(A)$.
 \end{enumerate}
 \end{fact}

It is easy to see that we can replace the assumption that $A$ and $B$ are compact in Kemperman's inequality in the introduction with the weaker assumption that $A$ and $B$ are $\sigma$-compact. Together with the inner regularity of $\mut_G$ (Fact~\ref{fact: Innermeasure}.2), this give us the first part of the following Fact~\ref{fact: GeneralKemperman}. The second part of Fact~\ref{fact: GeneralKemperman} follows from the fact that taking product sets preserves compactness, $\sigma$-compactness, and analyticity. Note that taking product sets in general does not preserve measurability, so we still need inner measure in this case.

\begin{fact}[Generalized Kemperman inequality for connected groups] \label{fact: GeneralKemperman} 
Suppose $\widetilde{\mu}_G$ is the inner Haar measure on $G$, and $A, B \subseteq G$ are nonempty. Then
$$ \widetilde{\mu}_G(AB) \geq \min\{ \mut_G(A)+ \mut_G(B), \mut_G(G) \}. $$
Moreover, if $A$ and $B$ are compact, $\sigma$-compact, or analytic, then we can replace $\widetilde{\mu}_G$ with $\mu_G$.
\end{fact}

 The remaining parts of Theorem~\ref{thm:mainequal} consist of classifying the minimally expanding pairs $(A,B)$ and show that they match the description in situations (iii) and (iv) of Theorem~\ref{thm:mainequal}.  For compact group, our strategy is to reduce the problem to the known situations of one dimensional tori. Hence, we need the following special case of Kneser's classification result, and the sharp dependence between $\varepsilon$ and $\delta$ is essentially due to Bilu~\cite{Bilu}.

\begin{fact}[Inverse theorem for $\TT^d$]\label{fact:inverse theorem torus}
Let $A,B$ be compact subsets of $\TT^d$. For every $\tau>0$, there is a constant $c=c(\tau)$ such that if
\[
\tau^{-1}\mu_{\TT^d}(A)\leq \mu_{\TT^d}(B)\leq \mu_{\TT^d}(A)\leq c,
\]
then either $\mu_{\TT^d}(A+B)\geq \mu_{\TT^d}(A)+2\mu_{\TT^d}(B)$, or there are compact intervals $I,J$ in $\TT$ with $\mu_\TT(I)=\mu_{\TT^d}(A+B)-\mu_{\TT^d}(B)$ and $\mu_\TT(J)=\mu_{\TT^d}(A+B)-\mu_{\TT^d}(A)$, and a continuous surjective group homomorphism $\chi:\TT^d\to\TT$, such that $A\subseteq \chi^{-1}(I)$ and $B\subseteq \chi^{-1}(J)$.
\end{fact}

When the group is a one dimensional torus $\TT$, a sharper result is recently obtained by Candela and de Roton~\cite{circle19}. The constant $c_G=3.1\cdot 10^{-1549}$ in the following fact is from an earlier result in $\ZZ/p\ZZ$ by Grynkiewicz~\cite{Grynkiewicz}. 

\begin{fact}[Inverse theorem for $\TT$]\label{fact: new inverse theorem torus}
There is a constant $c_G>0$ such that the following holds. If $A,B$ are compact subsets of $\TT$, with $\dis_\TT(A,B)<\min\{\mu_\TT(A),\mu_\TT(B),c_G\}$, and $1-\mu_\TT(A)-\mu_\TT(B)\geq 2\dis_\TT(A,B)$. Then there is a continuous surjective group homomorphism $\chi:\TT\to\TT$, and two compact intervals $I,J\subseteq \TT$ such that 
\[
\mu_\TT(I)\leq \mu_\TT(A)+\dis_G(A,B),\quad \mu_\TT(J)\leq \mu_\TT(B)+\dis_G(A,B),
\]
and $A\subseteq \chi^{-1}(I)$, $B\subseteq \chi^{-1}(J)$. 
\end{fact}

  For noncompact group, we reduce the problem to the known situation of additive group of real numbers. 
The following result can be seen as the stability theorem of the Brunn--Minkowski inequality in $\RR^d$ when $d=1$. 

\begin{fact}[Inverse theorem for $\RR$]\label{fact:Brunnd=1 R}
Let $A,B$ be compact subsets in $\RR$ with $\mu_G(A)\geq\mu_G(B)$, and let $\mu_\RR$ be the Lebesgue measure in $\RR$. Suppose we have
\[
\mu_\RR(A+B)< \mu_\RR(A)+2\mu_\RR(B).
\]
Then there are compact intervals $I,J\subseteq\RR$ with $\mu_\RR(I)=\mu_\RR(A+B)-\mu_\RR(B)$ and $\mu_\RR(J)=\mu_\RR(A+B)-\mu_\RR(A)$, such that $A\subseteq I$ and $B\subseteq J$.
\end{fact}

\section{Reduction to classifying nearly minimally expanding pairs}\label{sec: easy directions of thm 1.1}

  To set the stage for the later discussion, we would like to separate the core part of Theorem~\ref{thm:mainequal} from the more immediate parts. Throughout $G$ is a connected unimodular group, and $\mu_G$ is a Haar measure on $G$.

\begin{proposition} \label{prop:backwardof1.1}
Suppose $A, B\subseteq G$ are nonempty and compact and one of the situation listed in Theorem~\ref{thm:mainequal}
 holds, then $\mu_G(AB)=\min\{\mu_G(A)+\mu_G(B), \mu_G(G)\}.$
\end{proposition}
\begin{proof}
We will only consider situation (v) because (i-iv) are immediate and (vi) can be showed in a same way as (v). Suppose we are in situation (v) of Theorem~\ref{thm:mainequal}.  As $\chi$ is a group homomorphism, we have $AB =\chi^{-1}(I+J)$. Note that by quotient integral formula, we have $\mu_G(A) = \mu_\TT(I)$, $\mu_G(B) = \mu_\TT(J)$, $\mu_G(AB) = \mu_\TT(I+J)$. The desired conclusion follows from the easy that $\mu_\TT(I+J) = \mu_\TT(I)+\mu_\TT(J)$.
\end{proof}

  The following lemma clarifies the second statement in situation (ii) of Theorem~\ref{thm:mainequal}. Note that we do not have the later part of Theorem~\ref{thm:mainequal}(ii) when $A, B$ are not assumed to be compact. For example with $A=B=(0,1/2)+\ZZ$, we have $\mu_\TT(A)+\mu_\TT(B)= 1$, but $AB= (0, 1)+\ZZ$ is not the whole $\TT$.  

\begin{proposition}\label{prop: when a+b>G}
If either $A, B \subseteq G$ are measurable and $\mu_G(A)+\mu_G(B)> \mu_G(G)$ or  $A, B \subseteq G$ are compact and $\mu_G(A)+\mu_G(B)\geq \mu_G(G)$,  then $AB=G$.
\end{proposition}
\begin{proof}
The cases where either $A$ or $B$ is empty are immediate, so we assume that $A,B$ are nonempty.
 Suppose $g$ is an arbitrary element of $G$.  It suffices to show that $A^{-1}g$ and $ B$ has nonempty intersection. As $G$ is unimodular, $\mu_G(A) =\mu_G(A^{-1})$ by Fact~\ref{fact: measureofinverse}. Hence $\mu_G(A^{-1}g)  +\mu_G(B) = \mu_G(G)$. If $\mu_G(A^{-1}g \cap B)>0$, then we are done. Otherwise, we have $\mu_G(A^{-1}g \cap B)=0$, and so $\mu_G(A^{-1}g \cup B) =\mu_G(G)$ by the inclusion-exclusion principle. As $A$ and $B$ are compact, $A^{-1}g \cup B$ is also compact, and the complement of $A^{-1}g \cup B$ is open. Since nonempty open sets has positive measure,  $\mu_G(A^{-1}g \cup B) =\mu_G(G)$ implies $A^{-1}g \cup B = G$. Now, since $G$ is connected, we must have $A^{-1}g \cap B$ must be nonempty.
\end{proof}

The following corollary will be used many times later.

\begin{corollary}\label{cor: when a+b>G}
 If either $A \subseteq G$ is measurable and $n\mu_G(A)>\mu_G(G)$ or $A\subseteq G$ is compact and $n\mu_G(A)>\mu_G(G)$, then $A^n=G$. 
\end{corollary}

\begin{proof}
By replacing $A$ with a subset with the same measure if necessary, we can assume that $A$ is $\sigma$-compact. Then by the generalized Kemperman inequality (Fact~\ref{fact: GeneralKemperman}), we get $ \mu_G(A^{n-1})+\mu_G(A) \geq \mu_G(G)$. Applying Proposition~\ref{prop: when a+b>G}, we get the desired conclusion.
\end{proof}

  Now we clarify the situation in (iii) of Theorem~\ref{thm:mainequal}, situation (iv) can be proven in the same way.

\begin{proposition}

Suppose $A, B\subseteq G$ are nonempty, compact, and with $\mu_G(A)=0$, $0<\mu_G(B)< \mu_G(G)$, and $$\mu_G(AB)=\min(\mu_G(A)+\mu_G(B), \mu_G(G)).$$
Then there is a compact subgroup $H$ of $G$ such that $A \subseteq gH$ for some $g\in G$, and $B= B_1 \cup B_2$ with $HB_1=B_1$ and $\mu_G(AB_2)=0$. 
\end{proposition}

\begin{proof}
Without loss of generality, we can assume that $A$ and $B$ both contain $\text{id}_G$. Let $H =\Stab^0_G(B)$, let $B_1$ be the set of $b\in B$ such that whenever $U$ is an open neighborhood of $b$, we have $\mu_G(U\cap B) >0$, and let $B_2= B \setminus B_1$. We will now verify that $H$, $B_1$, and $B_2$ are as desired.

We make a number of immediate observations.  As $\mu_G(AB)= \mu_G(B)$ and $\text{id}_G$ is in $A$, we must have $A \subseteq H$. Note that $B_2$ consists of $b\in B$ such that there is open neighborhood $U$ of $B$ with $\mu_G(U \cap B)=0$. So $B_2$ is open in $B$. Moreover, if $K$ is a compact subset of $B_2$, then $K$ has a finite cover $(U)_{i=1}^n$ such that $\mu_G(U_i \cap B)=0$, which implies $\mu_G(K)=0$. It follows from inner regularity of $\mu_G$ (Fact~\ref{fact: Haarmeasurenew}(iv)), that $\mu_G(B_2)=0$. Hence, $B_1$ is a closed subset of $B$ with $\mu_G(B_1)=\mu_G(B)$. Since $B$ is compact, $B_1$ is also compact.
As $H$ is a closed subset of $G$, the compactness of $H$ follows immediately if we can show that $HB_1=HB$

It remains to verify that $HB_1=B_1$.
 As $\mu_G$ is both left and right translation invariant, we also have that for all $g\in G$, if $U$ is an open neighborhood of $gb \in gB_1$, then $\mu_G(U \cap gB_1)>0.$ 
Suppose $hb$ is in $HB_1 \setminus B_1$ with $h\in H$. Set $U = G \setminus B_1$. Then $U$ is an open neighborhood of $hb$. From the earlier discussion, we then have $\mu_G(U \cap hB_1) >0$. This implies that $\mu_G( hB \setminus B) >0$ contradicting the fact that $h \in H= \Stab^0_G$.
\end{proof}

\section{Reduction to sets with small measure}\label{sec: reduction to small}
  Throughout this section, $G$ is a connected \emph{compact} group, $\mu_G$ is the normalized Haar measure on $G$, and $A,B \subseteq G$ are $\sigma$-compact sets with positive measure.  We will show that if $(A,B)$ is nearly minimally expanding in $G$, then we can $\sigma$-compact $A'$ and $B'$ each with smaller measure such that the pair $(A',B')$ is also nearly minimally expanding. The similar approach used in this section is introduced by Tao \cite{T18} and used to obtain an inverse theorem in the abelian setting. We first prove the following easy fact, which will be used several times later in the paper.
  
   Let $f,g: G\to \mathbb C$ be functions. For every $x\in G$, we define the {\bf convolution} of $f$ and $g$ to be
\[
f*g(x)=\int_G f(y)g(y^{-1}x) \d\mu_G(y).
\]
Note that $f*g$ is not commutative, but associative by Fubini's Theorem. 
  
\begin{lemma}\label{lem:trans}
Let $t$ be any real numbers such that $\mu_G(A)^2\leq t\leq \mu_G(A)$. Then there are $x,y\in G$ such that $\mu_G(A\cap(xA))=\mu_G(A\cap(Ay))=t$.
\end{lemma}
\begin{proof}
Consider the maps: 
$$\pi_1: x\mapsto \e_A*\e_{A^{-1}}(x)=\mu_G(A\cap (xA)), \text{ and } \pi_2:y\mapsto \e_{A^{-1}}*\e_{A}(y)=\mu_G(A\cap (Ay)).$$  
By Fact~\ref{fact: Haarmeasurenew}, both $\pi_1$ and $\pi_2$ are continuous functions, and equals to $\mu_G(A)$ when $x=y=\text{id}_G$. By Fubini's theorem,
\[
\mathbb E \, (\e_A*\e_{A^{-1}})=\mu_G(A)^2=\mathbb E \, (\e_{A^{-1}}*\e_{A}).
\]
Then the lemma follows from the intermediate value theorem, and the fact that $G$ is connected.
\end{proof}

Recall that $\dis_G(A,B)=\mu_G(AB)-\mu_G(A)-\mu_G(B)$ is the discrepancy of $A$ and $B$ on $G$.  The following property is sometimes referred to as {\bf submodularity} in the literature. Note that this is not related to modular functions in locally compact groups or the notion of modularity in model theory.

\begin{lemma}\label{lem:iep}
Let $\gamma_1,\gamma_2>0$, and $A,B_1,B_2$ are $\sigma$-compact subsets of $G$. 
Suppose that $\dis_G(A,B_1)\leq \gamma_1$, $\dis_G(A,B_2)\leq \gamma_2$, and
\[
\mu_G(B_1\cap B_2)>0, \quad \text{ and }\ \mu_G(A)+\mu_G(B_1\cup B_2)\leq 1.
\]
Then both $\dis_G(A,B_1\cap B_2)$ and $\dis_G(A,B_1\cup B_2)$ are at most $\gamma_1+\gamma_2$.
\end{lemma}
\begin{proof}

Observe that for every $x\in G$ we have
\begin{equation*}
\e_{AB_1}(x)+\e_{AB_2}(x)\geq \e_{A(B_1\cap B_2)}(x)+\e_{A(B_1\cup B_2)}(x),
\end{equation*}
which implies 
\begin{equation}\label{eq:observation}
\mu_G(AB_1)+\mu_G(AB_2)\geq\mu_G(A(B_1\cap B_2))+\mu_G(A(B_1\cup B_2)).
\end{equation}
By the fact that $\dis_G(A,B_1)\leq \gamma_1$ and $\dis_G(A,B_2)\leq \gamma_2$, we obtain
\begin{align*}
    \mu_G(AB_1)\leq \mu_G(A)+\mu_G(B_1)+\gamma_1,\text{ and }
     \mu_G(AB_2)\leq \mu_G(A)+\mu_G(B_2)+\gamma_2.
\end{align*}
Therefore, by equation (\ref{eq:observation}) we have
\begin{align*}
    &\mu_G(A(B_1\cap B_2))+\mu_G(A(B_1\cup B_2))\\
    \leq&\, 2\mu_G(A)+\mu_G(B_1\cap B_2)+\mu_G(B_1\cup B_2)+\gamma_1+\gamma_2.
\end{align*}
On the other hand, as $\mu_G(B_1\cap B_2)>0$ and $\mu_G(A)+\mu_G(B_1\cup B_2)\leq1$, and using Kemperman's inequality, we have
\[
\mu_G(A(B_1\cap B_2))\geq \mu_G(A)+\mu_G(B_1\cap B_2),
\]
and 
\[
\mu_G(A(B_1\cup B_2))\geq \mu_G(A)+\mu_G(B_1\cup B_2).
\]
This implies
\[
\mu_G(A(B_1\cap B_2))\leq \mu_G(A)+\mu_G(B_1\cap B_2)+\gamma_1+\gamma_2,
\]
and 
\[
\mu_G(A(B_1\cup B_2))\leq \mu_G(A)+\mu_G(B_1\cup B_2)+\gamma_1+\gamma_2.
\]
Thus we have $\dis_G(A,B_1\cap B_2), \dis_G(A,B_1\cup B_2)\leq \gamma_1+\gamma_2$. 
\end{proof}

 The following lemma is the main result of this section, it says if $G$ admits a small expansion pair, one can another find pair of sets with sufficiently small measures, and still has small expansion. 

\begin{lemma}\label{lem:sml}
Let $d_1,d_2\in (0,1/4)$ be positive real numbers, and let $$0<\lambda=\min\{\mu_G(A),\mu_G(B),1-\mu_G(A)-\mu_G(B)\}.$$ Suppose $\dis_G(A,B)\leq\gamma$. Then there are $\sigma$-compact sets $A', B'\subseteq G$ satisfying 
\begin{enumerate}[\rm (i)]
    \item $\mu_G(A')=d_1$ and $\mu_G(B')=d_2$,
    \item $\dis_G(A',B)$, $\dis_G(A,B')$, and $\dis_G(A',B')$ are  $O_{d_1,d_2}(\gamma/m)$.
\end{enumerate}
\end{lemma}
\begin{proof}
We assume $\mu_G(A)>d_1$ and $\mu_G(A)\geq\mu_G(B)$. The case when $\mu_G(A)$ is less than $d_1$ can be proven in a similar way by replacing taking intersections by taking unions. For every $g\in G$, both $\dis_G(gA,B)$ and $\dis_G(A,Bg)$ are still upper bounded by $\gamma$.  
By Lemma~\ref{lem:trans}, for every $t$ with $\mu_G(A)^2\leq t\leq\mu_G(A)$, there is $g\in G$ such that $\mu_G(A\cap gA)=t$.
Assuming that in each step, we can choose $g$ such that $\mu_G(A\cap gA)=\mu_G(A)^2$, and replace $A$ by $A\cap gA$. Hence, after $O(\log\log 1/d_1)$ steps, the measure of $A$ will achieve $d_1$. The issue of this simple argument is that we may have $\mu_G(A\cup gA)+\mu_G(B)>1$ when $\mu_G(A)\geq 1/3$, so that we cannot apply Lemma~\ref{lem:iep}. Thus, in the first few steps, we will choose $g$ such that $\mu_G(A\cup gA)$ is not too large. 

We first consider the case when $\mu_G(A)\geq 1/3$, and $\mu_G(A)-\lambda\geq\mu_G(A)^2$. We are going to choose $g\in G$ such that
\begin{equation}\label{eq: condition for union}
2\mu_G(A)-\mu_G(A\cap gA)+\mu_G(B)=\mu_G(A\cup gA)+\mu_G(B)\leq 1,
\end{equation}
and $\mu_G(A\cap gA)\geq \max\{d_1,\mu_G(A)-\lambda\}$. Such $g$ exists by Lemma~\ref{lem:trans}. Let $A_1=A\cap gA$, then $\mu_G(A_1)\leq \mu_G(A)-\lambda,\mu_G(A)^2$. By Lemma~\ref{lem:iep}, $\dis_G(A_1,B)\leq 2\gamma$. Next we choose $g_1\in G$ satisfying \eqref{eq: condition for union} with $A$ replaced by $A_1$, and  $\mu_G(A_1\cap g_1A_1)\geq \min\{d_1,\mu_G(A_1)^2\}$. Let $A_2=A_1\cap g_1A_1$, then 
\[
\mu_G(A_2)\leq \max\{\mu_G(A_1)-2\lambda,\mu_G(A_1)^2\},
\]
and $\dis_G(A_2,B)\leq 4\gamma$. Repeat this procedure for $t_1$ steps until either $\mu_G(A_{t_1})=d_1$, or $\mu_G(A_t)-2^{t-1}\lambda\leq \mu_G(A_t)^2$. In either case, we have $t_1\leq \log (1/3\lambda)$. 

Next, if $\mu_G(A_{t_1})>d_1$, we choose $g_{t_1}$ in $G$ such that $\mu_G(A_{t_1}\cap g_{t_1}A_{t_1})=\mu_G(A_{t_1})^2$. By the way we define $t_1$, we have $\mu_G(A_{t_1}\cup g_{t_1}A_{t_1})+\mu_G(B)\leq 1$. Set $A_{t_1+1}=A_{t_1}\cap g_{t_1}A_{t_1}$. Repeat this procedure for $t_2$ steps until $\mu_G(A_{t_1+t_2})=d_1$. We have
\[
t_2\leq \log\frac{\log d_1}{\log \mu_G(A_{t_1})}\leq \log\log \frac{1}{d_1}, 
\]
and $\dis_G(A_{t_1+t_2},B)\leq 2^{t_1+t_2}\gamma =O_{d_1}(\gamma/\lambda)$. 
We then apply the same procedures for $B$ to arrange $B$ having measure $d_2$.

If we have $\mu_G(A)<1/3$ at the beginning, we are able to choose $g$ such that $\mu_G(A\cap gA)=\mu_G(A)^2$ and $\mu_G(A\cup gA)+\mu_G(B)\leq 1$. Hence, it only requires at most $\log\log (1/d_1)$ steps to make $A$ having measure $d_1$. 
\end{proof}

\section{Geometry of minimal and nearly minimal expansions I}\label{sec: geometry I}
This section studies the shape of a nearly minimally expanding pairs\ relative to a connected compact normal subgroup of the ambient topological group such that the images of the pair under the quotient map have small measure.
In Section~\ref{sec: 6.1}, we obtain results that will allow us to reduce the  Theorem~\ref{thm:mainequal} and Theorem~\ref{thm:mainapproximate} to analogous result about a simpler quotient group.  Section~\ref{sec: 6.2} applies Section~\ref{sec: 6.1} to reduce Theorem~\ref{thm:mainequal} and Theorem~\ref{thm:mainapproximate} to the case of Lie groups and also prove a coarse version of these results. Section~\ref{sec: 6.3} applies Section~\ref{sec: 6.1} to further reduce Theorem~\ref{thm:mainequal} and Theorem~\ref{thm:mainapproximate} to the problem of constructing suitable group homomorphism onto either $\TT$ or $\RR$.

Throughout this section, $G$ is a connected unimodular locally compact group with Haar measure $\mu_G$, and $A$ and $B$ are $\sigma$-compact subsets of $G$ with positive $\mu_G$-measure. We will assume familiarity with the preliminary Section~\ref{sec: 3.1} on locally compact group and Haar measure.

\subsection{Preservation of minimal expansion under quotient}\label{sec: 6.1}
In this section, $H$ is a \emph{connected compact} normal subgroup of $G$, 
so $H$ and $G/H$ are unimodular by Fact~\ref{fact: unimodular}. Let $\mu_H$, and $\mu_{G/H}$ be the Haar measure on $G$, $H$, and $G/H$, and let $\mut_G$ and $\mut_{G/H}$ be the inner Haar measures on $G$ and $G/H$. 

Suppose $r$ and $s$ are in $\RR$. We set
$$  A_{(r,s]}  := \{a \in A: \mu_H(A \cap aH) \in (r,s]  \}  $$
and
$$ \pi A_{(r,s]}:=  \{aH \in G \slash H : \mu_H( A \cap aH) \in (r,s]   \}. $$ 
In particular, $\pi A_{(r,s]}$ is the image of $A_{(r,s]}$ under the map $\pi$. We define $B_{(r',s']}$ and $\pi B_{(r',s']}$ likewise for $r', s' \in \RR$.
We have a number of immediate observations.  

\begin{lemma} \label{lem: corofmeasurability}
Let $r, s, r', s'$ be in $\RR^{>0}$. For all $aH \in \pi A_{(r,s]}$, $bH \in \pi B_{(r',s']}$, the sets $A_{(r,s]} \cap aH$, $B_{(r',s']} \cap bH$ are nonempty $\sigma$-compact.
For all subintervals $(r,s]$ of $(0,1]$,    $A_{(r,s]}$  is $\mu_G$-measurable and $\pi A_{(r,s]}$ is $\mu_{G/H}$-measurable.
\end{lemma}
\begin{proof}
The first assertion is immediate from the definition.
Let $\e_A$ be the indicator function of $A$. Then the function 
\begin{align*}
    \e^H_A: G/H &\to R\\
    aH &\mapsto \mu_H(A \cap aH)
\end{align*}
is well-defined and measurable by Lemma~\ref{lem: mesurability}. As $ \pi A_{(r,s]} = (\e^H_A)^{-1}(r, s]$ and 
$$A_{(r,s]} = A \cap \pi^{-1}( \pi A_{(r,s]}),$$ we get the second assertion.
\end{proof}

Note that $\pi A_{(r,s]}\pi B_{(r',s']}$ is not necessarily $\mu_{G/H}$-measurable, so Lemma~\ref{lem: intuition}(ii) does require the inner measure $\mut_{G/H}$.

\begin{lemma} \label{lem: intuition}
We have the following:
\begin{enumerate}[\rm (i)]
      \item For every $aH \in \pi A$ and $bH \in \pi B$,
    $$ \mu_H \big((A \cap aH)( B \cap bH) \big) \geq \min\{ \mu_H( A \cap aH)+ \mu_H(  B \cap bH), 1\}. $$
    \item  If $A_{(r,s]}$ and $B_{(r',s']}$ are nonempty, then 
        $$\mut_{G/H}(\pi A_{(r,s]} \pi B_{(r',s']}) \geq \min\{ \mu_{G/H}(\pi A_{(r,s]}) + \mu_{G/H}(\pi B_{(r',s']}), \mu_{G/H}(G/H)\}.$$
\end{enumerate}
\begin{proof}
Note that both $H$ and $G/H$ are connected. So (i) is a consequence of the generalized Kemperman inequality for $H$ (Fact~\ref{fact: GeneralKemperman}) and (ii) is a consequence of the generalized Kemperman inequality for $G/H$ (Fact~\ref{fact: GeneralKemperman}). 
\end{proof}
    
\end{lemma}

As the functions we are dealing with are not differentiable,  we will need Riemann--Stieltjes integral which we will now recall. Consider a closed interval $[a,b]$ of $\RR$, and functions $f:[a,b] \to \RR$ and $g: \RR \to \RR$. A {\bf partition} $P$ of $[a,b]$ is a sequence $(x_i)_{i=0}^n$ of real numbers with $x_0 =a$, $x_n=b$, and $x_i< x_{i+1}$ for $i \in \{0, \ldots, n-1\}$. For such $P$, its {\bf norm} $\Vert P\Vert$ is defined as $\max_{i=0}^{n-1}|x_{i+1}-x_i|$, and a corresponding {\bf partial sum} is given by $S(P,f, g) = \sum_{i=0}^n f(c_{i+1})(g(x_{i+1})-g(x_i))$ with $c_{i+1} \in [x_i, x_i+1]$.
We then define
$$ \int_{a}^b f(x)\d{g(x)}: = \lim_{\Vert P \Vert \to 0}  S(P,f,g)$$
if this limit exists where we let $P$ range over all the partition of $[a,b]$ and $S(P,f, g)$ ranges over all the corresponding partial sums of $P$. The next fact records some basic properties of the integral.
\begin{fact}\label{fact: RS int}
Let $[a,b]$, $f(x)$,  and $g(x)$ be as above. Then we have:
\begin{enumerate}[\rm (i)]
    \item{\rm (Integrability)} If $f(x)$ is continuous on $[a,b]$, and $g(x)$ is monotone and bounded on $[a,b]$, then $f(x)\d g(x)$ is  Riemann--Stieltjes integrable on $[a,b]$.
    \item{\rm (Integration by parts)} If $f(x) \d g(x)$ is Riemann--Stieltjes integrable on $[a,b]$, then $g(x) \d f(x)$ is also Riemann--Stieltjes integrable on $[a,b]$, and
    \[
\int_a^b f(x)\d g(x)=f(b)g(b)-f(a)g(a)-\int_a^b g(x)\d f(x).
\]
\end{enumerate}
\end{fact}

   The next lemma uses ``spillover'' estimate, which gives us a lower bound estimate on $\mu_G(AB)$ when the projection of $A$ and $B$ are not too large. 

\begin{lemma} \label{lem: Keyestimatequotientcompact}
Suppose $\mu_{G\slash H}( \pi A )+ \mu_{G \slash H} (\pi B) <1$. Set $\alpha =\sup_{a \in A} \mu_H( A \cap aH) $, $\beta =\sup_{b \in B} \mu_H( B \cap bH)$, and $\gamma =\max\{1, \alpha+\beta\}$. Then
\begin{align*}
    \mu_G(AB)  &\geq   \frac{\alpha+\beta}{\gamma}\left( \mu_{G/H}(\pi A_{(\alpha/\gamma,\alpha]}) + \mu_{G/H}(\pi B_{(\beta/\gamma,\beta]})\right)\\
    &\quad + \frac{\alpha+\beta}{\alpha} \mu_G(A_{(0,\alpha/\gamma]}) + \frac{\alpha+\beta}{\beta}\mu_G(B_{(0,\beta/\gamma])}.
\end{align*}
\end{lemma}
\begin{proof}
For $x \in (0,1]$, set $C_x = AB\cap \pi^{-1}(\pi A_{(x\alpha, \alpha]}\pi B_{(x\beta, \beta]})$. One first note that 
$$\mu_G(AB) \geq \mut_G(C_0) .$$ By Fact~\ref{fact: RS int}(1), $\mathrm{d}\mut_G( C_x)$ is Riemann--Stieltjes integrable on any closed  subinterval of $[0,1]$. Hence,
\[
    \mut_G( C_0)  = \mut_G ( C_{1/\gamma}) - \int_0^{\tfrac{1}{\gamma}}\d \mut_G(C_x). 
\]
Lemma~\ref{lem: corofmeasurability} and Lemma~\ref{lem: intuition}(1) give us that
\[
\mut_G (  C_{1/\gamma}) \geq \mut_{G/H} (  \pi{A}_{(\alpha/\gamma,\alpha]}   \pi{B}_{(\beta/\gamma,\beta]}). 
\]
Likewise, for $x, y \in \RR^{>0}$ with $x <y\leq 1/\gamma$,  $\mut_G (C_x) - \mut_G(C_y)  $ is at least  
$$r(\alpha+\beta) \left(\mut_{G/H} (  \pi {A}_{(x\alpha,\alpha]}   \pi {B}_{(x\beta,\beta]}) - \mut_{G/H}(\pi {A}_{(y\alpha,\alpha]}   \pi {B}_{(y\beta,\beta]})\right).$$ 

Therefore, 
 \[\mut_G( C_0) \geq \mut_{G/H} (  \pi{A}_{(\alpha/\gamma,\alpha]}   \pi{B}_{(\beta/\gamma,\beta]}) - \int_0^{\tfrac{1}{\gamma}}(\alpha+\beta) x\d \mut_{G/H}( \pi{A}_{(x\alpha, \alpha]} \pi{B}_{(x\beta, \beta]}).
 \]
Using integral by parts (Fact~\ref{fact: RS int}.2), we get
 \[\mut_G( C_0) \geq  \int_0^{\tfrac{1}{\gamma}}\mut_{G/H}( \pi{A}_{(x\alpha, \alpha]} \pi{B}_{(x\beta, \beta]})\d (\alpha+\beta) x.
 \]
Applying Lemma~\ref{lem: intuition}.2 and using the assumption that  $\mu_{G\slash H}( \pi A )+ \mu_{G \slash H} (\pi B) <1$ , we have
\[\mut_G( C_0)  \geq  \int_0^{\tfrac{1}{\gamma}}(\mu_{G/H}( \pi{A}_{(x\alpha, \alpha]}) + \mu_{G/H}( \pi{B}_{(x\beta, \beta]}))\d(\alpha+\beta) x.
 \]
Using integral by parts (Fact~\ref{fact: RS int}.2), we arrive at
\begin{align*}
    \mut_G( C_0)  &\geq    \frac{\alpha+\beta}{\gamma}\left(\mu_{G/H}( \pi{A}_{(\alpha/\gamma, \alpha]}) + \mu_{G/H}( \pi{B}_{(\beta/\gamma, \beta]}) \right)  \\
    & \quad  -\int_0^{\tfrac{1}{\gamma}}(\alpha+\beta) x\d( \mu_{G/H}( \pi{A}_{(x\alpha, \alpha]})+ \mu_{G/H}(\pi{B}_{(x\beta, \beta]})).
\end{align*}
As $\d( \mu_{G/H}( \pi{A}_{(x\alpha, \alpha]})+ \mu_{G/H}(\pi{B}_{(x\beta, \beta]}))= -\d( \mu_{G/H}( \pi{A}_{(0, x\alpha]})+ \mu_{G/H}(\pi{B}_{(0, x\beta]}))$,
\begin{align*}
    \mut_G( C_0)  &\geq    \frac{\alpha+\beta}{\gamma}\left(\mu_{G/H}( \pi{A}_{(\alpha/\gamma, \alpha]}) + \mu_{G/H}( \pi{B}_{(\beta/\gamma, \beta]}) \right)  \\
    & \quad  +\int_0^{\tfrac{1}{\gamma}}(\alpha+\beta) x\d( \mu_{G/H}( \pi{A}_{(0, x\alpha]})+ \mu_{G/H}(\pi{B}_{(0,x\beta]})).
\end{align*}  
Finally, recall that 
\[ \int_0^{1/\gamma} x\alpha\d\mu_{G/H}( \pi{A}_{(0, x\alpha]})= \mu_G({A}_{(0, \alpha/\gamma]}) \text{ and } \int_0^{1/\gamma} \beta x\d\mu_{G/H}( \pi{B}_{(0, x\beta]})= \mu_G({B}_{(0, \beta/\gamma]}).\] 
Thus, we arrived at the desired conclusion.
\end{proof}

 The next result in the main result in this subsection. It says if the projections of $A$ and $B$ are not too large, the small expansion properties will be kept in the quotient group.

\begin{theorem}[Quotient domination] \label{thm: criticality transfer}
Suppose $ {\mu}_{G/H}( \pi A) +  {\mu}_{G/H} ( \pi B ) <  \mu_{G/H}(G/H) $ and $\dis_G(A,B)< \min\{\mu_G(A), \mu_G(B)\}$.
Then there are $\sigma$-compact $A', B' \subseteq G/H$ such that
$$  \dis_{G \slash H}(A', B')  < 7\dis_G(A,B)$$
and $\max\{ \mu_{G} (A \tri \pi^{-1} A'), \mu_{G} (B \tri \pi^{-1} B')\} < 3\dis_G(A,B).$
\end{theorem}

\begin{proof}
Let $\alpha$ and $\beta$ be as in Lemma~\ref{lem: Keyestimatequotientcompact}. We first show that $\alpha+\beta\geq 1$. Suppose to the contrary that $\alpha+\beta<1$. Then Lemma~\ref{lem: Keyestimatequotientcompact} gives us
\[
\mu_G(AB) \geq \frac{\alpha+\beta}{\alpha} \mu_G(A) + \frac{\alpha+\beta}{\beta}\mu_G(B)
\]
It follows that $\mu_G(AB)> \mu_G(A)+\mu_G(B)+\min\{ \mu_G(A), \mu_G(B)\}$, a contradiction.

Now we have $\alpha+\beta \geq 1$. Hence, Lemma~\ref{lem: Keyestimatequotientcompact} yields
\begin{align*}
    \mu_G(AB)  &\geq    \mu_{G/H}(\pi A_{(\alpha/(\alpha+\beta),\alpha]}) + \mu_{G/H}(\pi B_{(\beta/(\alpha+\beta),\beta]})\\
    &\quad + \frac{\alpha+\beta}{\alpha} \mu_G(A_{(0,\alpha/\gamma]}) + \frac{\alpha+\beta}{\beta}\mu_G(B_{(0,\beta/(\alpha+\beta)])}.
\end{align*}
Choose $\sigma$-compact $A' \subseteq \pi A_{(\alpha/(\alpha+\beta),\alpha]}$ and $B' \subseteq \pi B_{(\beta/(\alpha+\beta),\beta]})$ $\sigma$-compact such that  \[
\mu_{G/H}(A')= \mu_{G/H}(\pi A_{(\alpha/(\alpha+\beta),\alpha]}) \text{ and } \mu_{G/H}(B') = \mu_{G/H}(\pi B_{(\beta/(\alpha+\beta),\beta]}).\] We will verify that $A'$ and $B'$ satisfy the desired conclusion.

Since $\mu_{G/H}(A') \geq (1/ \alpha) \mu_{G}(A_{(\alpha/(\alpha+\beta),\alpha]})$, $\mu_{G/H}(B') \geq (1/ \beta) \mu_{G}(B_{(\beta/(\alpha+\beta),\beta]})$ and $\alpha+\beta>1$, we have
\[ \mu_G(AB) \geq \frac{1}{\alpha} \mu_G(A)+\frac{1}{\beta}\mu_G(B).  \]
From $\mu_G(AB)- \mu_G(A)-\mu_G(B)= \dis_G(A,B)\leq \min\{\mu_G(A),\mu_G(B)\}$, we deduce that $\alpha, \beta \geq 1/2$.

By our assumption $\mu_G(AB) < \mu_G(A)+ \mu_G(B)+\dis_G(A,B)$. Hence, 
\begin{align*}
    \dis_G(A,B)  &\geq    \mu_{G/H}(A') - \mu_{G}( A_{(\alpha/(\alpha+\beta),\alpha]}) + \mu_{G/H}(B') - \mu_{G}( B_{(\beta/(\alpha+\beta),\beta]}) \\
    &\quad + \frac{\beta}{\alpha} \mu_G(A_{(0,\alpha/\gamma]}) + \frac{\alpha}{\beta}\mu_G(B_{(0,\beta/(\alpha+\beta)])}.
\end{align*}
Therefore, $\mu_{G/H}(A') - \mu_{G}( A_{(\alpha/(\alpha+\beta),\alpha]})$ and $(\beta/\alpha) \mu_G(A_{(0,\alpha/\gamma]})$ are at most $ \dis_G(A,B)$. Noting also that $\beta/\alpha \leq 1/2$, we get $\mu_G( A \tri \pi^{-1}(A') \leq 3\dis_G(A,B)$. A similar argument yield $\mu_G( B \tri \pi^{-1}(B') \leq 3\dis_G(A,B)$.

Finally, note that $ \pi^{-1}\left(A'B'\right)$ is equal to $A_{(\alpha/(\alpha+\beta),\alpha]}B_{(\beta/(\alpha+\beta),\beta]}$, which is a subset of $AB$. Combining with $\mu_G(AB) < \mu_G(A)+ \mu_G(B)+\dis_G(A,B)$, we get
\[
\mu_{G/H}(A'B') \leq  \mu_G(A)+\mu_G(B)+\dis_G(A,B)\leq \mu_{G/H}(A')+\mu_{G/H}(B')+ 7\dis_G(A,B),\]
which completes the proof.
\end{proof}

The next corollary of the proof of Theorem~\ref{thm: criticality transfer} gives a complementary result when without the assumption that $\mu_{G/H}(\pi A)+\mu_{G/H}(\pi B)< \mu_{G/H}(G/H)$.

\begin{corollary}\label{cor:transfer to Lie noncompact}
Suppose $G$ is noncompact and $\dis_G(A,B)=0$. Then there are $\sigma$-compact $A', B' \subseteq G/H$ such that  $\dis_{G/H}(A',B')=0$, $\mu_G(A\tri \pi^{-1}A')=0$, and $\mu_G(B\tri \pi^{-1}B')=0$.
\end{corollary}

\begin{proof}
Choose an increasing sequence $(A_n)$ of compact subsets of $A$ and an increasing sequence $(B_n)$ of compact subsets of $B$ such that $A=\bigcup^\infty_{n=0} A_n$ and $B= \bigcup_{i=0}^\infty B_n $. Then $\lim_{n \to \infty}\dis_G(A_n, B_n) =0$. For each $n$, $A_n$ and $B_n$ are compact, so $\pi A_n$ and $\pi B_n$ are also compact and has finite measure. Let $A_n'$  and $B'_n$ be defined for $A_n$ and $B_n$ as in the proof of Theorem~\ref{thm: criticality transfer}. Then for $n$ sufficiently large, we have
\[
\mu_G( \pi^{-1}A'_n \tri A_n) < 3\dis_G(A_n, B_n) \text{ and } \mu_G(\pi^{-1}B'_n \tri B_n) < 3\dis_G(A_n,B_n)
\]
 and 
\[
\mu_{G/H}(A_n'B_n')<\mu_{G/H}(A_n')+\mu_{G/H}(B_n')+5\dis_G(A_n, B_n). 
\]
Moreover, we can arrange that the sequences $(A'_n)$ and $(B'_n)$ are increasing.
Take $A' = \bigcup_{n=1}^\infty A_n'$ and $B' = \bigcup_{n=1}^\infty B_n'$. By taking $n\to\infty$, we have
$$  \mu_G(\pi^{-1} A' \tri A) =0 \text{ and } \mu_G(\pi^{-1} B' \tri B) =0.$$
and $\dis_{G/H}(A', B')=0$ as desired.
\end{proof}

\subsection{Coarse versions of the main theorems}\label{sec: 6.2}
  For the given $G$, there might be no continuous surjective group homomorphism to either $\TT$ or $\RR$ (e.g. $G = \mathrm{SO}_3(\RR))$. However, the famous theorem below by Gleason~\cite{Gleason} and Yamabe~\cite{Yamabe} allows us to naturally obtain continuous and surjective group homomorphism to a Lie group. Using together with Corollary~\ref{cor:transfer to Lie noncompact}, this allows us to reduce the noncompact case of Theorem~\ref{thm:mainequal} to that of Lie group.
  The connectedness of $H$ is not often stated as part of the result, but can be arranged by replacing $H$ with its identity component.

\begin{fact}[Gleason--Yamabe Theorem]\label{fact:Yamabe}
For any connected locally compact group $G$
and any neighborhood $U$ of the identity in $G$, there is a connected compact normal subgroup $H\subseteq U$ of $G$
such that $G/H$ is a  Lie group.
\end{fact}

 With some further effort, we can also arrange that $\mu_{G/H}(\pi A)+\mu_{G/H}(\pi B)< \mu_{G/H}(G/H)$ as necessary to apply Theorem~\ref{thm: criticality transfer}. However, when $\dis_G(A,B)>0$, we will need a dimension control on the Lie group we obtained from the Gleason--Yamabe Theorem. For that, we need Fact~\ref{fact: thesis}, which can be thought of as a refinement of the Gleason--Yamabe theorem coming from arithmetic combinatorics and model theory.

 Recall that an open and precompact set $S\subseteq G$ is a {\bf $K$-approximate} group if $\id \in S$,  $S^{-1}=S$, and $S^2\subseteq XS$ for some finite set $X$ of cardinality $K$. The next theorem by Tao~\cite{T08} allows us to extract an approximate group from a nearly minimally expanding pair; as stated in~\cite{T08}, this theorem is only applicable when $A,B$ are open, but the proof also goes through without this assumption. 

\begin{fact}[Approximate groups from small expansion]\label{fact: find approximate group}
Suppose $K$ is a constant and $\mu_G(AB)<K\mu_G^{1/2}(A)\mu_G^{1/2}(B)$, then there is an open precompact $O(K^{O(1)})$-approximate group $S$ with $\mu_G(S)= O(K^{O(1)})\mu_G^{1/2}(A)\mu_G^{1/2}(B)$
and a finite set $X$ or cardinality $O(K^{O(1)})$ such that $A\subseteq XS$ and $B\subseteq SX$.
\end{fact}

  The study of continuous approximate groups by Carolino~\cite{thesis} is able to find a  Lie model, and to control the dimension of the Lie model. This can be seen as a finer version of the Gleason--Yamabe theorem.

\begin{fact}[Lie model from approximate groups]\label{fact: thesis}
Suppose $K$ is a constant and $S$ is an open precompact  $K$-approximate group on $G$. Then there is a connected compact normal subgroup $H$ of $G$, such that $H \subseteq S^4$ and $G/H$ is a Lie group of dimension $O_K(1)$.
\end{fact}

Fact~\ref{fact: thesis} can also be deduced from the main theorem in~\cite{MW} and the strong approximate group theory in~\cite{BGT}. The following sketch was explained to us by Arturo Rodriguez Fanlo. We assume the reader is sufficiently familiar with model theory and the results in~\cite{BGT}

\begin{proof}[Sketch of proof of Fact~\ref{fact: thesis}] 
Suppose to the contrary. Take an ultraproduct $(G,S)$ of counterexamples $(G_n,S_n)$. By the main result in~\cite{MW}, using the ultraproduct of the Haar measures, $G^{00}$ is contained in $S^4$, so this ultraproduct has a Lie model $L=G/H$. 

Now, using a result in~\cite{BGT}, namely, the strong approximate groups theory, we can find a definable strong approximate group $S'\subseteq S^4$ such that $O_K(1)$ left-translates of $S'$ cover $S$. By \L o\'s' theorem, for sufficiently large $n$, this gives us strong approximate groups in the factors $S'_n \subseteq S^4_n$ such that $O_K(1)$ left-translates of $S'_n$ cover $S_n$. 

By the theory of strong approximate groups, taking $H_n$ the subgroups of non-escaping elements, we get that $G_n=\langle S'_n\rangle$ and $G_n/H_n$ has dimension $O_K(1)$, a contradiction.
\end{proof}

  The next lemma is the main result in this subsection. Using this, we can pass the problem to connected Lie groups with bounded dimensions.

\begin{lemma}[Lie model from small expansions]\label{lem:gleason2}
If $\mu_G(AB) \leq K \mu^{1/2}_G(A)\mu^{1/2}_G(B)$ and then there is a connected compact subgroup $H$ of $G$ such that  $G/H$ is a Lie group of dimension $O_K(1)$ and, with $\pi: G \to G/H$ the quotient map, $\pi A$ and $\pi B$ have $\mu_G$-measure 
 $O(K^{O(1)})\mu^{1/2}_G(A)\mu^{1/2}_G(B)$. 
\end{lemma}
\begin{proof}
By Fact~\ref{fact: find approximate group}, there is an open $K$-approximate group $S$, with  
\[ 
\mu_G(S)= O(K^{O(1)})\mu_G^{1/2}(A)\mu_G^{1/2}(B)
\]
such that $A$ can be covered by $ O(K^{O(1)})$ right translation of $S$, and $B$ can be covered by $ O(K^{O(1)})$ left translation of $S$. By Fact~\ref{fact: thesis}, there is a closed connected normal subgroup $H$ in $S^4$, such that $G/H$ is a Lie group of dimension at most $O_K(1)$. Let $\pi$ be the quotient map. Since $H \subseteq S^{4}$, we have
\[
\mu_{G/H}(\pi(S)) = \mu_G(SH)\leq\mu_G(S^5) = O(K^{O(1)})\mu_G^{1/2}(A)\mu_G^{1/2}(B).
\]
Note that $\pi(A)$ can
be covered by $O(K^{O(1)})$ right translations of $\pi(S)$, and $\pi(B)$ can
be covered by $O(K^{O(1)})$ left translations of $\pi(S)$.
 Hence, we get the desired conclusion.
\end{proof}

The following proposition tells us that small measure expansion phenomenon can always be reduced to the same phenomenon on a Lie group with small dimension.

\begin{proposition}[Coarse versions of the main theorems] \label{prop: coarsetheorem}
There is a constant $\tau$ such that if either $G$ is noncompact and $\mu_G(A)=\mu_G(B)$ or $G$ is compact and  $\mu_G(A)= \mu_G(B)<\tau$, and  $$\dis_G(A,B)< \min\{\mu_G(A),\mu_G(B)\},$$
then there is a connected compact normal subgroup $H$ of $G$, and $\sigma$-compact subsets $A', B'$ of $G/H $ satisfying:
\begin{enumerate}[\rm (i)]
    \item $G/H$ is a Lie group of dimension $O(1)$;
    \item With $\pi:G \to G/H$ the quotient map,  \[ \mu_G(A \tri\pi^{-1}A')< 3 \dis_G(A,B) \text{ and } \mu_G(B \tri\pi^{-1}B')< 3 \dis_G(A,B);\]
    \item $\dis_{G/H}(A', B') < 7\dis_{G}(A, B).$
\end{enumerate}
\end{proposition}

\begin{proof}
From the assumption we have $\mu_G(AB)<3\mu^{1/2}_G(A)\mu_G^{1/2}B$. Obtain $H$ as in Lemma~\ref{lem:gleason2}, and when $\mu_G(A),\mu_G(B)$ are small enough, we have $\mu_{G/H}(\pi A)+\mu_{G/H}(\pi B)<1$. Then $G/H$ is a Lie group of dimension $O(1)$. By applying Theorem~\ref{thm: criticality transfer}, we get the desired conclusion.
\end{proof}

\subsection{Structure control on the nearly minimally expanding pairs}\label{sec: 6.3}

 The following useful lemma is a corollary of Theorem~\ref{thm: criticality transfer}, which will be used at various points in the later proofs. It tells us the character given in the parallel Bohr sets is essentially unique. 
 
\begin{lemma}[Stability of characters]\label{lem:fromsmalltobig}
Suppose $G$ is compact, $\chi: G\to\TT$ is a continuous surjective group homomorphism, $J\subseteq \TT$ is a compact interval, and $\eta$ is a constant. 
Suppose we have
\begin{enumerate}[\rm (i)]
\item $\dis_G(A,B)<\min\{\mu_G(A),\mu_G(B),c_G/7\}$ where $c_G$ is from Fact~\ref{fact: new inverse theorem torus};
    \item $\mu_\TT(J)=\mu_G(B)$ and $\mu_G(B\tri\chi^{-1}(J))=\eta \dis_G(A,B)$;
    \item $\mu_G(A)+\mu_G(B)\leq 1-(2\eta+20)\dis_G(A,B)$ and $(3\eta+30)\dis_G<\mu_G(B)$.
\end{enumerate}
 Then there is a compact interval $I$ in $\TT$ such that $\mu_\TT(I)=\mu_G(A)$ and $$\mu_G(A\tri\chi^{-1}(I))\leq 10\dis_G(A,B).$$
\end{lemma}
\begin{proof}
Let $H=\ker(\chi)$.
Note that $\chi$ is an open map, and hence by Fact~\ref{fact: first iso thm} we have $G/H\cong \TT$. By Theorem~\ref{thm: criticality transfer}, there are sets $A'$ and $B'$ in $\TT$ such that $\dis_{\TT}(A',B')\leq 7\dis_G(A,B)$, and 
\[
\mu_G(A\tri \chi^{-1}(A'))=3\dis_G(A,B),\quad \mu_G(B\tri \chi^{-1}(B'))=3\dis_G(A,B).
\]
Then $\mu_\TT(B'\tri J)\leq (\eta+3)\dis_G(A,B)$. By Fact~\ref{fact: new inverse theorem torus}, there is a continuous surjective group homomorphism $\rho:\TT\to\TT$ and two compact intervals $I_A,I_B\subseteq\TT$, such that
\[
\mu_\TT(\rho^{-1}(I_A)\tri A')\leq 7\dis_G(A,B),\quad \mu_\TT(\rho^{-1}(I_B)\tri B')\leq7\dis_G(A,B), 
\]
and $\mu_\TT(A')=\mu_\TT(I_A)$, $\mu_\TT(B')=\mu_\TT(I_B)$. 
On the other hand, as $\mu_\TT(B'\tri J)\leq (\eta+3)\dis_G(A,B)$, thus $\mu_\TT( \rho^{-1}(I_B)\tri J )\leq (\eta+10)\dis_G(A,B)$. Assume $\rho$ is not the identity map, then by an elementary analysis, the fact that $1-\mu_\TT(I)>(2\eta+20)\dis_G(A,B)$, and the upper bound on $\eta$, we have that $\rho=\mathrm{id}$. Set $I=I_A$. Then, we have \[
\mu_G(A\tri\chi^{-1}(I))\leq \mu_G(A\tri\chi^{-1}(A'))+\mu_\TT(A'\tri I)\leq 10\dis_G(A,B)
\]
which finishes the proof. 
\end{proof}

The next lemma shows that, if the symmetric difference of a set $A$ and an interval is small, then $A$ is also contained in an interval of bounded length. To handle arbitrary sets, we use inner measure here.
\begin{lemma}\label{lem: from sym dif to subset}
Suppose $G$ is compact, $\widetilde{\mu}_G$ is the inner measure associated to $\mu_G$, and $A,B \subseteq G$ has $\widetilde{\dis}_G(A,B)<\varepsilon$ with 
\[
\widetilde{\dis}_G(A,B) = \widetilde{\mu}_G(A,B)-\widetilde{\mu}_G(A)- \widetilde{\mu}_G(B). 
\]
Assume further that $\widetilde{A}\subseteq A$ and $\widetilde{B}\subseteq B$ are $\sigma$-compact with $\mu_G(\widetilde{A}) =\widetilde{\mu}_G(A)$ and $\mu_G(\widetilde{B}) =\widetilde{\mu}_G(B)$, $\chi:G\to \TT$ is a continuous surjective group homomorphism, and $I,J$ compact intervals in $\TT$, with $\mu_\TT(I)=\mu_G(A),\mu_\TT(J)=\mu_G(B)$, and
\[
\mu_G(\widetilde{A}\tri\chi^{-1}(I))<\varepsilon,\quad \mu_G(\widetilde{B}\tri \chi^{-1}(J))<\varepsilon.
\]
Then there are intervals $I',J'\subseteq\TT$, such that $A\subseteq\chi^{-1}(I')$, $B\subseteq\chi^{-1}(J')$, and 
\[
\mu_\TT(I')-\widetilde{\mu}_G(A)<10\varepsilon,\quad\mu_\TT(J')-\widetilde{\mu}_G(B)<10\varepsilon. 
\]
\end{lemma}
\begin{proof}

We will show that for all $g \in A\setminus \chi^{-1}(I)$, the distance between $\chi(g)$ and $I$   in $\TT$ is at most $5\varepsilon$, and  for all $g' \in B\setminus \chi^{-1}(J)$, the distance between $g'$ and $J$ is at most $5\varepsilon$. This implies there are intervals $I',J'$ in $\TT$ such that $A\subseteq \chi^{-1}(I')$ and $B\subseteq \chi^{-1}(J')$, and
\[
\mu_G(\chi^{-1}(I')\setminus A)<10\varepsilon\, \quad\mu_\TT(\chi^{-1}(J')\setminus B)<10\varepsilon,
\]
as desired. Observe that $\widetilde{A}\widetilde{B}$ is a $\sigma$-compact subset of $AB$, and so $\dis_G(\widetilde{A}, \widetilde{B})<\varepsilon$.

By symmetry, it suffices to show the statement for $g \in A\setminus \chi^{-1}(I)$. Suppose to the contrary that $g$ is in $A\setminus \chi^{-1}(I)$, and the distance between $\chi(g)$ and $I$ in $\TT$ is strictly greater than $5\varepsilon$. By replacing $\widetilde{A}$ with $\widetilde{A}\cup\{g\}$ if necessary, we can assume $g \in \widetilde{A}$.
We then have
\[
\mu_\TT(\chi(g)\chi(\widetilde{B})\setminus I\chi(\widetilde{B}))\geq 5\varepsilon-\mu_\TT(J\setminus \chi(\widetilde{B}))\geq 4\varepsilon. 
\]
and this implies that $\mu_G(g\widetilde{B}\setminus \chi^{-1}(I)\chi^{-1}(J))\geq3\varepsilon$. Therefore,
\begin{align*}
\mu_G(\widetilde{A}\widetilde{B})&\geq \mu_G\big((\chi^{-1}(I)\cap \widetilde{A})(\chi^{-1}(J)\cap \widetilde{B})\big)+\mu_G(g\widetilde{B}\setminus \chi^{-1}(I)\chi^{-1}(J))\\
&\geq \mu_G(\widetilde{A})+\mu_G(\widetilde{B})-2\varepsilon+3\varepsilon,
\end{align*}
and this contradicts the fact that $\dis_G(\widetilde{A},\widetilde{B})<\varepsilon$.  
\end{proof}

The stability lemma, together with Theorem~\ref{thm: criticality transfer}, will be enough to derive a different proof of a theorem by Tao~\cite{T18}, with a sharp exponent bound. As we mentioned in the introduction, the same result with a sharp exponent bound was also obtained by Christ and Iliopoulou~\cite{ChristIliopoulou} recently, via a different approach. 
\begin{theorem}[Theorem~\ref{thm:mainapproximate} for compact abelian groups]\label{thm: abelian case}
   Let $G$ be a connected compact abelian group, and $A,B$ be compact subsets of $G$ with positive measure. Set $$\lambda =\min\{\mu_G(A),\mu_G(B),1-\mu_G(A)-\mu_G(B)\}.$$
  Given $0<\varepsilon<1$, there is a constant $K=K(\lambda)$ does not depend on $G$, such that if $\delta<K\varepsilon$ and
  \[
  \mu_G(A+B)<\mu_G(A)+\mu_G(B)+\delta\min\{\mu_G(A),\mu_G(B)\}.
  \]
 Then there is a surjective continuous group homomorphism $\chi: G \to \TT$ together with two compact intervals $I,J\in \mathbb T$ with
 \[
 \mu_\TT(I)-\mu_G(A)<\varepsilon\mu_G(A),\quad \mu_\TT(J)-\mu_G(B)<\varepsilon\mu_G(B),
 \]
 and $A\subseteq \chi^{-1}(I)$, $B\subseteq\chi^{-1}(J)$. 
\end{theorem}
\begin{proof}
We first assume that $\dis_G(A,B)$ is sufficiently small, and we will compute the bound on $\dis_G(A,B)$  later. As  $G$ is abelian, by Proposition~\ref{prop: coarsetheorem}, there is a quotient map $\pi: G\to\TT^d$, and $A',B'\subseteq \TT^d$, such that  \[ \mu_G(A \tri\pi^{-1}A')< 3 \dis_G(A,B) \text{ and } \mu_G(B \tri\pi^{-1}B')< 3 \dis_G(A,B)\]
and $\dis_{G/H}(A', B') < 7\dis_{G}(A, B).$ Let $c=c(\tau)$ be as in Fact~\ref{fact:inverse theorem torus}, and by Lemma~\ref{lem:sml}, there is a constant $L$ depending only on $\lambda$ and $c$, and sets $A'',B''\subseteq\TT^d$ with $\mu_{\TT^d}(A'')=\mu_{\TT^d}(B'')=c$ such that
\[
\max\{\dis_G(A'',B'), \dis_G(A'',B''), \dis_G(A',B'')\}<L\dis_G(A,B). 
\]
By Fact~\ref{fact:inverse theorem torus}, there are intervals $I',J'\subseteq\TT$ with $\mu_\TT(I')=\mu_\TT(J')=c$, and a continuous surjective group homomorphism $\rho:\TT^d\to\TT$, such that 
\[
\mu_{\TT^d}(A''\tri \rho^{-1}(I'))<L\dis_G(A,B)\quad\text{and}\quad \mu_{\TT^d}(B''\tri \rho^{-1}(J'))<L\dis_G(A,B). 
\]
By Lemma~\ref{lem:fromsmalltobig}, there are intervals $I',J'\subseteq \TT$ with
\[
\mu_{\TT^d}(A'\tri \rho^{-1}(I'))<10L\dis_G(A,B)\text{ and } \mu_{\TT^d}(B'\tri \rho^{-1}(J'))<10L\dis_G(A,B). 
\]
Let $\chi=\pi\circ\rho$. Hence, we have
\[
\mu_{G}(A\tri \chi^{-1}(I'))<(3+10L)\dis_G(A,B)\text{ and } \mu_{G}(B\tri \chi^{-1}(J'))<(3+10L)\dis_G(A,B). 
\]
Using Lemma~\ref{lem: from sym dif to subset}, there are intervals $I,J\subseteq\TT$, such that $A\subseteq\chi^{-1}(I)$, $B\subseteq\chi^{-1}(J)$, and
\begin{align*}
    &\mu_\TT(I)-\mu_G(A)<(30+100L)\dis_G(A,B),\\
    &\mu_\TT(J)-\mu_G(B)<(30+100L)\dis_G(A,B).
\end{align*}
Now, we fix 
\[
K:=\min\Big\{\frac{1}{30+100L}, \frac{c}{L}\Big\},
\]
and $\delta<K\varepsilon$, where $\dis_G(A,B)=\delta\min\{\mu_G(A),\mu_G(B)\}$. Clearly, we will have 
\begin{align*}
    &\mu_\TT(I)-\mu_G(A)<\varepsilon\min\{\mu_G(A),\mu_G(B)\},\\
    &\mu_\TT(J)-\mu_G(B)<\varepsilon\min\{\mu_G(A),\mu_G(B)\}.
\end{align*}
Note that in the above argument, we apply Fact~\ref{fact:inverse theorem torus} on $A'',B''$, and this would require that $L\dis_G(A,B)<c$. By the way we choose $K$, we have
\[
L\dis_G(A,B)=L\delta\min\{\mu_G(A),\mu_G(B)\}<c,
\]
as desired. 
\end{proof}

  The following theorem shows that, once we have a certain group homomorphism to tori, we will get a good structural control on the (nearly) minimal expansion sets.

\begin{proposition}[Toric domination from a given homomorphism]\label{prop: struc on AB}
Suppose $A, B$ have $\dis_G(A,B)<\min\{\mu_G(A), \mu_G(B)\}$, and  $\chi:G\to\TT$ is a continuous surjective group homomorphism such that $\mu_\TT(\chi(A))+\mu_\TT(\chi(B))<1/5$. Then there is a continuous and surjective group homomorphism $\rho: G \to \TT$, a constant $K_0$ only depending on $\min\{\mu_G(A),\mu_G(B)\}$, and compact intervals $I, J \subseteq \TT$ with $\mu_\TT(I)=\mu_G(A)$ and $\mu_\TT(J)=\mu_G(B)$, such that 
\[
\mu_G(A\tri \rho^{-1}(I))<K_0\dis_G(A,B),\quad \text{and}\quad\mu_G(B\tri \rho^{-1}(J))<K_0\dis_G(A,B).
\]
\end{proposition}
\begin{proof}
By Theorem~\ref{thm: criticality transfer}, there are $A',B'\subseteq \TT$, such that 
\begin{equation}\label{eq: struc for AB R case}
\mu_G(A\tri \chi^{-1}(A'))<3\dis_G(A,B) \text{ and } \mu_G(B\tri \chi^{-1}(B'))<3\dis_G(A,B),
\end{equation}
and $\dis_\TT(A',B')<7\dis_G(A,B)$. By Theorem~\ref{thm: abelian case}, there are continuous surjective group homomorphism $\eta: \TT \to \TT$, a constant $L$ depending only on $\min\{\mu_G(A),\mu_G(B)\}$, and compact intervals $I, J \subseteq \TT$ such that $\mu_\TT(I)=\mu_\TT(A')$,  $\mu_\TT(J)=\mu_\TT(B')$, and 
\begin{equation}\label{eq: struc for AB R case2}
\mu_\TT(A'\tri \eta^{-1}(I))<L\dis_G(A,B) \text{ and } \mu_\TT(B'\tri \eta^{-1}(J))<L\dis_G(A,B).
\end{equation}
Set $\rho = \eta \circ \chi$. The conclusion follows from \eqref{eq: struc for AB R case} and \eqref{eq: struc for AB R case2} with $K_0=L+3$.
\end{proof}

  In light of Proposition~\ref{prop: struc on AB}, in the rest of the paper, we will be focusing on finding the desired group homomorphism mapping $G$ to tori.




\section{Pseudometrics and group homomorphisms onto tori}\label{section:pseudometric}

Proposition~\ref{prop: coarsetheorem} and Proposition~\ref{prop: struc on AB} reduce the proof of Theorem~\ref{thm:mainequal} and Theorem~\ref{thm:mainapproximate} to problems of constructing certain group homomorphisms from a Lie group onto $\TT$ or $\RR$. In this section, we show these problems can be reduced further to problems of constructing  pseudometrics with certain properties on the ambient group.  Section~\ref{sec: 7.1} shows that a linear pseudometric suffices, and Section~\ref{sec: 7.2} and Section~\ref{sec: 7.3} does so when the pseudometric is almost linear and almost monotone.

 Throughout, $G$ is a  {\it connected} and {\it unimodular} \emph{Lie group} with Haar measure $\mu_G$. Recall that a {\bf pseudometric} on a set $X$ is a function $d: X \times X \to \RR$ satisfying the following three properties:
\begin{enumerate}
    \item (Reflexive) $d(a,a) =0$ for all $a \in X$,
    \item (Symmetry) $d(a,b)= d(b,a)$ for all $a,b \in X$,
    \item (Triangle inequality) $d(a,c)  \leq d(a,b) +d(b,c) \in X$.
\end{enumerate}
Hence, a pseudometric on $X$ is a metric if for all $a,b \in X$, we have  $d(a,b)=0$ implies $a =b$. If $d$ is a pseudometric on $G$, for an element $g\in G$, we set $\|g\|_d = d(\id, g)$. 

\subsection{Linear pseudometrics}\label{sec: 7.1}

 Suppose $d$ is a pseudometric on $G$. We say that $d$ is {\bf left-invariant} if for all $g, g_1, g_2 \in G$, we have $d(gg_1, gg_2) = d(g_1,g_2)$. left-invariant pseudometrics arise naturally from measurable sets in a group; the pseudometric we will construct in Section~\ref{sec: geometry II} is of this form.

\begin{proposition}\label{prop: construct pseudo-metric}
Suppose  $A$ is a measurable subset of $G$. For $g_1$ and $g_2$ in $G$, define
$$ d(g_1, g_2) = \mu_G(A) - \mu_G( g_1A \cap g_2A). $$
Then
$d$ is a continuous left-invariant pseudometric on $G$.
\end{proposition}
 \begin{proof}
 We first verify the triangle inequality. Let $g_1$, $g_2$, and $g_3$ be in $G$, we need to show that 
\begin{equation}\label{eq: pseudo metric d_A}
    \mu_G( A) - \mu_G(g_1A \cap g_3 A) \leq \mu_G( A) - \mu_G(g_1A \cap g_2 A)  + \mu_G( A) - \mu_G(g_2A \cap g_3 A).
\end{equation}
As $\mu_G(A) = \mu_G(g_2 A)$, we have $\mu_G( A) - \mu_G(g_1A \cap g_2 A) = \mu_G( g_2A \setminus g_1 A)$, and   $\mu_G( A) - \mu_G(g_2A \cap g_3 A) = \mu_G( g_2A \setminus g_3 A)$. Hence, \eqref{eq: pseudo metric d_A} is equivalent to 
$$   \mu_G( g_2A) - \mu_G( g_2A\setminus g_1A) -\mu_G(g_2A\setminus g_3A) \leq \mu(g_1A \cap g_3A). $$
Note that the left-hand side is at most $ \mu_G( g_1A \cap g_2A \cap g_3A)$, which is less than the right-hand side. Hence, we get the desired conclusion. The continuity of $d$ follows from Fact~\ref{fact: Haarmeasurenew}(vii), and the remaining parts  are straightforward.
 \end{proof}

Another natural source of  left-invariant pseudometrics is group homomorphims onto metric groups. Suppose $\widetilde{d}$ is a continuous left-invariant metric on a group $H$ and $\pi: G \to H$ is a group homomorphism, then for every $g_1,g_2$ in $G$, one can naturally define a pseudometric $d(g_1, g_2) = \widetilde{d}( \pi(g_1), \pi(g_2))$. It is easy to see that such  $d$ is a continuous left-invariant pseudometric, and $\{g\in G : \Vert g \Vert_d=0\}= \ker(\pi)$ is a normal subgroup of $G$.  The latter part of this statement is no longer true for an arbitrary continuous left-invariant pseudometric, but we still have the following:

\begin{lemma} \label{lem: kerd}
Suppose $d$ is a continuous left-invariant pseudometric on $G$. 
Then the set
 $\{g \in G : \|g\|_d =0 \}$
is the underlying set of a closed subgroup of $G$.
\end{lemma}

\begin{proof}
Suppose $g_1$ and $g_2$ are elements in $G$ such that $\|g_1\|_d = \|g_2\|_d=0$. Then 
\[
d(\id, g_1g_2) \leq d(\id, g_1) + d(g_1, g_1g_2) = d(\id, g_1)+ d(\id, g_2)=0.
\]
Now, suppose $(g_n)$ is a sequence of elements in $G$ converging to $g$ with $\|g_n\|_d=0$. Then $\|g\|_d=0$ by continuity, we get the desired conclusions.
\end{proof}

  In many situations, a left-invariant pseudometric allows us to construct surjective continuous group homomorphism to metric groups.
The following lemma tells us precisely when this happens. We omit the proof as the result is motivationally relevant but will not be used later on.
\begin{lemma} Let $d$ be a continuous left-invariant pseudometric on $G$. The following are equivalent,
\begin{enumerate}[\rm (i)]
    \item The set
 $\{g \in G : \|g\|_d =0 \}$
is the underlying set of a closed normal subgroup of $G$.
    \item There is a continuous surjective group homomorphism  $\pi: G \to H$, and $\widetilde{d}$ is a left-invariant metric on $H$. Then
    \[
    d(g_1, g_2) = \widetilde{d}( \pi g_1, \pi g_2). 
    \]
\end{enumerate}
Moreover, when (ii) happens, $\{g \in G : \|g\|_d =0 \} = \ker \pi$, hence $H$ and $\widetilde{d}$ if exist are uniquely determined up to isomorphism. 
\end{lemma}

  The group $\RR$ and $\TT = \RR/ \ZZ$ are naturally equipped with the metrics $d_{\RR}$ and $d_{\TT}$ induced by the Euclidean norms, and these metrics interact in a very special way with the additive structures. Hence one would expect that if there is a group homomorphism from $G$ to either $\RR$ or $\TT$, then $G$ can be equipped with a pseudometric which interacts nontrivially with addition.

 Let $d$ be a left-invariant pseudometric on $G$. The {\bf radius} $\rho$ of $d$ is defined to be $\sup\{\|g\|_d : g \in G\}$; this is also $\sup\{d(g_1, g_2) : g_1, g_2 \in G\}$ by left invariance. We say that  $d$ is {\bf locally linear} if it satisfies the following properties:
\begin{enumerate}
    \item $d$ is continuous and left-invariant;
    \item for all $g_1$, $g_2$, and $g_3$ with $d(g_1, g_2)+d(g_2,g_3) < \rho$, we have either
    \begin{equation}\label{eq: key of linear pseudometric}
        d(g_1,g_3) = d(g_1, g_2) + d(g_2,g_3), \text{ or } d(g_1,g_3) = |d(g_1, g_2) - d(g_2,g_3)|.
    \end{equation}
\end{enumerate}
A pseudometric $d$ is {\bf monotone} if for all $g\in G$ such that $\|g\|_d< \rho/2$, we have $$\|g^2\|_d= 2 \|g\|_d.$$ To investigate the property of this notion further, we need the following fact about the adjoint representations of Lie groups~\cite[Proposition 9.2.21]{hilgert}.
\begin{fact}\label{fact: adjoint}
Let $\mathfrak{g}$ be the Lie algebra of $G$, and let $\mathrm{Ad}: G\to \mathrm{Aut}(\mathfrak{g})$ be the adjoint representation. Then $\ker(\mathrm{Ad})$ is the center of $G$. 
\end{fact}

The following result is the first time we need $G$ to be a  Lie group instead of just a locally compact group.

\begin{proposition}\label{prop: automaticmonotone}
If $d$ is a locally linear pseudometric on $G$,  then $d$ is monotone.
\end{proposition}

\begin{proof}
We first prove an auxiliary statement.
\begin{claim}
Suppose $s: G \to G, g \mapsto g^2$ is the squaring map. Then there is no open $U \subseteq G$ and proper closed subgroup $H$ of $G$ such that $s(U) \subseteq H$.\medskip

\noindent\emph{Proof of Claim.}
Consider the case where $G$ is a connected component of a linear algebraic subgroup of $\mathrm{GL}_n(\RR)$. Let $J_s$ be the Jacobian of the function $s$. Then the set $$\{ g\in G : \det J_s(g)=0\}$$
has the form $G \cap Z$ where $Z$ is a solution set of a system of polynomial equations. It is not possible to have $G\cap Z = G$, as $s$ is a local diffeomorphism at $\id$. Hence, $G\cap Z$ must be of strictly lower dimension than $G$. By the inverse function theorem, $s|_{G \setminus Z}$ is open. Hence $s(U)$ is not contained in a subgroup of $G$ with smaller dimension.

We also note a stronger conclusion for abelian Lie group: If $V$ is an open subset of a not necessarily connected abelian Lie group $A$, then the image of $A$ under $a \mapsto a^2$ is not contained in a closed subset of $A$ with smaller dimension. Indeed, $A$ is isomorphic as a topological group to $D \times \TT^m \times \RR^m$, with $D$ a discrete group. If $$U \subseteq D \times \TT^m \times \RR^m, $$ then it is easy to see that $\{ a^2 : a\in V\}$ contains a subset of $D \times \TT^m \times \RR^m$, and is therefore not a subset of a closed subset of $A$ with smaller dimension.

Finally, we consider the general case. Suppose to the contrary that $s(U) \subseteq H$ with $H$ a proper closed subgroup of $G$. Let $Z(G)$ be the center of $G$, $G'= G/Z(G)$, $\pi: G \to G'$ be the quotient map, $U'= \pi(U)$, and
$$s': G' \to G', g' \mapsto (g')^2.$$ Then $U'$ is an open subset of $G'$, which is isomorphic as a topological group to a connected component of an algebraic group by Fact~\ref{fact: adjoint}. By the earlier case, $s'(U')$ is not contained in any proper closed subgroup of $G'$, so we must have $\pi(H)=G'$. In particular, this implies $\dim(H \cap Z(G))< \dim Z(G)$, and $HZ(G)=G$. Choose $h \in  H$ such that $hZ(G) \cap U$ is nonempty. Then 
$$ s(hZ(G) \cap U) = \{  h^2  a^2 : a \in Z(G) \cap h^{-1} U   \}. $$
As $s(hZ(G) \cap U)  \subseteq H$, we must have $\{ a^2 : a \in Z(G) \cap h^{-1} U \}$ is a subset of $H\cap Z(G)$. Using the case for abelian Lie groups, this is a contradiction, because  $H\cap Z(G)$ is a closed subset of $Z(G)$ with smaller dimension.
\end{claim}

We now get back to the problem of showing that $d$ is monotone. As $d$ is invariant, $d(\id, g)= d(g, g^2)$ for all $g\in G$.
From local linearity of $d$, for all $g\in G$ with $\|g\|_d< \rho/2$, we either have $$\|g^2\|_d= 2 \|g\|_d \quad\text{or}\quad \|g^2\|_d=0.$$ It suffices to rule out the possibility that $0<\|g\|_d<\rho/4$, and $\|g^2\|_d= 0$. 

As $d$ is continuous, there is an open neighborhood $W$ of $g$ such that for all $g' \in W$, we have $\|g'\|_d>0$ and $\|(g')^2\|_d=0$. From Lemma~\ref{lem: kerd}, the set $\{g \in G :  \|g\|_d =0\}$ is a closed subgroup of $G$. As $d$ is nontrivial and $G$ is a connected Lie group,  $\{g \in G :  \|g\|_d =0\}$ must be a Lie group with smaller dimension. Therefore, we only need to show that if $W$ is an open subset of $G$, then $s(W)$ is not contained in a closed subgroup of $G$ with smaller dimension, where $s:G\to G$ is the squaring map, and this is guaranteed by the earlier claim.
\end{proof}

  The next result confirms our earlier intuition: locally linear pseudometric in $G$ will induce a homomorphism mapping to either $\TT$ or $\RR$.

\begin{proposition}\label{prop: strong linear}
Suppose $d$ is a locally linear pseudometric with radius $\rho>0$. Then $\ker d$ is a normal subgroup of $G$, $G/\ker d$ is isomorphic to $\TT$ if $G$ is compact, and $G/\ker d$ is isomorphic to $\RR$ if $G$ is noncompact.
\end{proposition}
\begin{proof}
We first prove that $\ker d$ is a normal subgroup of $G$.
Suppose $\|g\|_d =0$ and $h\in G$ satisfies $\|h\|_d< \rho/4$. We have 
\begin{align*}
d(h, hgh^{-1}) &= d(\id, gh^{-1}) \\
&=|d(\id,g)\pm d(g, gh^{-1})| = d(\id, h^{-1})= d(\id, h).    
\end{align*}
Hence, $d(\id, hgh^{-1})=|d(\id,h)\pm d(h,hgh^{-1})|$ is either $0$ or $2d(\id, h)$. Assume first that $\|hgh^{-1}\|_d=0$ for every such $h$ when $\|g\|_d=0$. Let 
$$U:=\{h: \|h\|_d< \rho/4\}.$$ By the continuity of $d$, $U$ is open. Hence for every $h$ in $G$, $h$ can be written as a finite product of elements in $U$. By induction, we conclude that for every $h\in G$, $\|hgh^{-1}\|_d=0$ given $\|g\|_d=0$, and this implies that $\ker d$ is normal in $G$.

Suppose $\|hgh^{-1}\|_d = 2 \|h\|_d$. By Proposition~\ref{prop: automaticmonotone}, $d$ is monotone. Hence, we have   $$\| hg^2h^{-1}\|_d = 4 \| h\|_d.$$ On the other hand, as $\|g\|_d=0$, repeating the argument above, we get $\| hg^2h^{-1}\|_d$ is either $0$ or $2\| h\|_d$. Hence, $\| h\|_d=0$, and so $\| hgh^{-1}\|_d=0$.

We now show that $G'= G/\ker d$ has dimension $1$. Let $d'$ be the pseudometric on $G'$ induced by $d$.  Choose $g\in G'$ in the neighborhood of $\mathrm{id}_{G'}$ such that $g$ is in the image of the exponential map and $\|g\|_{d'}< \rho/4$. If $g'$ is another element in the neighborhood of $\mathrm{id}_{G'}$ which is in the image of the exponential map and $\| g'\|_{d'}< \rho/4$. Without loss of generality, we may assume $\|g'\|_{d'}\leq\| g\|_{d'}$. Suppose $g'=\exp(X)$. Then, by monotonicity, there is $k\geq1$ such that $\|(g')^k\|_{d'}\geq\| g\|_{d'}$. By the continuity of the exponential map, there is $t\in (0,1]$ such that $$\|g\|_{d'}=\|\exp(tkX)\|_{d'}.$$ This implies that $g$ and $g'$ are on the same one parameter subgroup, which is the desired conclusion.
\end{proof}

\subsection{Almost linear pseudometrics: relative sign and total weight functions}\label{sec: 7.2}
 In this section, we will introduce a weakening of the notion of a locally linear pseudometric and define the relative sign function and total weight function associate to it. When $d$ is a pseudometric arising from a measurable subset $A$ as in Proposition~\ref{prop: construct pseudo-metric}, these roughly give the ``direction'' and the ``distance'' that an element of the group translates $A$.
 
Throughout this section,  $d$ is a pseudometric on $G$ with radius $\rho>0$, and $\gamma$ is a constant with $0<\gamma<10^{-8}\rho$. For a constant $\lambda$,  we write $I(\lambda)$ for the interval $(-\lambda,\lambda)$ in either $\RR$ or $\TT$, and we write $N(\lambda)$ for $\{g \in G : \|g\|_d \in I(\lambda)\}$. By Fact~\ref{fact: Haarmeasurenew}(vii), $N(\lambda)$ is an open set, and hence measurable.
We say that $d$ is {\bf $\gamma$-linear} if it satisfies the following conditions:
\begin{enumerate}
    \item $d$ is continuous and left-invariant;
    \item for all $g_1, g_2, g_3 \in G$ with $d(g_1, g_2)+ d(g_2, g_3) <\rho-\gamma$, we have either
\[
d(g_1,g_3) \in   d(g_1, g_2) + d(g_2,g_3) + I(\gamma), 
\]
or
\[
d(g_1,g_3) \in   |d(g_1, g_2) - d(g_2,g_3)| + I(\gamma).
\]
\end{enumerate}
Given $\alpha\leq \rho$, let $N(\alpha)=\{g\in G:\|g\|_d\leq\alpha\}$. 
We say that $d$ is {\bf $\gamma$-monotone} if for all $g \in N(\rho/2-\gamma)$, we have 
\[
\|g^2\|_d \in 2\|g\|_d+I(\gamma). 
\]

  The next lemma says that under the $\gamma$-linearity condition, the group $G$ essentially has only one ``direction'': if there are three elements have the same distance to $\id$, then at least two of them are very close to each other. 
\begin{lemma} \label{lem:dichotomy}
Suppose $d$ is a $\gamma$-linear pseudometric on $G$. If $g, g_1, g_2 \in G$ such that
\[
\| g \|_d =\| g_1 \|_d = \| g_2 \|_d  \in    I(\rho/4-\gamma) \setminus I(2\gamma),
\]
and $d(g_1, g_2) \in 2\| g \|_d + I(\gamma)$. Then either $d(g, g_1) \in I(\gamma)$ or $d(g, g_2) \in I(\gamma)$.
\end{lemma}
\begin{proof}
Suppose both $d(g, g_1)$ and $d(g, g_2)$ are not in $I(\gamma)$. By $\gamma$-linearity of $d$, we have 
\[
d(g,g_1)\in |d(\id,g)\pm d(\id,g_1)|+I(\gamma),
\]
and so $d(g,g_1)\in 2\|g\|_d+I(\gamma)$. Similarly, we have  $d(g,g_2) 2\|g\|_d+I(\gamma)$. 

Suppose first that $ d(g_1, g_2)\in d(g, g_1)+d(g, g_2) + I(\gamma)$, then
 \[
 d(g_1, g_2)\in 4\|g\|_d+I(3\gamma).
 \]
 On the other hand, by $\gamma$-linearity we have $d(g_1,g_2)\leq 2\|g\|_d+\gamma.$
 Hence, we have $\|g\|_d \in I(2\gamma)$, a contradiction. 
 
  The other two possibilities are
$d(g_1, g_2)+ d(g, g_2) \in d(g, g_1) +I (\gamma)$ or $d(g_1, g_2)+d(g, g_1) \in d(g, g_2)+ I (\gamma)$, but similar calculations also lead to contradictions.
\end{proof}

Proposition~\ref{prop: localmonotoneimplyglobalmonotone} below is a partial replacement for Proposition~\ref{prop: automaticmonotone} for linear pseudometric. The fact that we do not automatically have monotonicity is a reason that the later Section~\ref{subsec: almost linear fiber} is much harder than Section~\ref{subsec: linear fiber}.

\begin{proposition}[Path monotonicity implies global monotonicity]\label{prop: localmonotoneimplyglobalmonotone}
Let $\mathfrak{g}$  be the Lie algebra of $G$, $\exp: \mathfrak{g} \to G$ the exponential map, and $d$ a $\gamma$-linear pseudometric on $G$. Suppose for each $X$ in $\mathfrak{g}$, we have one of the following two possibilities:
\begin{enumerate}[\rm (i)]
    \item $\|\exp(tX)\|_d < \gamma$ for all $t \in \RR$;
    \item there is $t_0\in  \RR^{>0}$ with 
$\|\exp(t_0X)\|_d \in I(\rho/2-\gamma)\setminus I(\rho/4)$,  
\begin{equation}\label{eq: condition (1)}
\|\exp(2t_0X)\|_d = 2\|\exp(t_0X)\|_d+ I(\gamma), \end{equation} 
and 
\begin{equation}\label{eq: condition (2)}
\|\exp(tX)\|_d + \|\exp((t_0-t)X)\|_d \in \|\exp(t_0X)\|_d+I(\gamma)
\end{equation}
for all $t\in [0, t_0]$.
\end{enumerate}
Then $d$ is $(9\gamma)$-monotone.
\end{proposition}
\begin{proof}
 Fix an element $g$ of $G$ with $\|g\|_d \in I(\rho/2 -16\gamma)$. Our job is to show that $\|g^2\|_d \in 2\|g\|_d + I(9\gamma)$. Since $G$ is compact and connected, the exponential map $\exp$ is surjective. We get $X \in \mathfrak{g}$ such that $g\in \{ \exp(tX) : t \in \RR\}$. If we are in scenario (i), then  $\|g\|_d< \gamma$, hence $ \|g^2\|_d \in 2\|g\|_d +I(3\gamma)$. Therefore, it remains to deal with the case where we have an $t_0$ as in (ii). 
 
  Set $g_0 = \exp(t_0X)$. We consider first the special case where $\|g\|_d < \|g_0\|_d-2\gamma$.  As $d$ is continuous, there is $t_1 \in [0, t_0]$ such that with $g_1= \exp(t_1X)$, we have $\|g_1\|_d =\|g\|_d$. Let $t_2 = -t_1$, and $g_2 = \exp(t_2X) = g_1^{-1}$. Since $d$ is invariant, 
  \[
  \|g_2\|_d = d( g^{-1}_1, \id) = d( \id, g_1) = \|g_1\|_d.
  \]
  Hence,  $\|g_1\|_d = \|g_2\|_d = \| g \|_d$. If $\| g\|_d <2\gamma$, then $\|g^2\|_d \in 2\|g\|_d+ I(5\gamma)$ and we are done. Thus we suppose $\| g\|_d \geq 2\gamma$. Then, by Lemma~\ref{lem:dichotomy}, either $d(g, g_1)<\gamma$, or $d(g, g_2) <\gamma$. 
  
  Since these two cases are similar, we assume that $d(g, g_1)<\gamma$. By $\gamma$-linearity, $\|g_1^2\|_d$ is in either
 $
  2\|g_1\|_d+I(\gamma)$ or $I(\gamma)$. Using $\|g_0^2\|_d \in 2\|g_0\|_d+I(\gamma)$ and the assumption that $\|g\|_d < \|g_0\|_d-2\gamma$, in either case, we have 
  \begin{equation} \label{g1andg0}
      \| g^2_1\|_d< \|g^2_0\|_d-2\gamma. 
  \end{equation}
   
  Since $g^{-1}_0g_1= g_1 g_0^{-1}$, and by $\gamma$-linearity of $d$,  we get
  \begin{equation}\label{eq:d(g_1^2, g^2_0)}
       d(g_1^2, g^2_0) = d( \id, g^{-2}_1g_0^{2}) = d( \id, (g^{-1}_1g_0)^2) \in  \{ 0, 2d(g_1, g_0)\} + I(\gamma).  
       \end{equation}
  By \eqref{eq: condition (2)}, we have $\|g_1\|_d+ d(g_1, g_0) \in \| g_0\|_d+I(\gamma)$. Recalling that $\|g_1\|_d =\|g\|_d >2\gamma$, and from \eqref{eq: condition (1)} and \eqref{eq:d(g_1^2, g^2_0)}, we have
  \begin{equation}\label{eq: g_1^2, g_0^2}
  d(g^2_1, g^2_0)< 2\| g_0\|_d-3\gamma = \|g_0^2 \| -2\gamma.
    \end{equation}
By \eqref{g1andg0}, \eqref{eq: g_1^2, g_0^2}, and the $\gamma$-linearity of $d$, we have
\[
\|g_1^2\|_d\in \|g_0^2\|_d-d(g_1^2,g_0^2)+I(\gamma).
\]
Therefore by \eqref{eq: condition (2)} and \eqref{eq:d(g_1^2, g^2_0)}, we have either
\begin{align*}
   \|g_1^2\|_d\in  2\|g_1\|_d +I(5\gamma) \quad\text{or}\quad \|g_1^2\|_d\in  2\|g_0\|_d +I(3\gamma).
\end{align*}
As $\|g_1\|_d^2 \leq  2\|g_1\|+ \gamma <2\|g_0\|-5\gamma$, we must have $\|g_1^2\| \in 2\|g_1\|+I(5 \gamma)$. Now, since $\|g^{-1}_1g\|_d= d(g_1, g) < \gamma$, again by the $\gamma$-linearity we conclude that
\[
d(g^2_1, g^2)= \|(g^{-1}_1g)^2\|_d < 3\gamma. 
\]
  Thus, $\|g^2\|_d \in  2\|g\|_d+I(9\gamma).$
  
  Finally, we consider the other special case where $\|g_0\|_d+2\gamma<\|g\|_d<\rho/2-16\gamma$. For $g_1=\exp(t_1X)$ with $t_1\in [0,t_0]$,  we have $\|g_1^2\|_d\in 2\|g_1\|+I(8\gamma)$ by a similar argument as above. Using continuity, we can choose $t_1$ such that $\|g_1^2\|_d=\|g\|_d$, and let $g_2 = g_1^{-1}$. The argument goes in exactly the same way with the role of $g_1$ replaced by $g_1^2$ and the role of $g_2$ replaced by $g_2^2$.
\end{proof}

  Suppose $d$ is $\gamma$-linear.
We define $s(g_1,g_2)$ to be the {\bf relative sign}   for  $g_1, g_2 \in G$ satisfying   $\|g_1\|_d +\|g_2\|_d< \rho -\gamma$ by
$$ s(g_1, g_2) = 
\begin{cases}
0 &\text{ if }  \min\{\|g_1\|_d, \|g_2\|_d \}\leq  4\gamma, \\
1 &\text{ if } \min\{\| g_1 \|_d, \| g_2 \|_d \}>  4\gamma \text{ and }  \| g_1g_2 \|_d \in \| g_1\|_d + \| g_2 \|_d + I(\gamma).  \\
-1 &\text{ if } \min\{\| g_1 \|_d, \| g_2 \|_d \}>  4\gamma \text{ and }  | g_1g_2 |_d \in \big| \| g_1\|_d - \| g_2 \|_d \big| + I(\gamma).
\end{cases}
$$
Note that this is well-defined because when $\min\{\| g_1 \|_d, \| g_2 \|_d \}\geq  4\gamma$ in the above definition, the differences between $\big| \| g_1\|_d - \| g_2 \|_d \big|$ and $ \| g_1\|_d + \| g_2 \|_d$ is at least $6\gamma$. The following lemma gives us tools to relate signs between different elements.

\begin{proposition}\label{prop: propertoesofsign}
Suppose $d$ is $\gamma$-linear and $\gamma$-monotone. Then for $g_1$, $g_2$,  and $g_3$ in $N(\rho/4-\gamma)\setminus N(4\gamma)$, we have the following
\begin{enumerate}[\rm (i)]
    \item $s(g_1, g^{-1}_1)=-1$ and $s(g_1, g_1)=1$.
    \item $s(g_1, g_2) =s(g_2, g_1)$.
    \item $ s(g_1, g_2) = s(g^{-1}_1, g^{-1}_2) = -s(g^{-1}_1, g_2) = -s(g_1, g^{-1}_2)  $.
    \item $s(g_1, g_2)s(g_2,g_3)s(g_3,g_1)=1.$
    \item If $\|g_1\|_d\leq  \|g_2\|_d$, and $g_1g_2$ is in $N(\rho/4-\gamma)\setminus N(4\gamma)$, then $$s(g_0,g_1g_2) = s(g_0,g_2g_1) = s(g_0, g_2).$$
\end{enumerate}
\end{proposition}

\begin{proof}
As  $g_1$, $g_2$,  and $g_3$ are in $N(\rho/4-\gamma)\setminus N(4\gamma)$, one has $s(g_i, g_j) \neq 0$ for all $i, j \in \{1, 2, 3\}$.
The first part of (i) is immediate from the fact that $\|\id\|_d =0$, and the second part of (i) follows from  the $\gamma$-monotonicity and the definition of the relative sign. 

We now prove (ii). Suppose to the contrary that $s(g_1,g_2) = -s(g_2,g_1)$. Without loss of generality, assume $s(g_1, g_2)=1$. Then $\|g_1g_2g_1g_2\|_d$ is in $2\|g_1g_2\|_d  + I(\gamma)$, which is a subset of $2\|g_1\|_d+ 2\|g_2\|_d + I(3 \gamma)$.  On the other hand, as $s(g_2,g_1)=-1$, we have
\[
\|g_1g_2g_1g_2\|_d\in \big| \|g_1\|_d \pm (  \|g_2\|_d -\|g_1\|_d)\pm\|g_2\|_d \big| +I(3 \gamma). 
\]
This contradicts the assumption that $g_1$ and $g_2$ are not in $N(4\gamma)$.

Next, we prove the first and third equality in (iii).
Note that $\|g\|_d = \|g^{-1}\|_d$ for all $g\in G$ as $d$ is symmetric and invariant. Hence, $\|g_1g_2\|_d = \|g_2^{-1}g_1^{-1}\|_d$. This implies that $s(g_1, g_2) = s(g_2^{-1}, g_1^{-1})$. Combining  with (ii), we get the first equality in (iii). The third equality in (iii) is a consequence of the first  equality in (iii). 

Now, consider the second equality in (iii). Suppose $s(g^{-1}_1, g_2^{-1}) = s(g^{-1}_1, g_2)$. Then, from (ii) and the first equality of (iii), we get $s(g_2, g_1) = s(g^{-1}_1, g_2)$. Hence, either 
$$\|g_2g_1 g_1^{-1} g_2\|_d \in 2\left(\|g_1\|_d+ \|g_2\|_d\right) + I(3 \gamma)$$  or $$\|g_2g_1 g_1^{-1} g_2\|_d \in 2\big|\|g_1\|_d-\|g_2\|_d\big| + I(3 \gamma).$$ On the other hand, $\|g_2g_1 g_1^{-1} g_2\|_d = \|g_2^2\|_d$, which is in $2\|g_2\|_d + I(\gamma)$. We get a contradiction with the fact that $g_1$ and $g_2$ are not in $N( 4\gamma)$.

We now prove (iv).
Without loss of generality, assume $\|g_1\|_d \leq \|g_2\|_d \leq \|g_3\|_d$. Using (iii) to replace $g_3$ with $g_3^{-1}$ if necessary, we can further assume that  $s(g_2, g_3) =1$. We need to show that $s(g_1, g_2)= s(g_1,g_3)$. Suppose to the contrary. Then, from (iii), we get $s(g_1, g_2) = s(g^{-1}_1,g_3)$.  Using (iii) to replacing $g_1$ with $g_1^{-1}$ if necessary, we can assume that 
$s(g_1, g_2) = s(g^{-1}_1,g_3)=1.$ Using (ii), we get $s(g_2,g_1)=1$. Hence,  either 
$$\|g_2g_1 g_1^{-1}g_3\|_d \in 2 \|g_1\|_d+\|g_2\|_d+ \|g_3\|_d+ I(3\gamma)$$ or  $$\|g_2g_1 g_1^{-1}g_3\|_d \in \|g_3\|_d -\|g_2\|_d +I(3\gamma).$$ On the other hand,  $\|g_2g_1 g_1^{-1}g_3\|_d = \|g_2g_3\|_d$ is in $\|g_2\|_d+\|g_3\|_d+I(\gamma)$. Hence, we get a contradiction to the fact that $g_1$, $g_2$, and $g_3$ are not in $N(4\gamma)$.

Finally, we prove (v). Using (iv), it suffices to show  $s(g_1g_2, g_2) = s(g_2g_1, g_2)=1$. We will only show the former, as the proof for the latter is similar. Suppose to the contrary that $s(g_1g_2, g_2)=-1$. Then $\|g_1g^{2}_2\|_d$ is in $\big|\|g_1g_2\|_d -\|g_2\|_d\big|+ I(\gamma)$, which is a subset of $\|g_1\|_d+ I(2\gamma)$. On the other hand, $\|g_1g^{2}_2\|_d$ is also in $\big|\|g_1\|_d -\|g^2_2\|_d\big|+ I(\gamma)$ which is a subset of $ 2\|g_2\|_d - \|g_1\|_d + I(2\gamma). $ Hence, we get a contradiction with the assumption that $g_1$ and $g_2$ are not in $N(4\gamma)$.
\end{proof}

  The notion of relative sign corrects the ambiguity in calculating distance, as can be seen in the next result.

\begin{lemma} \label{lem: Estimationusinsignchange}
Suppose $d$ is $\gamma$-monotone  $\gamma$-linear, and $g_1$ and $g_2 $ are in $N(\rho/16- \gamma)$ with $\|g_1\|_d \leq \|g_2\|_d$. Then we have the following
\begin{enumerate}[\rm (i)]
    \item Both $\|g_1g_2\|_d$ and $\|g_2g_1\|_d$ are in  $ s(g_1, g_2)\|g_1\|_d + \|g_2\|_d  + I(5\gamma)$.
    \item  If $g_0$ is in $N(\rho/4) \setminus  N(4 \gamma)$, then both $ s(g_0, g_{1}g_{2}) \|g_{1}g_{2}\|_d$ and $ s(g_0, g_{2}g_{1}) \|g_{2}g_{1}\|_d$ are in 
    $$s(g_0, g_{1}) \|g_{1}\|_d +  s(g_0, g_2) \|g_{2}\|_d + I(25\gamma).  $$
\end{enumerate}
\end{lemma}

\begin{proof}
We first prove (i). When $g_1, g_2 \notin N(4\gamma)$, the statement for $\|g_1g_2\|_d$ is immediate from the definition of the relative sign, and the statement for $\|g_2g_1\|_d$ is a consequence of Proposition~\ref{prop: propertoesofsign}(ii). Now suppose $\|g_1\|_d< 4 \gamma$. From the $\gamma$-linearity, we have  $$ \|g_2\|_d-\|g_1\|_d - \gamma < \|g_1g_2\|_d  <\|g_1\|_d+ \|g_2\|_d + \gamma. $$  
We deal with the case where $\|g_2\|_d< 4 \gamma$ similarly.

We now prove (ii). Fix $g_0$ in $N(\rho/4-\gamma)\setminus N(4\gamma)$. We will consider two cases, when $g_1$ is not in $N(4\gamma)$ and when $g_1$ is in $N(4\gamma)$. 
Suppose we are in the first case, that is $g_1 \notin N(4\gamma)$. As $\|g_1\|_d \leq \|g_2\|_d$, we also have $g_2 \notin N(4\gamma)$.
If both $g_1g_2$ and $g_2g_1$ are not in $N(4\gamma)$,  then the desired conclusion is a consequence of (i) and Proposition~\ref{prop: propertoesofsign}(iv, v). Within the first case, it remains to deal with the situations where $g_1g_2$ is in $N(4\gamma)$ or $g_2g_1$ is in $N(4\gamma)$. 

Since these two situations are similar, we may assume $g_1g_2$ is in $N(4\gamma)$. From (i), we have $s(g_1, g_2)=-1$ and $\|g_2\|_d - \|g_1\|_d$ is at most $5\gamma$. Therefore, $\|g_2g_1\|_d$ is in $I(6\gamma)$. By Proposition~\ref{prop: propertoesofsign}(iv), we have $s(g_0, g_1) = -s(g_0, g_2)$, and so
$$s(g_0, g_{1}) \|g_{1}\|_d +  s(g_0, g_2) \|g_{2}\|_d \in I(6\gamma).$$ Since both $ s(g_0, g_{1}g_{2}) \|g_{1}g_{2}\|_d$ and $ s(g_0, g_{2}g_{1}) \|g_{1}g_{2}\|_d$ are in $I(6\gamma)$,  they are both in
$s(g_0, g_{1}) \|g_{1}\|_d +  s(g_0, g_2) \|g_{2}\|_d + I(12\gamma)$ giving us the desired conclusion.

Continuing from the previous paragraph, we consider the second case when $g_1$ is in $N(4\gamma)$. If $g_2$ is in $N(16\gamma)$, then both $\|g_{1}g_{2}\|_d$ and  $ \|g_{2}g_{1}\|_d$ are in $I(25\gamma)$ by (i), and  the desired conclusion follows.  Now suppose $g_2$ is not in $N(16\gamma)$. Then from (i) and the fact that $g_1\in N(4\gamma)$, we get $g_1g_2$ and $g_2g_1$ are both not in $N(4\gamma)$. Note that $s(g_1g_2, g^{-1}_2)= -1$, because otherwise we get $$\|g_1\|_d \geq \|g_1g_2\|_d+\| g^{-1}_2\|_d -5 \gamma > 4\gamma.$$ A similar argument gives $s(g_2^{-1}, g_2g_1)=-1$. Hence, 
$s(g_1g_2, g_2)= s(g_2g_1, g_2)=1.$ By Proposition~\ref{prop: propertoesofsign}(v), we get $$s(g_0, g_2)= s(g_0, g_1g_2) = s(g_0, g_2g_1).$$ From (i), $\|g_1g_2\|_d$ and $\|g_2g_1\|
_d$ are both in $\|g_2\|_d+I(9\gamma)$. On the other hand, as $s(g_0, g_{1})=0$, we have
$s(g_0, g_{1}) \|g_{1}\|_d +  s(g_0, g_2) \|g_{2}\|_d = s(g_0, g_2) \|g_{2}\|_d$. The desired conclusion follows. 
\end{proof}

The next corollary will be important in the subsequent development.

\begin{corollary}{\label{cor: welldefinedoftotalweight}}
Suppose $d$ is $\gamma$-linear and $\gamma$-monotone,  $g_0$ and $g_0'$ are elements in $N(\rho/4-\gamma)\setminus N(4\gamma)$, and $(g_1, \ldots, g_n)$ is a sequence with $g_i \in N(\rho/4-\gamma)\setminus N(4\gamma)$ for $i \in \{1, \ldots, n\}$. Then
$$   \left|\sum^n_{i=1} s(g_0, g_i) \|g_i\|_d \right|  =  \left| \sum^n_{i=1} s(g'_0, g_i) \|g_i\|_d \right|.  $$
\end{corollary}

\begin{proof}
As $s(g_0, g_i) = s(g_0', g_i)=0$ whenever $\|g_i\|_d< 4\gamma$, we can reduce to the case where $\min_{1\leq i\leq n} \|g_i\|_d \geq 4 \gamma$. Using Proposition~\ref{prop: propertoesofsign}(iii) to replace $g_0$ with $g_0^{-1}$ if necessary, we can assume that $s(g_0, g_1) = s(g'_0, g_1)$. Then by Proposition~\ref{prop: propertoesofsign}(iii), $s(g_0, g_i) = s(g'_0, g_i)$ for all $i \in \{1, \ldots, n\}$. This gives us the desired conclusion.
\end{proof}

The following auxiliary lemma allows us to choose $g_0$ as in Corollary~\ref{cor: welldefinedoftotalweight}.

\begin{lemma} \label{lem: nonemptytodefine}
The set $N(\rho/4-\gamma)\setminus N(4\gamma)$ is not empty.
\end{lemma}

\begin{proof}
It suffices to show that $\mu_G(N(4\gamma))< \mu_G(N(\rho/4-\gamma)$. 
Since $\id$ is in $N(4\gamma)$, $N(4\gamma)$ is a nonempty open set and has $\mu_G(N(4\gamma))>0$. Therefore,  $N^2(4\gamma)$ and $N^4(4\gamma)$ are also open. By $\gamma$-linearity, we have 
$$N^2(4\gamma) \subseteq N_{9\gamma}\quad \text{and}\quad N^4(4\gamma) \subseteq N_{19\gamma}.$$ As $19\gamma<\rho$, we have $N^4(4\gamma) \neq G$. Using Corollary~\ref{cor: when a+b>G}, we get
$$\mu_G( N^2(4\gamma))\leq  2/3 \quad\text{and}\quad  \mu_G( N(4\gamma))<  1/3.$$ Hence, by Kemperman's inequality $\mu_G(N(4\gamma))< \mu_G(N^2(4\gamma)) \leq \mu_G(N(\rho/4-\gamma))$, which is the desired conclusion.
\end{proof}

Suppose $(g_1, \ldots, g_n)$ is a sequence of elements in $N(\rho/4-\gamma)\setminus N(4\gamma)$. We set
$$t(g_1, \ldots, g_n) = \left|\sum^n_{i=1} s(g_0, g_i) \|g_i\|_d \right|  $$
with $g_0$ is an arbitrary element in $N(\rho/4-\gamma)\setminus N(4\gamma)$, and call this the {\bf total weight} associated to $(g_1, \ldots, g_n)$. This is well-defined by Corollary~\ref{cor: welldefinedoftotalweight} and Lemma~\ref{lem: nonemptytodefine}.

\subsection{Almost linear pseudometrics: group homomorphisms onto tori}\label{sec: 7.3}
In this section, we will use the relative sign function and the total weight function defined in Section~\ref{sec: 7.2} to define a universally measurable multivalued group homomorphism onto $\TT$. We will then use a number or results in descriptive set theory and geometry to refine this into a continuous group homomorphism.

We keep the setting of Section~\ref{sec: 7.2}, and assume further that $G$ is compact. Let $s$ and $t$ be the relative sign function and the total weight function defined earlier. Set $\lambda= \rho/36$, and $N[\lambda] = \{g \in G : \|g\|_d \leq \lambda\}$. The set $N[\lambda]$ is compact, and hence measurable. Moreover, Lemma~\ref{lem: Estimationusinsignchange} is applicable when $g_0$ is an arbitrary element in $N(\rho/4-\gamma)\setminus N(4\gamma)$, and $g_1$ are $g_2$ are in $N[\lambda]$.

A sequence $(g_1, \ldots, g_n)$ of elements in $G$ is a {\bf $\lambda$-sequence} if $g_i$ is in $N[\lambda]$ for all $i \in \{1, \ldots, n\}$. We are interested in expressing an arbitrary $g$ of $G$ as a product of a $\lambda$-sequence where all components are ``in the same direction''. The following notion captures that idea. A $\lambda$-sequence $(g_1, \ldots, g_n)$  is {\bf irreducible} if for all $2 \leq j \leq 4$, we have $$ g_{i+1} \cdots g_{i+j} \notin N(\lambda).$$    
 A {\bf concatenation} of a $\lambda$-sequence $(g_1, \ldots, g_n)$ is a $\lambda$-sequence $(h_1, \ldots, h_m)$ such that there are  $0 =k_0< k_1< \cdots< k_m =n$ with 
 $$h_i = g_{k_{i-1}+1} \cdots g_{k_i} \text{ for } i \in \{1, \ldots, m\}.$$ The next lemma allows us to reduce an arbitrary sequence to irreducible $\lambda$-sequences via concatenation.

\begin{lemma} \label{lem: concatenationcomparison}
Suppose $d$ is $\gamma$-linear and $\gamma$-monotone, and $(g_1, \ldots, g_n)$ is a $\lambda$-sequence. Then $(g_1, \ldots, g_n)$ has an irreducible  concatenation $(g'_1, \ldots, g'_m)$ with
$$ t(g'_1, \ldots, g'_m) \in t(g_1, \ldots, g_n) + I(25(n-m)\gamma). $$
\end{lemma}

\begin{proof}
The statement is immediate when $n=1$. Using induction, suppose we have proven the statement for all smaller values of $n$.
If $(g_1, \ldots, g_n)$ is irreducible, we are done. Consider the case where $g_{i+1}g_{i+2}$ is in $N(\lambda)$ for some $0 \leq i \leq n-2$. Fix $g_0$ in $N(\lambda/4-\gamma)\setminus N(4\gamma)$. Using Lemma~\ref{lem: Estimationusinsignchange}(ii)
$$ s(g_0, g_{i+1}g_{i+2}) \|g_{i+1}g_{i+2}\|_d \in s(g_0, g_{i+1}) \|g_{i+1}\|_d +  s(g_0, g_{i+2}) \|g_{i+2}\|_d + I(25\gamma).  $$
From here, we get the desired conclusion. The cases where either $g_{i+1}g_{i+2}g_{i+3}$ for some $0 \leq i \leq n-3$ or $g_{i+1}g_{i+2}g_{i+3}g_{i+4}$ is in $N(\lambda)$ for some $0 \leq i \leq n-4$ can be dealt with similarly.
\end{proof}

The following lemma makes the earlier intuition of ``in the same direction'' precise:  

\begin{lemma} \label{lem: almostmonotonicityofirredseq}
Suppose $d$ is $\gamma$-linear and $\gamma$-monotone, $g_0$ is in $ N(\rho/4-\gamma)\setminus N(4\gamma)$, and $(g_1, \ldots, g_n)$ is an irreducible $\lambda$-sequence. Then for all $i$, $i'$, $j$, and $j'$ such that  $2\leq j, j' \leq 4$, $0 \leq i \leq n-j$, and $0 \leq i' \leq n-j'$, we have
$$ s( g_0, g_{i+1}\cdots g_{i+j}) = s(g_0,  g_{i'+1}\cdots g_{i'+j'}).  $$
\end{lemma}
\begin{proof}
It suffices to show for fixed $i, j$ with $0 \leq i \leq n-j-1$ and  $2 \leq j \leq 3$ that $$s( g_0, g_{i+1}\cdots g_{i+j}) = s(g_0,  g_{i+1}\cdots g_{i+j+1}).$$ Note that both $g_{i+1}\cdots g_{i+j}$ and $g_{i+1}\cdots g_{i+j+1}$ are in $N(\rho/4-\gamma)\setminus N(4\gamma)$. Hence, applying Proposition~\ref{prop: propertoesofsign}(iv), we reduce the problem to showing  $$s( g^{-1}_{i+j}\cdots g_{i+1}^{-1}, g_{i+1}\cdots g_{i+j+1)}) =-1  .$$ This is the case because otherwise, $\|g_{i+j+1}\|_d \geq 2 \lambda - \gamma > \lambda$, a contradiction.    
\end{proof}

We now get a lower bound for the total distance of an irreducible $\lambda$-sequence:

\begin{corollary} \label{Cor: Totallengthofirreduciblesequence}
Suppose $d$ is $\gamma$-linear and $\gamma$-monotone, and $(g_1, \ldots, g_n)$ is an irreducible $\lambda$-sequence. Then
$$   t(g_1, \ldots, g_n) > n\lambda/4. $$
\end{corollary}
\begin{proof}
If $n=2k$, let $h_i = g_{2i-1}g_{2i}$ for $i \in \{1, \ldots, k\}$. If $n=2k+1$, let $h_i = g_{2i-1}g_{2i}$ for $i \in \{1, \ldots, k-1\}$,   and $h_k = g_{2n-1}g_{2n}g_{2n+1}$. From Lemma~\ref{lem: Estimationusinsignchange},  we have 
\begin{equation}\label{eq: sequence estimate}
   t(h_1, \ldots, h_k) \in t(g_1, \ldots, g_n) + I(25(n-k)\gamma). 
\end{equation}
As $(g_1, \ldots, g_n)$ is irreducible, $h_i$ is in $N(3\lambda)\setminus N(\lambda)$ for $i \in \{1, \ldots, k\}$.
By Lemma~\ref{lem: almostmonotonicityofirredseq}, we get $s(g_0,h_i) = s(g_0, h_j)$ for all $i$ and $j$ in $i \in \{1, \ldots, k\}$. Thus, by the definition of the total weight again, 
$
t(h_1, \ldots, h_k)> n\lambda/3. 
$ Combining with the assumption on $\lambda$ and \eqref{eq: sequence estimate}, we get $t(g_1, \ldots, g_n)>n\lambda/3-11n\gamma > n\lambda/4 $. \end{proof}

  When $(g_1.\dots,g_n)$ is an irreducible $\lambda$-sequence,  $g_1\cdots g_m$ is intuitively closer to $g_0$ than $g_1\cdots g_{m+k}$ for some positive $k$. However, as $G$ is compact, the sequence may ``return back'' to $\id$ when $n$ is large. The next proposition provides a lower bound estimate on such $n$.

\begin{lemma}[Monitor lemma]\label{prop: lower bound on n}
Suppose $d$ is $\gamma$-linear and $\gamma$-monotone, and $(g_1, \ldots, g_n)$ is an irreducible $\lambda$-sequence with $g_1 \cdots g_n =\id$. Then $n \geq 1/ \mu_G(N(4 \lambda)).$
\end{lemma}

\begin{proof}
Let $m>0$. For convenience, when $m>n$ we write $g_m$ to denote the element $g_i$ with $i\leq n$ and $i\equiv m\pmod n$.  Define
\[
N^{(m)}(4 \lambda) = \{ g\in G \mid d( g,  g_1\cdots g_{m})< 4\lambda\}.
\]
Note that we have $N^{(m)}(4\lambda) = N^{(m')}(4\lambda) $ when $m \equiv m' \pmod{n}$. By invariance of $d$ and $\mu_G$, clearly $\mu_G(N^{(m)}(4 \lambda)) = \mu_G( N(4\lambda))$ for all $m$. We also write $N^{(0)}(4\lambda)=N(4\lambda)$. We will show that 
\[
G = \bigcup_{m\in\ZZ} N^{(m)}(4\lambda) =\bigcup_{m =0}^{n-1} N^{(m)}(4\lambda),
\]
which yields the desired conclusion. 

As $g_1 \cdots g_n =\id$, we have $\id$ is in $N^{(0)}(2\lambda)$, and hence in $\bigcup_{m \in\ZZ} N^{(m)}(4\lambda)$. As every element in $G$ can be written as a product of finitely many elements in $N(\lambda)$, it suffices to show for every $g \in \bigcup_{m \in \ZZ} N^{(m)}(4\lambda)$ and $g' = gh$ with $h \in N(\lambda)$ that $g'$ is in $\bigcup_{m \in \ZZ} N^{(m)}(4\lambda)$. The desired conclusion then follows from the induction on the number of translations in $N(\lambda)$.

Fix $m$ which minimizes $d(g, g_1 \ldots g_m)$. We claim that $d(g, g_1 \ldots g_m)< 2\lambda+\gamma$. This claim gives us the desired conclusion because we then have $d(g', g_1 \ldots g_m) < 3\lambda+ 2\gamma < 4\lambda$ by the $\gamma$-linearity of $d$.   

We now prove the claim that  $d(g, g_1 \ldots g_m)< 2\lambda+\gamma$. Suppose to the contrary that $d(g, g_1 \ldots g_m)\geq 2 \lambda+\gamma$. Let $u=(g_1\cdots g_m)^{-1}g$. Now by Lemma~\ref{lem: almostmonotonicityofirredseq} we have either $s(u, g_{m+1}g_{m+2})=1$, or $s(u, g_m^{-1}g_{m-1}^{-1})=1$. Suppose it is the former, since the latter case can be proved similarly. Then $s(u, g_{m+2}^{-1}g_{m+1}^{-1}) =-1$. Note that $g = g_1\cdots g_m u = (g_1 \cdots g_{m+2}) g^{-1}_{m+2}g^{-1}_{m+1}u$. By the definition of $u$, and the linearity of $d$, we have $\|u\|_d\geq 2\lambda+\gamma>\|g_{m+1}g_{m+2}\|_d$, therefore by the irreducibility we have
\begin{align*}
d(g, g_1\cdots g_{m+2}) &= \|g^{-1}_{m+2}g^{-1}_{m+1}u\|_d\\
&< \|u\|_d-\|g^{-1}_{m+2}g^{-1}_{m+1}\|_d +\gamma< \|u\|_d-\lambda +\gamma< \|u\|_d.
\end{align*}
This contradicts our choice of $m$ having $d(g, g_1, \ldots, g_m)$ minimized.
\end{proof}

In the later proofs of this section, we will fix an irreducible $\lambda$ sequence $g_1\cdots g_n=\id$ to serve as ``monitors''. As each element of $G$ will be captured by one of the monitors, this will help us to bound the error terms in the final almost homomorphism we obtained from the pseudometric. 

 Suppose $d$ is $\gamma$-linear and $\gamma$-monotone, and $N(\rho/4-\gamma)\setminus N(4\gamma) \neq \varnothing$. Define the {\bf returning weight} of $d$ to be
$$ \omega = \inf\{ t(g_1, \ldots, g_n) : ( g_1, \ldots, g_n) \text{ is an irreducible } \lambda\text{-sequence with } g_1\cdots g_n=\id \}.  $$
The following corollary translate Lemma~\ref{prop: lower bound on n} to a bound on such $\omega$:

 \begin{corollary}\label{cor: a_lambdanew}
 Suppose $d$ is $\gamma$-linear and $\gamma$-monotone, and $\omega$ is the returning weight of $d$. Then we have the following:
\begin{enumerate}[\rm (i)]
    \item $\lambda/ 4\mu_G(N(4 \lambda)) \leq \omega \leq 4\lambda/ \mu_G(N(\lambda)).$
    \item There is  an irreducible $\lambda$-sequence $(g_1, \ldots, g_n)$ such that $\omega = t(g_1, \ldots, g_n)$ and $ 1/ \mu_G(N(4 \lambda)) \leq n \leq 4/ \mu_G(N(\lambda))$.
\end{enumerate}
 \end{corollary}

\begin{proof}
Note that each irreducible $\lambda$-sequence $(g_1, \ldots, g_n)$  has $n \geq 1/ \mu_G(N(4\lambda))$ by using Lemma~\ref{prop: lower bound on n}.
Hence,  by Corollary~\ref{Cor: Totallengthofirreduciblesequence}, we get 
$
\omega \geq \lambda/ 4\mu_G(N(4 \lambda)) .
$
On the other hand, by Corollary~\ref{cor: when a+b>G}, $G= (N(\lambda))^k$ for all $k> 1/\mu_G(N(\lambda))$. Hence, with Lemma~\ref{lem: concatenationcomparison}, there is an irreducible $\lambda$-sequence $(g_1, \ldots, g_n)$ with $g_1\cdots g_n =\id$ and
$n \leq  4/ \mu_G(N(\lambda)) $. From the definition of $t$, we get $\omega \leq  4\lambda/ \mu_G(N(\lambda))$.

Now if an irreducible $\lambda$-sequence $(g_1, \ldots, g_n)$ has $n > 4/ \mu_G(N(\lambda))$, then by (i) and Corollary~\ref{Cor: Totallengthofirreduciblesequence}, 
\[t(g_1, \ldots, g_n) > \frac{4\lambda}{\mu_G(N(\lambda))}\geq \omega,
\]
a contradiction. Therefore, we have
\begin{align*}
\omega  = \inf\{ t(g_1, \ldots, g_n) : &\,( g_1, \ldots g_n) \text{ is an irreducible } \lambda\text{-sequence with }\\
&\, n \leq  4/ \mu_G(N(\lambda))\text{ and } g_1\cdots g_n=\id \}.
\end{align*}
For fixed $n$ the set of irreducible $\lambda$-sequence of length $n$ is closed under taking limit. Hence, we obtain desired $(g_1, \ldots, g_n) $ using the Bozalno--Wierstrass Theorem.
\end{proof}

  The next lemma allows us to convert between $\mu_G(N(\lambda))$ and $\mu_G(N(4\lambda))$:
\begin{lemma}\label{lem: N lambda}
Suppose $d$ is $\gamma$-linear and $\gamma$-monotone. Then $$\mu_G(N(4\lambda))\leq 16\mu_G(N(\lambda)).$$ 
\end{lemma}
\begin{proof}
Fix $h\in N(\lambda)\setminus N(\lambda/2-\gamma)$. Such $h$ exists since
by $\gamma$-monotonicity we have $N^2(\lambda/2-\gamma)\subseteq N(\lambda)$, and by Kemperman's inequality, $\mu_G(N(\lambda)>2\mu_G( N(\lambda/2-\gamma))$. Let $g$ be an arbitrary element in $N(4\lambda)$, and  assume first $s(g,h)=1$. Let $k\geq0$ be an integer, and define $g_k=g(h^{-1})^k$. Then by Proposition~\ref{prop: propertoesofsign} and Lemma~\ref{lem: Estimationusinsignchange}, 
\[
\|g_k\|_d\in \|g\|_d-k\|h\|_d+I(5k\gamma) \text{ for } k <\|g\|_d/\|h\|_d. 
\]
Hence,  there is $k< 8$ such that $g_k\in N(\lambda)$.  When $s(g,h)=-1$, one can similarly construct $g'_k$ as $gh^k$, and find $k< 8$ such that $g'_k\in N(\lambda)$. Therefore 
\[
N(4\lambda)\subseteq \Big(\bigcup_{i=0}^7 N(\lambda)h^{i}\Big)\cup \Big(\bigcup_{j=0}^7 N(\lambda){h^{-j}}\Big).
\]
Thus, $\mu_G(N(4\lambda))\leq 16\mu_G(N(\lambda))$.
\end{proof}

The following proposition implicitly establish that $t$ defines an approximate multivalued group homomorphism from $G$ to $\RR/\omega\ZZ$.
\begin{proposition} \label{lem: estimationoftotallength}
Suppose $d$ is $\gamma$-linear and $\gamma$-monotone,  $\omega$ is the returning weight of $d$, and $(g_1, \ldots, g_n)$ is a $\lambda$-sequence with 
$g_1 \ldots g_n =\id$ and $n \leq  4/ \mu_G(N(\lambda)) $. 
Then
$$t(g_1, \ldots, g_n) \in \omega\ZZ + I(\omega/400).$$
\end{proposition}

 \begin{proof}
 Let $g_0$ be in $N(\rho/4-\gamma)\setminus N(4\gamma) $. Using Proposition~\ref{prop: propertoesofsign}(iii) to replace $g_0$ with $g_0^{-1}$ if necessary, we can assume that
 $$t(g_1,\dots,g_n) = \sum_{i=1}^ns(g_0, g_i) \|g_i\|_d.$$
As $n\leq 4/ \mu_G(N(\lambda))$,  we have
 $
 t(g_1, \ldots,\allowbreak g_n)\leq  4\lambda/ \mu_G(N(\lambda)).
 $
From Corollary~\ref{cor: a_lambdanew}(i), we have $\lambda \leq 4\omega\mu_G(N(4\lambda))$. Hence, 
\begin{equation}\label{eq: t(g_1,..,g_n)}
t(g_1, \ldots, g_n) < \frac{16\omega \mu_G(N(4\lambda))}{ \mu_G(N(\lambda))}.
\end{equation}
Using Corollary~\ref{cor: a_lambdanew} again, we obtain an irreducible $\lambda$-sequence $(h_1, \ldots, h_m)$   such that $t(h_1, \ldots, h_m)=\omega$ and $ 1/ \mu_G(N(4 \lambda)) \leq m \leq 4/ \mu_G(N(\lambda))$. Using Proposition~\ref{prop: propertoesofsign}(iii) to replace $(h_1, \ldots, h_m)$ with $(h_m^{-1}, \ldots, h_1^{-1})$ if necessary, we can assume that 
$$t(h_1, \ldots, h_m)=-\sum_{i=1}^ns(g_0, h_i) \|h_i\|_d.$$

We now define a sequence $(g'_1,\ldots,g_{n'}')$ such that
\begin{enumerate}
    \item $n'=n+km$ for some integer $k\geq0$.
    \item $g'_i=g_i$ for $1\leq i\leq n$. 
    \item For $i\geq n+1$, $g'_i=h_{j}$ with $j\equiv i-n\pmod m$.
\end{enumerate}
From the definition of the total weight, for $k< t(g_1, \ldots, g_n)/\omega$, we have
\[
t(g'_1,\ldots,g'_{n'})=t(g_1,\ldots,g_n)-k\omega.
\]
We choose an integer $k< t(g_1, \ldots, g_n)/\omega+1$ such that $|t(g'_1,\ldots,g'_{n'})|\leq\omega/2$. Then by \eqref{eq: t(g_1,..,g_n)}, and the trivial bound $\mu_G(N(\lambda))< \mu_G(N(4\lambda))$, we have
\begin{equation*}\label{eq: bound for n'}
n'<n+km< \frac{4}{\mu_G(N(\lambda))}+\left(\frac{ 16\mu_G(N(4\lambda))}{\mu_G(N(\lambda))}+1\right) \frac{4}{\mu_G(N(\lambda))} \leq \frac{ 72\mu_G(N(4\lambda))}{\mu_G^2(N(\lambda))}. 
\end{equation*}

  Note that $(g'_1, \ldots, g'_{n'})$ is a $\lambda$-sequence with 
$g'_1 \ldots g'_{n'} =\id$. We assume further that $0\leq t(g'_1, \ldots, g'_{n'}) <  \omega/2$ as the other case can be dealt with similarly.
 Obtain an irreducible concatenation  $(h'_1, \ldots, h'_{m'})$ of $(g'_1, \ldots, g'_{n'})$. 
 From Lemma~\ref{lem: concatenationcomparison}, we get
 \begin{equation*}
 t(h'_1, \ldots, h'_{m'})< t(g'_1, \ldots, g'_{n'}) +25(n'-m')\gamma
 \leq \frac{\omega}{2}+ \frac{ 1800\mu_G(N(4\lambda))\gamma}{\mu_G^2(N(\lambda))}.
  \end{equation*}
  Using Corollary~\ref{cor: a_lambdanew}(i) and Lemma~\ref{lem: N lambda}, we have
  \[
  \frac{ 1800\mu_G(N(4\lambda))\gamma}{\mu_G^2(N(\lambda))} \leq \frac{ 1800\mu_G(N(4\lambda))\gamma}{\mu^2_G(N(4\lambda))/16^2}<  \frac{5\cdot 10^5\gamma}{N(4\lambda)} \leq  5\cdot 10^5\gamma \frac{4\omega}{\lambda}.
  \]
  As  $\gamma< 10^{-8}\rho$, and $\lambda =\rho/16-\gamma$, one can check that the lass expression is at most $\omega/400$. Hence, 
  $t(h'_1, \ldots, h'_{m'})< \omega$. From the definition of $\omega$, we must have $t(h'_1, \ldots, h'_{m'}) =0$. Thus by Lemma~\ref{lem: concatenationcomparison} again, 
\[
t(g'_1, \ldots, g'_n) \in I(25n'\gamma)\subseteq I(\omega/400), 
\]
which completes the proof. 
  \end{proof}

Recall that a {\bf Polish space} is a topological space which is  separable and completely metrizable. In particular, the underlying topological space of any connected compact Lie group is a Polish space. Let $X$ be a Polish space. A subset $B$ of $X$ is {\bf Borel} if $B$ can be formed from open subsets of $X$   (equivalently, closed subsets of $X$) through taking countable unions, taking countable intersections, and taking complement. A function $f: X\to Y$ between Polish space is Borel, if the inverse image of any Borel subset of $Y$ is Borel. A subset $A$ of $X$ is  {\bf analytic} if it is the continuous image of another Polish space $Y$. Below are some standard facts about these notions; see~\cite{Kechris} for details. 

\begin{fact} \label{fact: analyticsets}
Suppose $X, Y$ are Polish spaces, and $f: X\to Y$ is continuous. We have the following:
\begin{enumerate}[\rm (i)]
    \item Every Borel subset of $X$ is analytic.
    \item  Equipping $X\times Y$ with the product topology, the graph of a Borel function from $X$ to $Y$ is analytic.
    \item The collection of analytic subsets of $X$ is closed under taking countable unions, taking intersections and cartesian products.
    \item Images of analytic subsets in $X$ under $f$ is analytic.
\end{enumerate}

\end{fact}

Given $x\in \RR$, let $\|x\|_\TT$ be the distance of $x$ to the nearest element in $\ZZ$. We now obtain a consequence of Lemma~\ref{lem: estimationoftotallength}.
 
 \begin{corollary}[Analytic multivalued almost homomorphism] \label{cor: analyticmulitvalued}
 There is an analytic subset $\Gamma$ of $G \times \TT$ satisfying the following properties:
  \begin{enumerate}[\rm (i)]
     \item The projection of $\Gamma$ on $G$ is surjective. 
     \item $(\id, \mathrm{id}_{\RR/\omega\ZZ})$ is in $\Gamma$.
      \item If  $g_1, g_2\in G$ and $t_1, t_2, t_3\in \RR$ are such that  $(g_1, t_1/\omega+\ZZ)$, $(g_2, t_2/\omega+\ZZ)$, and $(g_1g_2, t_3/\omega+\ZZ)$ in $\Gamma$, then 
      $$\|(t_1+t_2 -t_3)/\omega\|_\TT< 1/400.$$
      \item There are $g_1, g_2\in G$ and $t_1, t_2 \in \RR$ with that  $(g_1, t_1/\omega+\ZZ), (g_2, t_2/\omega+\ZZ) \in \Gamma$ and $\|(t_1-t_2)/\omega\|_\TT> 1/3$.
  \end{enumerate}
 \end{corollary}
 \begin{proof}
   Let $\Gamma$ consist of $(g,t/\omega+\ZZ) \in G \times \TT$ with $g\in G$ and $t\in \RR$ such that there is  $n\leq 1/\mu_G(N(\lambda))+1$ and an irreducible $\lambda$-sequence $(g_1, \ldots, g_n)$ satisfying 
   \[
   g= g_1 \cdots g_n \quad\text{and}\quad t = t(g_1, \ldots, g_n).
   \]
   Note that the relative sign function $s: G\times G \to \RR$ is Borel, the set $N[\gamma]$ is compact, and the function $x\to \| x\|_d$ is continuous. Hence, by Fact~\ref{fact: analyticsets}(i,ii), the function $(g_1, \ldots, g_n) \mapsto t(g_1, \ldots, g_n)$ is Borel, and its graph is analytic.
 For each $n$, by Fact~\ref{fact: analyticsets}(iii)
  \begin{align*}
  \widetilde{\Gamma}_n:= \{ (g,t, g_1, \ldots, g_n) &\in G \times \RR \times G^n :   \\
  &\|g_i\|_d < \lambda \text{ for } 1 \leq i \leq n, g=g_1\cdots g_n, t=t(g_1, \ldots, g_n)\} 
  \end{align*}
is analytic. Let ${\Gamma}_n$ be the image of $\widetilde{\Gamma}_n$ under the continuous map
$$ (g, t, g_1, \ldots, g_n) \mapsto (g, t/\omega+\ZZ).$$ Then by Fact~\ref{fact: analyticsets}(iv), ${\Gamma}_n$ is analytic. Finally,  $\Gamma = \bigcup_{n< 1/\mu_G(N(\lambda))+1}\Gamma_n$ is analytic by Fact~\ref{fact: analyticsets}(iii). 

We now verify that $\Gamma$ satisfies the desired properties. It is easy to see that (i) and (ii) are immediately from the construction, and (iii) is a consequence of  Lemma~\ref{lem: estimationoftotallength}. We now prove (iv). Using Corollary~\ref{cor: a_lambdanew}, we obtain an irreducible $\lambda$-sequence $(g_1, \ldots, g_n)$ with $t(g_1, \ldots,  g_n) = \omega$ and $n < 4/\mu_G(N(\lambda))$.  Note that 
\[
|t(g_1, \ldots, g_{k+1})-t(g_1, \ldots, g_k)|\leq \lambda.
\]
Hence, there must be $k \in \{1, \ldots n\}$ such that $\omega/3<  t(g_1, \ldots, g_k)<2\omega/3 $. Set $t_1 =  0$ and $t_2 =  t(g_1, \ldots, g_k) $ for such $k$. It is then easy to see that $\|(t_1-t_2)/\omega\|_\TT> 1/3$.
 \end{proof}

  To construct a group homomorphism from $G$ to $\TT$, we will need three more facts. Recall the following measurable selection theorem from descriptive set theory; see~\cite[Theorem~6.9.3]{Measuretheory}.
 
\begin{fact}[Kuratowski and Ryll--Nardzewski measurable selection theorem] \label{KRN}
Let $(X, \mathscr{A})$ be a measurable space, $Y$ a complete separable metric space equipped with the usual Borel $\sigma$-algebra, and $F$ a function on  $X$ with values in the set of nonempty closed subsets
of $Y$. Suppose that for every open $U \subseteq Y$, we have
$$ \{ a \in X: F(a) \cap U \neq \varnothing\} \in \mathscr{A}.$$
Then $F$ has a selection $f: X \to Y$ which is measurable with respect to $\mathscr{A}$.
 \end{fact}

  A  {\bf Polish group} is a topological group whose underlying space is a Polish space. In particular, Lie groups are Polish groups. 
A subset  $A$ of a Polish space $X$ is {\bf universally measurable} if $A$ is measurable with respect to every complete probability measure on $X$ for which every Borel set is measurable. In particular, every analytic set is universally measurable; see ~\cite{Rosendal} for details. A map $f: X \to Y$ between Polish spaces is {\bf universally measurable} if inverse images of open sets are universally measurable.  We have the following recent result from descriptive set theory by~\cite{Rosendal}; in fact, we will only apply it to Lie groups so a special case which follows from an earlier result by Weil~\cite[page 50]{Weils} suffices.
 
 \begin{fact}[Rosendal] \label{automaticcontinuity}
Suppose $G$ and $H$ are Polish groups, $f: G \to H$ is a  universally measurable group homomorphism. Then $f$  is continuous.
\end{fact}

   Finally, we need the following theorem from geometry by Grove, Karcher, and Ruh~\cite{almosthomo} and independently by Kazhdan~\cite{almosthomo2}, that in a compact Lie groups an almost homomorphism is always close to a homomorphism uniformly. We remark that the result is not true for general compact topological groups, as a counterexample is given in~\cite{almosthomonottrueingeneral}.

\begin{fact}[Grove--Karcher--Ruh; Kazhdan]\label{fact:almosthomo}
Let $G,H$ be compact Lie groups. There is a constant $c$ only depending on $H$, such that for every real number $q$ in $[0,c]$, if $\pi: G\to H$ is a  $q$-almost homomorphism, then there is a  homomorphism $\chi:G\to H$ which is $1.36q$-close to $\pi$. Moreover, if $\pi$ is universally measurable, then $\chi$ is universally measurable. When $H=\TT$, we can take $c=\pi/6$. 
\end{fact}

   The next theorem is the main result in this subsection. It tells us from an almost linear pseudometric, one can construct a group homomorphism to $\TT$ or to $\RR$; this can also be seen as a stability theorem of Proposition~\ref{prop: strong linear}.

  \begin{theorem}\label{thm: homfrommeasurecompact}
Let $\lambda=\rho/36$, and $\gamma<10^{-6}\rho$. Suppose $d$ is $\gamma$-linear and $\gamma$-monotone. Then there is a continuous surjective group homomorphism $\chi: G \to \TT$ such that for all $g\in \ker(\chi)\cap N(\lambda)$, we have $\| g\|_d \in N(\lambda/2)$.
 \end{theorem}
 
 \begin{proof}
Let $\omega$ be the returning weight of $d$, and let $\Gamma$ be as in the proof of Corollary~\ref{cor: analyticmulitvalued}. Equip $G$ with the $\sigma$-algebra $\mathscr{A}$ of universally measurable sets. Then $\mathscr{A}$ in particular consists of analytic subsets of $G$.   Define $F$ to be the function from $G$ to the set of closed subsets of $\TT$ given by
\[
F(g) =\overline{\{ t/\omega+\ZZ : t\in \RR,  (g, t/\omega +\ZZ)\in \Gamma\}}. 
\]
If $U$ is an open subset of $\TT$, then $\{g \in G : F(g) \cap U \neq \varnothing\}$ is in $\mathscr{A}$ being the projection on $G$ of the analytic set $\{ (g, t/\omega+\ZZ) \in G \times \TT : (g, t/\omega+\ZZ) \in \Gamma \text{ and } t \in U \}$. Applying Fact~\ref{KRN}, we get a universally measurable $1/400$-almost homomorphism $\pi: G \to \TT$.  
Using Fact~\ref{fact:almosthomo}, we get a universal measurable group homomorphism $\chi: G \to \TT$ satisfying 
\[
\|\chi(g)-\pi(g)\|_\TT < 1.36/400 = 0.0068.
\]The group homomorphism $\chi$ is automatically continuous by Fact~\ref{automaticcontinuity}. Combining with Corollary~\ref{cor: analyticmulitvalued}(iv), we see that $\chi$ cannot be the trivial group homomorphism, so  $\chi$ is surjective.

Finally, for $g$ in $\ker (\chi)\cap (N(\lambda))$, we need to verify that $g$ is in $N(\lambda/2)$. Suppose to the contrary that $g \notin N(\lambda/2)$. 
Choose $n=\lfloor1/\mu_G(N(4\lambda)) \rfloor$, and $(g_1, \ldots, g_n)$ the $\lambda$-sequence such that $g_i =g$ for $i\in \{1, \ldots, n\}$.
By Proposition~\ref{prop: propertoesofsign}, $(g_1, \ldots, g_n)$ is irreducible.   Hence, by Lemma~\ref{prop: lower bound on n}, $t(g_1, \ldots, g_n)< \omega$. As $n \leq 1/\mu_G(N(\lambda)+1)$, by construction and Corollary~\ref{cor: analyticmulitvalued}(iii), we have
\[
\pi(g^n) \in  t(g_1,\ldots,g_n)/\omega+ I(1/400)+\ZZ = n\|g\|_d/\omega + I(1/400)+\ZZ. 
\]
Since $g^n\in\ker\chi$, we  have $\|\pi(g^n)\|_\TT<0.0068$, so $n\|g\|_d/\omega< (0.0068+1/400)$. 
By Corollary~\ref{cor: a_lambdanew}(i) and Lemma~\ref{lem: N lambda}, this implies
\[
\|g\|_d\leq \frac{(0.0068+1/400)\cdot 4\lambda}{\mu_G(N(\lambda))}\mu_G(N(4\lambda))<\frac{\lambda}{2},
\]
which is a contradiction. This completes our proof.
\end{proof}

\section{Geometry of minimal and nearly minimal expansions II}\label{sec: geometry II}

In this section, we study the shape of a nearly minimally expanding pair relative to a connected closed proper subgroup of the ambient Lie group such that the cosets of the subgroup intersect the nearly minimal expanding pair ``transversally in measure''. Section~\ref{subsec: toric exp} shows that in a compact connected Lie group, such a subgroup exists and can in fact be chosen to be a one-dimensional torus. In Section~\ref{subsec: linear fiber} and~\ref{subsec: almost linear fiber},  we obtain shape description for the minimally expanding pair and nearly minimally expanding pair, respectively. Using that we will construct linear and almost linear pseudometric as described in Section~\ref{section:pseudometric}. 

Throughout this section,  $G$ is a {\it connected unimodular} {\it Lie} group, and $H$ is a {\it connected unimodular} closed subgroup of $G$. In particular, the left Haar measures $\mu_G$ and $\mu_H$ are Haar measures. We let $A$, and $B$ be $\sigma$-compact subsets of $G$.  We will assume familiarity with the preliminary Section~\ref{sec: 3.1} on locally compact group and Haar measure.

We set $\mu_{G/H}$ and $\mu_{H \backslash G}$ to be the Radon measures on $G/H$ and $H \backslash G$ such that we have the quotient integral formulas (Fact~\ref{fact: QuotientIF} and Lemma~\ref{lem: mesurability}(vi)). We also remind the reader that we normalize the measure whenever a group under consideration is compact, and $\pi: G \to G/H$, $ \widetilde{\pi}: G \to H\backslash G$ are quotient maps. Hence, if $H$ is compact, then we have
$$ \mu_G(AH)=\mu_{G/H}(\pi A)\quad \text{and}\quad \mu_{G}(HB) =\mu_{H \backslash G}( \widetilde{\pi} B), $$
and if $\chi: H \to \RR$ is a continuous and surjective group homomorphism with compact kernel, then the pushforward of $\mu_H$  is the Lesbegue measure $\mu_\RR$.

\subsection{Toric  transversal intersection in measure}\label{subsec: toric exp}
In this section, we assume that $G$ is {\it compact}. We will consider a more general situation than what we need assuming 
\[
\mu_G(A)= \mu_G(B)=\kappa. \]
 We will prove that if $\kappa$ is sufficiently small measure, and when $\mu_G(AB)<M\kappa$ for some constant $M$, then there is a torus $H\subseteq G$ such that we are in the short fiber scenario (i.e., for every $x,y\in G$,
\begin{equation}\label{eq: short fiber assump <c}
    \min\{\mu_H(A\cap xH),\mu_H( Hy\cap B)\}<\lambda
\end{equation}
for some given constant $\lambda$). Assume \eqref{eq: short fiber assump <c} fails for every maximal tori $H$, which means we have a long fiber ``in every direction''. Then both $A$ and $B$ can be seen as a variant of \emph{Kakeya sets} in Lie groups; see \cite{Kakeya} for some properties of Kakeya sets in this setting.
In general, it is well-known that Kakeya sets can have arbitrarily small measure; but when $A,B$ has nearly minimally expansion, we will show in this section that both $A$ and $B$ must not be too small. Our result  in particular applies to approximate groups, as described in Section~\ref{sec: 6.2}.

We use the following lemma, which can be seen as a corollary of the quotient integral formula (Fact~\ref{fact: QuotientIF}).

\begin{lemma}\label{lem:int}
For every $b\in G$, the following identity holds
 $$ \mu_G\big(A (B \cap Hb)\big) = \int_G \mu_H\big(( A \cap aH )(B \cap Hb)\big) \d\mu_G(a). $$
\end{lemma}
\begin{proof}
Let $C = A (B \cap Hb)$.  From Fact~\ref{fact: QuotientIF}, one has $$\mu_G(C)= \int_G \mu_H( C \cap a b^{-1}H b)  \d\mu_G(a)=  \int_G \mu_H( C \cap a H b)  \d\mu_G(a).$$  Hence, it suffices to check that  $$C \cap a Hb = ( A \cap aH )(B \cap Hb) \quad \text{ for all } a, b \in G.$$
The backward inclusion is clear. Note that $aH Hb = aHb = ab (b^{-1}Hb)$ for all $a$ and $b$ in $G$. For all $a$, $a'$, and $b$ in $G$, we have 
we have  $a'Hb = aHb$ when $aH=a' H$ and $ aHb \cap a'Hb = \varnothing$ otherwise.  
An arbitrary element $c \in C$ is in $(A \cap a'H)(B \cap Hb)$ for some $a'\in A$. Hence, if $c$ is also in $aHb$, we must have $a'H =aH$. So we also get the forward inclusion.
\end{proof}

Suppose $r$ and $s$ are in $\RR$, the sets $ A_{(r,s]}$ and $ \pi A_{(r,s]}$ are given by 
$$  A_{(r,s]}  := \{a \in A: \mu_H(A \cap aH) \in (r,s]  \}  $$
and
$$ \pi A_{(r,s]}:=  \{aH \in G \slash H : \mu_H( A \cap aH) \in (r,s]   \}. $$ 
In particular, $\pi A_{(r,s]}$ is the image of $A_{(r,s]}$ under the map $\pi$. By Lemma~\ref{lem: mesurability}, $\pi A_{(r,s]}$ is $\mu_{G/H}$-measurable, and $A_{(r,s]}H = \pi^{-1}(\pi A_{(r,s]})$ and $A_{(r,s]} $ are $\mu_G$-measurable.
\begin{lemma}
\label{lem: A=AH toric exp prepare}
Suppose $\mu_G(AB)<M\kappa$, and $\lambda<1$ is a constant, and either there is $a\in A$ such that $\mu_H(A\cap aH)>\lambda$, or there is $b\in B$ such that $\mu_H(B\cap Hb)>\lambda$. Then,
\[
\min\bigg\{\frac{\mu_G(A)}{\mu_{G/H}(\pi A)}, \frac{\mu_G(B)}{\mu_{H\backslash G}(\widetilde{\pi} B)}\bigg\}\geq \frac{\lambda}{(M+2)^2}.
\]
\end{lemma}
\begin{proof}
Without loss of generality, suppose $\mu_H(B\cap Hb)>\lambda$ for a fixed $b\in B$. By the quotient integral formula, $\kappa=\mu_G(A)$ is at least

\[
\int_{\pi A_{(1/2,1]}} \mu_H(A\cap xH) \d\mu_{G/H}xH> \frac{1}{2}\mu_{G/H}(\pi A_{(1/2,1]}) = \frac{1}{2}\mu_{G}(A_{(1/2,1]}H) ,
\]
Hence,  $\mu_G(A_{(1/2,1]}H)<2\kappa$.  We now prove that $\mu_G(A_{(0,1/2]}H)< M\kappa/\lambda$.   Suppose that it is not the case. By Lemma~\ref{lem:int} we get
\begin{align*}
  \mu_G(A(B\cap Hb))&\geq \int_{A_{(0,1/2]}H} \mu_H((A\cap aH)(B\cap Hb))\d\mu_G(a)
\end{align*}
Observe that $\mu_H((A\cap aH)(B\cap Hb))>\lambda$ since $\mu_H(B\cap Hb)>\lambda$, we have
\[
\mu_G(AB)\geq\mu_G(A(B\cap Hb))\geq \lambda \mu_G(A_{(0,1/2]}H)>M\kappa,
\]
contradicting the assumption that $\dis_G(AB)<(M-1) \kappa$. Hence $\mu_G(AH)<(M+2)\kappa/\lambda$. This implies that there is $a\in A$ such that $\mu_H(A\cap aH)>\lambda/(M+2)$. Now we apply the same argument switching the role of $A$ and $B$, we get $\mu_{H\backslash G}(\widetilde{\pi} B)<(M+2)^2\mu_G(B)/\lambda$ which completes the proof.
\end{proof}

  Lemma~\ref{lem: A=AH toric exp prepare} leads us to consider the problem of obtaining lower bound for the measure of toric nonexpanders, which is of independent interest. 
 
\begin{definition}
We say $A$ is called a {\bf toric $K$-expander}, if there is a one-dimensional torus subgroup $H$ of $G$ such that $\mu_G(AH) \geq K \mu_G(A)$. 
\end{definition}

The next lemma shows that the notion of toric expanders is stable under right translations and taking unions. 

\begin{lemma}\label{lem:nonexpander}
Suppose $A$ is not a toric $K$-expander, $g_1,\dots,g_n$ are in $G$, and $A' =\bigcup_{j=1}^d Ag_j$. Then $A'$ is not a toric $(nK)$-expander. 
\end{lemma}
\begin{proof}
We need to verify for each $T$ that $\mu_G(A'T)< nK\mu_G(A')$. Note that 
\[
A'T= \Big(\bigcup_{i=1}^nAg_i\Big)T = \bigcup_{i=1}^nA(g_iT) = \bigcup_{i=1}^nA(g_iT g_i^{-1})g_i = \bigcup_{i=1}^nAT_ig_i,
\]
where $T_i$ is  the torus subgroup $g_iTg^{-1}_i$ of $G$. Hence, $$\mu_G(A'T)< nK\mu_G(A)\leq nK\mu_G(A')$$
as desired.
\end{proof}

We will need the following inequality:

\begin{fact}[Bhatia--Davis inequality] \label{fact: Bhatia-Davis}
Suppose $(X, \mathscr{A}, \mu)$ is a measure space, $\alpha$ and  $\beta$ are constants, and $f: X \to \RR^{>0}$ is a measurable function with \[\alpha\leq \inf_{x\in X} f(x) < \sup_{x \in X}f(x)\leq \beta\]
Then $(\mathbb E_x f^2(x)) - (\mathbb E_x f(x))^2 \leq (\beta- \mathbb E_x f(x))(\mathbb E_x f(x) -\alpha).$
\end{fact}

  The following lemma will help us to translate the set along some direction. 

\begin{lemma}\label{lem: trans along H}
Suppose $K$, $\alpha$, and $\beta$ are constant with $K>1$, $0\leq\alpha\leq\beta\leq1$, $\mu_G(AH) = K\mu_G(A)$, and 
\[
\alpha\leq\inf_{g\in A} \mu_H(A\cap gH)\leq \sup_{g\in A} \mu_H(A\cap gH)\leq \beta.
\]
Then for every number $\gamma \geq (\alpha+\beta-K\alpha\beta) \mu_G(A)$, there is $h\in H$ with $\mu_G(A\cap Ah)=\gamma$.
\end{lemma}
\begin{proof}
Let $\mu_G$, $\mu_H$ be normalized Haar measures of $G$ and $H$. Choose $h$ from $H$ uniformly at random. Note that
\begin{align*}
\mathbb E_{h\in H}\mu_G(A\cap Ah)&= \int_H \mu_G(A\cap Ah)\d\mu_H(h)\\
&=\int_H\int_G\mathbbm{1}_{A}(g)\mathbbm{1}_{A}(gh)\d\mu_G(g)\d\mu_H(h).
\end{align*}
Using the quotient integral formula (Fact~\ref{fact: QuotientIF}), the above equality is
\begin{align*}
\int_G\mathbbm{1}_{A}(g)\mu_H(A\cap gH) \d\mu_G(g)&=\int_{G/H} \mu_H^2(A\cap gH)\d\mu_{G/H}(gH)\\
                                                &= \mathbb{E}_{gH \in G/H} \big(\mu_H^2(A\cap gH)\big)\\
                                                &= \mu_{G/H}( \pi A)\mathbb{E}_{gH \in \pi A} \big(\mu_H^2(A\cap gH)\big)
\end{align*}
Note that $\mu_{G/H}(\pi A) = \mu_G(AH)= K\mu_G(A)$, and 
$$\mathbb{E}_{gH \in \pi A} \big(\mu_H(A\cap gH)\big)= 1/K.$$ Hence, applying the Bhatia–Davis inequality (Fact~\ref{fact: Bhatia-Davis}), we get
\[
\mathbb E_{h\in H}\mu_G(A\cap Ah)\leq (\alpha+\beta-K\alpha\beta)\mu_G(A).
\]
The desired conclusion follows from the continuity of $H \to \RR, h \mapsto \mu_G(A\cap Ah)$.
\end{proof}

  The following lemma says that for a toric nonexpander $A$ and a torus subgroup $H$ of $G$, one can slightly modify $A$ to get $A'$ such that most of $A'H$ can be covered by finitely many right translations of $A'$. 
\begin{lemma}\label{lem:translationinonedirection}
Suppose $K>1$ is a constant, $A$ is not a toric $K$-expander, and $H$ is a one-dimensional torus subgroup of $G$. Then for every $0<\varepsilon<1/2K$, there is a $\sigma$-compact  $A' \subseteq A$, integer $m =m(K, \varepsilon)>0$, and $h_1,\dots,h_m\in H$, satisfying \begin{enumerate}[\rm (i)]
    \item $A'$ is not a toric $2K$-expander
    \item $\mu_G(A')>(1-\varepsilon K)\mu_G(A)$
    \item $\mu_G\left(A'H\setminus \bigcup_{i=1}^m A'h_i\right)<\varepsilon\mu_G(AH)$.
\end{enumerate}
\end{lemma}
\begin{proof}
Let $\mu_G,\mu_H$ be normalized Haar measures on $G$ and $H$, and let $\mu_{G/H}$ be the invariant Radon measure induced by $\mu_G$ and $\mu_H$ on the homogeneous space $G/H$. Let
\begin{align*}
&\pi A_{(\varepsilon,1]}=\{g\in G/H\mid \varepsilon< \mu_H(A\cap gH)\leq1\},\\
&\pi A_{(0,1]}=\{g\in G/H\mid A\cap gH\neq\varnothing\}.
\end{align*}
Let $x = \mu_{G/H}(\pi A_{( 0, \varepsilon)})/\mu_{G/H}(\pi A_{(0,1]})$, then $ \mu_{G/H}(\pi A_{(\varepsilon, 1]} )=    (1-x)(\mu_{G/H}(\pi A_{(0, 1]})$. 
One has \[
\frac{\mu_{G/H}(\pi A_{(0,1]})}{K}<\mu_G(A)\leq (\varepsilon x+1-x)\mu_{G/H}(\pi A_{(0,1]})). 
\]
It follows that
\[
x=\frac{\mu_{G/H}(\pi A_{(0,\varepsilon]})}{\mu_{G/H}(\pi A_{(0,1]})}< \frac{K-1}{K(1-\varepsilon)}.
\]
Choose $\sigma$-compact $A' \subseteq A_{(\varepsilon,1]} $ such that $\mu_G(A_{(\varepsilon,1]} \setminus A')=0$. One has
\begin{align*}
    \mu_G(A')&\geq \mu_G(A)-\varepsilon \frac{K-1}{K(1-\varepsilon)}\mu_{G/H}(\pi A_{(0,1]})\\
    &\geq \left(1-\varepsilon\frac{K-1}{1-\varepsilon}\right)\mu_G(A)=\frac{1-\varepsilon K}{1-\varepsilon}\mu_G(A)\geq (1-\varepsilon K)\mu_G(A).
\end{align*}
Hence we have
\[
\mu_G(A')\geq(1-\varepsilon K)\mu_G(A)\geq \frac{1-\varepsilon K}{K}\mu_G(AT)>\frac{1}{2K}\mu_G(A'T),
\]
for every torus $T$ when $\varepsilon<1/2K$.

It remains to obtain $m=m(\varepsilon, K)$ and $h_1, \ldots, h_m$ such that (iii) is satisfied. Construct a sequence $(A'_n)$ of $\sigma$-compact subsets of $G$ with $A'_nH= A'H$ as follows. Let $A'_0 =A'$. Suppose $A'_n$ has been constructed. Set 
\[
\varepsilon_n = \inf_{gH \in \pi_{A} }\mu_H( A'_n \cap gH)  \text{ and } K_n =\mu_G(A'_nH)/\mu_G(A'_n).
\] Using Lemma~\ref{lem: trans along H}, obtain $h_n'\in H$ such that
\[
\mu_G(A'h_n'\setminus A')=\varepsilon_n(K_n'-1)\mu_G(A_n').
\]
Finally, let $A'_{n+1} = A'_n \cup A'_nh'_n$. Then, $A_{n+1}H= A_nH=A'H$. It is also easy to see that  $\varepsilon_n \geq \varepsilon$ for all $n$. Now, if $\mu_G(A_n)< (1-\varepsilon)\mu_G(A'H)$ for some $n$, then
\[
K_n=\frac{\mu_G(A'H)}{\mu_G(A_n)}\geq \frac{1}{1-\varepsilon},
\]
 and hence, 
\[
\mu_G(A'_{n+1}) \geq \frac{1-\varepsilon+\varepsilon^2}{1-\varepsilon}\mu_G(A'_n).
\]
As $(1-\varepsilon+\varepsilon^2)/(1-\varepsilon)>1$, this cannot be the case for all $n$. Let $N$ be the first $n$ such that $\mu_G(A'_N)> (1-\varepsilon)\mu_G(A'H)$. Then 
\[
2K\mu_G(A')> \mu_G(A'H) \geq \mu_G(A'_{N}) \geq \left(\frac{1-\varepsilon+\varepsilon^2}{1-\varepsilon}\right)^N\mu_G(A').
\]
This implies that 
\[
N \leq \frac{\log 2K}{\log(1-\varepsilon+\varepsilon^2) - \log(1-\varepsilon)}.
\]
Finally set $m=2^N$, and choose $h_1, \ldots, h_m$ such that $A'_N= \bigcup^m_{i=1}A'h_i$, we get the desired conclusion.
\end{proof}

   The next simple fact shows that we can find finitely many tori such that the product of them is $G$. 

\begin{fact}\label{fact:maxtori}
Let $G$ be compact. Then there is a constant $n$ depending only on the dimension of $G$ such that there are $n$  tori $H_1,\dots, H_n$ in $G$ with $H_1\cdots H_n=G$.
\end{fact}

  Suppose $A$ is not a toric expander. The next important ``cage'' lemma provides an inductive construction to construct a set $C$ from $A$, such that the size of $C$ is bounded from above, and any right translations of $A$ cannot ``escape'' $C$. 

\begin{lemma}[Cage lemma]\label{lem:controllinglargeset}
Suppose $A \subseteq G$ is not a toric $K$-expander, then  there is a $\sigma$-compact $C \subseteq A$ such that $\mu_G(C) = O_K(1)\mu_G(A)$ and for all $g \in G$
$$ \frac{\mu_G( C\cap Ag)}{ \mu_G(A)} > \frac12. $$
\end{lemma}

\begin{proof}
Using Fact~\ref{fact:maxtori}, we obtain one-dimensional torus subgroups $H_1, \ldots, H_n$ of $G$ such that  $G=H_1\cdots H_n$. For every constant $\varepsilon_0,\dots,\varepsilon_{n-1}$,
we construct a sequence $(A_i)^n_{i=0}$ of $\sigma$-compact subsets of $G$  and a sequence $(K_i)^n_{i=0}$ of constants satisfying the following conditions
\begin{enumerate}
\item $A_0=A$ and $K_0=K$.
\item $A_{i}$ is not a toric $K_i$-expander for $0\leq i\leq n$.
\item $A_{i}\subseteq A_{i+1}$ with $\mu_G(A_{i+1}) \leq K_{i} \mu_G(A_{i})$ for $0\leq i\leq n-1$.
\item $\mu_G(A_i\setminus A_{i+1}h_{i+1}) \leq \varepsilon_i \mu_G(A_i)$ for any $h_{i+1} \in H_{i+1}$ and  $0\leq i\leq n-1$.
\item  $K_{i+1} = K_i(\varepsilon_0, \ldots, \varepsilon_{i})$ for $0\leq i\leq n-1$.
\end{enumerate}
Suppose we have $A_0,\ldots, A_i$ and $K_0,\ldots, K_i$ satisfying all the conditions restricted down to $i$. We are going to construct $A_{i+1}$. 
By (2), $A_i$ is not a toric $K_i$-expander. 
Let $\delta>0$ be a parameter which  will be determined later. 
Using Lemma~\ref{lem:translationinonedirection}, we obtain a $\sigma$-compact $A'_i\subseteq A_i$, $m=m(K_i,\delta)$, and $h'_1,\dots,h'_m\in H_{i+1}$ such that \[\mu_G(A'_i)>(1-\delta K_i)\mu_G(A_i),
\]
$A'_i$ is not a toric $2K_i$-expander, and  
\[
\mu_G\bigg(A'_iH_{i+1}\setminus\bigcup_{j=1}^m A'_ih'_j\bigg)<\delta\mu_G(A'_iH_{i+1}).
\]
By adding one element of $H_{i+1}$ if necessary, we can arrange that $\id$ is in $\{h'_1, \ldots, h'_m\}$.
Set $A_{i+1}=\bigcup_{j=1}^m A'_ih'_j$ and set $K_{i+1}=mK_i$. By Lemma~\ref{lem:nonexpander}, $A_{i+1}$ is not a toric $K_{i+1}$-expander, so (2) is satisfied. By construction $A_i \subseteq A_{i+1}$, and 
\[
\mu_G(A_{i+1}) \leq \mu_G(A'_iH_{i+1}) \leq \mu_G(A_iH_{i+1}) < K_i\mu_G(A_i), 
\]
so we have (3). For every $h'\in H_{i+1}$, since $A_{i+1}h'\subseteq A'_iH_{i+1}$, we have
\[
\mu_G(A'_i\setminus A_{i+1}h')< \delta\mu_G(A'_iH_{i+1}) \leq \delta \mu_G(A_iH_{i+1}) < \delta K_i\mu_G(A_i).
\]
Therefore,
\[
\mu_G(A_i\setminus A_{i+1}h')< 2\delta K_i\mu_G(A_i).
\]
Note that the construction so far depends on $\delta$.
Now, by choosing $\delta=\delta(K_i, \varepsilon_i)$ sufficiently small, we can make 
\[\mu_G(A_i\setminus A_{i+1}h')<\varepsilon_i\mu_G(A_i),\]
so we get (4). Finally, note that $K_{i+1}=mK_i$,  $m=m(K_i, \delta)$, $\delta= \delta(K_i, \varepsilon_i)$, and $K_i=K_i(\varepsilon_0, \ldots, \varepsilon_{i-1})$, so 
\[K_{i+1} =K_{i+1}(\varepsilon_0, \ldots, \varepsilon_{i}),\]
which gives us (5).

We now proceed with the proof of the lemma. Let $\varepsilon_0,\dots,\varepsilon_{n-1}$ be parameters which we will determine later, and obtain $(A_i)^n_{i=0}$ and $(K_i)^n_{i=0}$ as in the earlier step.
Set $C=A_n$. 
Note that 
\[
G = G^{-1} = (H_1 \ldots H_n)^{-1} = H_n \cdots H_1.
\] 
Hence, an arbitrary $g\in G$ can be written as a product $h_n \cdots h_1$ with $h_i \in H_i$ for $i \in \{1, \ldots, n\}$. Now consider $Cg = A_n h_n \cdots h_1$. By (4), $\mu_G(A_{n-1}\setminus A_nh_n)<\varepsilon_{n-1} \mu_G(A_{n-1})$. Next, for $A_nh_nh_{n-1}$, again by (4), 
\begin{align*}
 \mu_G(A_{n-2}\setminus A_nh_nh_{n-1})&\leq  \mu_G(A_{n-2}\setminus A_{n-1}h_{n-1})+\mu_G(A_{n-1}h_{n-1}\setminus A_nh_nh_{n-1} ) \\ &< \varepsilon_{n-2}\mu_G(A_{n-2}) + \varepsilon_{n-1} \mu_G(A_{n-1}).  
\end{align*}

Hence,  by induction and (3), we conclude that
\[
\mu_G(A\setminus Cg)<\sum_{i=0}^{n-1}\varepsilon_{i}\mu_G(A_i)\leq \left(\sum_{i=0}^{n-1}\varepsilon_i\prod_{j=0}^{i-1}K_j\right) \mu_G(A).
\]
Using (5), we can choose $\varepsilon_i$ sufficiently small such that $\left(\sum_{i=0}^{n-1}\varepsilon_i\prod_{j=0}^{i-1}K_j\right)<1/2$.  Then, for all $g\in G$.
$$ \frac{\mu_G( C\cap Ag)}{ \mu_G(A)}= \frac{\mu_G( Cg^{-1}\cap A)}{ \mu_G(A)} >\frac12. $$
Finally, note that we can choose $\varepsilon_0, \ldots, \varepsilon_{n-1}$ depending only on $K$. Hence, $$\mu_G(C) = O_K(1)\mu_G(A),$$ which completes the proof.
\end{proof}

Suppose $\mu$ and $\nu$ are measures on $G$. Their {\bf convolution} $\mu*\nu$ is the unique measure  satisfying the property
$$  \int_G f(x) \d\mu*\nu(x)  = \int_{G\times G} f(xy) \d\mu(x)\d\nu(y). $$
The convolution exists for all the case we care about. If $\mu(A) >0$, the {\bf uniform measure} on $A$ is defined by
\[
\mu_A(X):=\frac{\mu_G(A\cap X)}{\mu_G(A)} \quad \text{ for measurable }  X \subseteq G. 
\] 
 The following lemma is an immediate consequence of the definition and Fubini's theorem.

\begin{lemma} \label{lem: leftabsorb}
Let $\mu_A$ be the uniform measure on $A$.
Then $\mu_A * \mu_G = \mu_G.$
\end{lemma}
\begin{proof}
Let $X$ be a measurable set in $G$, then
\begin{align*}
  \mu_{A}*\mu_G(X)=\int_G\left(\int_G\mathbbm{1}_{X}(xy)\d\mu_G(y) \right)\d\mu_A(x)=\mu_G(X)\int_G\d\mu_A(x)=\mu_G(X)
\end{align*}
as desired.
\end{proof}

  The next proposition gives us a lower bound control on the measure of sets which are not toric expanders. Together with Lemma~\ref{lem: A=AH toric exp prepare}, this quantitative result shows that a sufficiently small set which inside some nearly minimal expansion pair cannot contain a long fiber in every direction.

\begin{proposition}[Bounding size of toric nonexpanders]\label{prop:toricexpander}
For each $K$, there is $S=O_K(1)$ such that if $A$ is not a toric $K$-expander, then $\mu_G(A)>S$. 
\end{proposition}

\begin{proof}
Let $C$ be as in Lemma~\ref{lem:controllinglargeset}. Recall $\mu_A$ is the uniform measure on $A$, and $\mu_A*\mu_G$ is the convolution measure. Then 
\begin{align*}
    \int_{G} \mathbbm{1}_{C}(x)\d\mu_A*\mu_G(x)&=\int_G\int_G\mathbbm{1}_{C}(xy)\d\mu_A(x)\d\mu_G(y)\\
   &= \int_G \frac{\mu_G(C y^{-1}\cap A)}{\mu_G(A)} \d\mu_G(y)\geq \frac12.
\end{align*}
This means $\mu_A*\mu_G(C)\geq1/2$. By Lemma~\ref{lem: leftabsorb}, this implies $\mu_G(C)\geq1/2$. Since $\mu_G(C)= O_K(1)\mu_G(A)$ we get the desired conclusion.
\end{proof}

We  now deduce the main theorem of this section.

\begin{theorem} \label{thm: Torictransversal}
There is $S=O_{M, \lambda}(1)$ such that if $\mu_G(A)=\mu_G(B) =\kappa < S$ and $\mu_G(AB)< M\kappa$, then there is a one-dimensional torus subgroup $H$ of $G$ such that for all $x, y \in G$
$$ \mu_H(A \cap xH)< \lambda \text{ and } \mu_H(B \cap Hy)< \lambda. $$
\end{theorem}
\begin{proof}
Suppose for every one-dimensional torus subgroup $H$ of $G$, either $\mu_H(A \cap xH)> \lambda$ some $x\in G$ or $\mu_H(B\cap By)> \lambda$ for some $y\in G$. Then by Lemma~\ref{lem: A=AH toric exp prepare}, $A$ is not a toric $K$-expander with 
$ K= (M+2)^2/\lambda.$
Hence, by Proposition~\ref{prop:toricexpander}, we have $\mu_G(A)> S$ with $S= O_{M, \lambda}(1)$. Thus, we get the desired conclusion.
\end{proof}

By computing the constant coefficients carefully in Lemma~\ref{lem:controllinglargeset}, the constant $S$ in Theorem~\ref{thm: Torictransversal} is of order $M^{-M^{2^{d+2}}}$. We did not try to optimise this number in the proof, and we suspect a single exponential bound should be enough.

We get the following immediate corollary for approximate groups. Since our proof is quantitative, we can make the constant below quantitative if we wish.

\begin{corollary}There is $S=O_{K}(1)$ such that if $A$ is a $K$-approximate group with $\mu_G(A)<S$, then there is a one-dimensional torus subgroup $H$ of $G$ such that for all $x, y \in G$
$$ \mu_H(A \cap xH)< \lambda \text{ and } \mu_H(B \cap Hy)< \lambda. $$
\end{corollary}

\subsection{Linear pseudometric from minimal expansions}\label{subsec: linear fiber}

 Throughout the subsection $G$ is a connected \emph{noncompact} unimodular group, and $H$ is a closed subgroup of $G$ which is either isomorphic to $\RR$, or some smaller dimensional connected unimodular group, so that by induction on dimension we may assume Theorem~\ref{thm:mainequal} holds on $H$. Suppose  $(A,B)$ is minimally expanding, that is,
\[
\mu_G(AB)=\mu_G(A)+\mu_G(B),
\]
and both $A,B$ have positive measure. 

This section can also be seen as a preview of Section~\ref{subsec: almost linear fiber}. The strategy of this section also works for compact $G$  replacing $H$ with $\TT$ and results of Section~\ref{subsec: toric exp}. For convenience of notation, we will treat the compact case in Section~\ref{subsec: almost linear fiber} together with the situation where $(A,B)$ is nearly minimally expanding.

\begin{lemma} \label{lem: Kempermansubgroup}
For all $a \in A$ and $b \in B$, we have
\begin{enumerate}
    \item $\mu_H\big(( A \cap aH)(B \cap Hb)\big) \geq \mu_H(A \cap aH) + \mu_H(B\cap Hb).  $
    \item $\mu_G\big(A ( B \cap Hb)\big) \geq  \mu_G(A) + \mu_{G/H}(\pi A)   \mu_{H} ( B \cap Hb ).$
    \item $\mu_G\big((A \cap aH) B\big) \geq  \mu_{H} ( A \cap aH ) \mu_{H \backslash G}(\widetilde{\pi}B) + \mu_G(B).$  
    \end{enumerate}
The equality in (2) holds if and only if the equality in (1) holds for almost all $a \in AH$. A similar conclusion holds for (3).
\end{lemma}

\begin{proof}
The first inequality comes from a direct application of the Kemperman inequality. For the second inequality, by right translating $B$ and using the unimodularity of $G$, we can arrange that $Hb =H$. The desired conclusion follows from applying (1) and Lemma~\ref{lem:int}. 
\end{proof}

  The next theorem gives us the important geometric properties of $A$ and $B$.

\begin{theorem}[Rigidity fiberwise]\label{thm: equal fiberwise same length}
There is a continuous surjective group homomorphism $\chi:H\to\RR$, two compact intervals $I, J \subseteq \RR$  
with 
\[ \mu_\RR(I)=\frac{\mu_G(A)}{\mu_{G/H}(\pi A)}\  \text{ and } \  \mu_\RR(J)=\frac{\mu_G(B)}{\mu_{H\backslash G}(\widetilde{\pi} B)},\]
$\sigma$-compact sets $A'\subseteq A$ and $B'\subseteq B$ with 
\[ \mu_{G/H}(\pi A')=\mu_{G/H}(\pi A) \ \text{ and }\  \mu_{H \backslash G}(\widetilde{\pi} B') = \mu_{H \backslash G}(\widetilde{\pi} B),\]
such that the following hold:
\begin{enumerate}[\rm (i)]
    \item $\mu_{G/H}(\pi A)=\mu_{H \backslash G}(\widetilde{\pi} B)$;
    \item for each $a\in A'H$, we have
    \[  \mu_H( A \cap aH) =\frac{\mu_G(A)}{\mu_{G/H}(\pi A)}, \]
    and there is $\zeta_a\in \RR$ such that $$\mu_H\big((A\cap aH) \tri a\chi^{-1}(\zeta_a+I)\big)=0.$$ 
    \item for each $b\in HB'$, we have 
    \[  \mu_H( B \cap Hb) =\frac{\mu_G(B)}{\mu_{H \backslash G}(\widetilde{\pi} B)}, \]
    and there is $\widetilde{\zeta}_b\in \RR$ such that  $$ \mu_H\big((B \cap Hb) \tri \chi^{-1}(\widetilde{\zeta}_b+J)b\big)=0.$$
\end{enumerate}
\end{theorem}
\begin{proof}
Without loss of generality we assume $\mu_{G/H}(\pi A)\geq\mu_{H\backslash G}(\widetilde{\pi} B)$. Below, we let $Hb$  range over $\widetilde{\pi} B$, and choose $Hb$ uniformly at random in the expectation. 
By Lemma~\ref{lem:int} and the quotient integration formula, we have
\begin{align}
     \sup_{Hb}\mu_G(A(B\cap Hb)) 
    &\geq \mu_G(A)+ \mu_{G/H}(\pi A) \sup_{Hb}\mu_H(B\cap Hb) \label{eq: noncompact1}  \\
    &\geq \mu_G(A)+ \mu_{G/H}(\pi A) \mathbb E_{Hb}\mu_H(B\cap Hb) \nonumber  \\
    &=\mu_G(A)+\frac{\mu_{G/H}(\pi A)}{\mu_{H\backslash G}(\widetilde{\pi}  B)}\mu_G(B)\geq \mu_G(A)+\mu_G(B).\nonumber
\end{align}
Note that $\mu_G(AB)\geq\sup_{Hb}\mu_G(A(B\cap Hb))$. 
Since $\mu_G(AB) = \mu_G(A)+\mu_G(B)$, the equality must hold at each step. In particular, we have
\begin{equation}\label{eq: AH=HB}
    \mu_{G/H}(\pi A)=\mu_{H\backslash G}(\widetilde{\pi}  B),
\end{equation}
and for $\mu_{H \backslash B}$-almost all $Hb\in \widetilde{\pi}(B)$, we  have 
\[
\mu_H(B\cap Hb)=\mathbb E_{Hb'}(B\cap Hb') = \frac{\mu_G(B)}{\mu_{H \backslash G}(\widetilde{\pi} B)}.
\]
Now, using \eqref{eq: AH=HB} and applying the same argument again switching the role of $A$ and $B$, we conclude that for $\mu_{G/H}$-almost all $aH$ in $\pi A$, we have 
\begin{equation} \label{eq: fiberwiseequality}
\mu_H(A\cap aH)=\mathbb E_{a'H}\mu_H(A\cap a'H)= \frac{\mu_G(A)}{\mu_{G/H}(\pi A)}.
\end{equation}
Moreover, the fact that equality holds in \eqref{eq: noncompact1} shows that for $\mu_{G/H}$-almost all $aH\in \pi A$ and $\mu_{H\backslash G}$-almost all $Hb\in \widetilde{\pi} B$ we have
$$ \mu_H((A\cap aH)(B\cap Hb))=\mu_H(A\cap aH)+\mu_H(B\cap Hb).  $$
By the relationship between $\mu_G$ and $\mu_{G/H}$,  in the preceding statement,  we can replace $\mu_{G/H}$-almost all $aH\in \pi A$ with $\mu_G$-almost all $a \in AH$. We can do the same for $\mu_G$ and $\mu_{ H \backslash G}$.

Now, as $H$ satisfies Theorem~\ref{thm:mainequal}, for $a$ an $b$ such that \eqref{eq: fiberwiseequality} holds, we can choose  continuous surjective group homomorphisms $\chi_{a},\chi_b: H\to\RR$ and compact intervals $I_{a}$, and $J_{b}$ in $\RR$ with 
\[
\mu_\RR(I_{a})=\mu_H(A\cap aH), \quad \mu_\RR(J_{b})=\mu_H(B\cap Hb), 
\]
and
\[
\mu_H\big((A\cap aH)\tri \chi^{-1}_{a}(I_{a})\big)=0, \quad \mu_H\big((B\cap Hb)\tri \chi^{-1}_{b}(J_{b})\big)=0.
\]
Applying Fact~\ref{fact:Brunnd=1 R}, we deduce that $\chi_{a}$, $\chi_b$ are the same for $\mu_G$-almost all $a \in AH$ and $\mu_G$-almost all $b \in HB$. We also have
$I_{a}$ and $J_{b}$ have constant lengths for $\mu_G$-almost all $a \in AH$ and $\mu_G$-almost all $b \in HB$. Hence there are $\zeta_a,\widetilde{\zeta}_b\in\TT$ for almost all $a\in AH$ and almost all $b\in HB$, such that $I_a=\zeta_a+I$, and $J_b=\widetilde{\zeta_b}+J$, where
\[
I=\left[0, \frac{\mu_G(A)}{\mu_{G/H}(\pi A)}\right]\quad\text{and}\quad J=\left[0, \frac{\mu_G(A)}{\mu_{H\backslash G}(\widetilde{\pi} B)}\right]. 
\]
It follows that we can choose $A'\subseteq A$ and $B'\subseteq B$ 
as described in the statement of the theorem.
\end{proof}

\begin{corollary}[Global structure of $AH$]\label{lem:AH=gAH}  For all $g\in G$, we have 
\[ \mu_{G/H}(\pi A \tri \pi(gA) ) =0.\]
\end{corollary}
\begin{proof}
Set $\rho =\mu_G(A)$. Recall that $\Stab^{<2\rho}_G(A)$ is open in $G$, and every $g \in G$ can be expressed as a finite products of elements in $\Stab^{<2\rho}_G(A)$ since $G$ is connected. Therefore, it suffices to consider the case where $g$ is in $\Stab^{<2\rho}_G(A)$.  Clearly, $(gA,B)$ is a minimally expanding pair. In the current case, we also have $\mu_G(A\setminus gA)< \rho$ and $\mu_G(A\cap gA)>0$. 
 By  Lemma~\ref{lem:iep}, $(A\cup gA,B)$ is also a minimally expanding pair. Theorem~\ref{thm: equal fiberwise same length}(i) then gives us
 $$ \mu_{G/H}(\pi A) = \mu_{G/H}(\pi (gA)) = \mu_{G/H}(\pi A \cup \pi (gA)) = \mu_{H \backslash G}(\pi B). $$
 This gives us $\mu_{G/H}(\pi A \tri \pi(gA) ) =0$ as desired.
\end{proof}

Theorem~\ref{thm: equal fiberwise same length} and Corollary~\ref{lem:AH=gAH} essentially allows us to define a ``directed linear pseudometric'' on $G$ by ``looking at the generic fiber'' as discussed in the following remark:
\begin{remark}\label{rem}
Fix $a\in AH$ and let the notation be as in Theorem~\ref{thm: equal fiberwise same length}. For $g_1, g_2$ in $G$ such that $g_1^{-1}a, g_2^{-1}a \in A'H$, set
\[
\delta_{a,A}(g_1,g_2)=\zeta_{g_1^{-1}a}-\zeta_{g_2^{-1}a}.
\]
 We have the following linearity property of $\delta_{a,A}$ when the relevant terms are defined, which is essentially the linearity property of the metric from $\RR$.  
\begin{enumerate}
    \item $\delta_{a,A}(g_1,g_1)$=0.
    \item $\delta_{a,A}(g_1,g_2)=\delta_{a,A}(g_2,g_1)$.
    \item $\delta_{a,A}(g_1,g_3)=\delta_{a,A}(g_1,g_2)+\delta_{a,A}(g_2, g_3).$
\end{enumerate}
Properties (1) and (2) are immediate, and property (3) follows from the easy calculation below:
\begin{align*}
  \delta_{a,A}(g_1,g_2)&=\zeta_{g_1^{-1}a}-\zeta_{g_2^{-1}a} \\
    &=\zeta_{g_1^{-1}a}-\zeta_{g_3^{-1}a}+\zeta_{g_3^{-1}a}-\zeta_{g_2^{-1}a} \\
    &= \delta_{a,A}(g_1,g_3)\pm \delta_{a,A}(g_3,g_2).  
\end{align*}
Properties (3) also implies that
\[
| \delta_{a,A}(g_1,g_3)|=\big|\pm |\delta_{a,A}(g_1,g_2)| \pm |\delta_{a,A}(g_2, g_3)|\big|.
\]
which tells us that $|\delta_{a,A}|$ is a linear pseudometric.
The problem with the above definitions is that they are not defined everywhere. 

There are two ways to overcome this difficulty. The new approach, using difference in measure, will be presented later on. The old approach, presented in an earlier version of this paper, is to define a pseudometric on $G$ directly by setting 
$$ d(g_1, g_2) = \xi \text{ if  for } \mu_G\text{-almost all } a \in AH,\  |\delta_{a,A}(g_1,g_2)|=\xi.    $$
This is indeed possible. In fact, one can bypass the pseudometric machinery altogether and define the group homomorphism $\chi: G \to \TT$ directly by setting 
$$  \chi(g) = \zeta \text{ if  for } \mu_G\text{-almost all } a \in AH,\  |\delta_{a,A}(\id,g)|=\zeta. $$
However, this does not come for free, and one need to work equally hard to verify that $\chi$ is indeed a group homomorphism.

The old approach can moreover be extended to the case of nearly minimal expansion. However, we can only handle a quadratic error with this old approach because we only have Corollary~\ref{cor:AH=1}, which lacks the global property of Corollary~\ref{lem:AH=gAH}. The real problem solved by introducing the pseudometric is to get the linear error bound for the nearly minimal expansion problem.
\end{remark}

Recall that for every $g_1,g_2\in G$, 
$
d_A(g_1,g_2)=\mu_G(g_1A\setminus g_2A).  
$
By Proposition~\ref{prop: construct pseudo-metric}, $d_A$ is a pseudometric on $G$. The next lemma builds an important bridge between fiberwise information $\delta_a(g_1,g_2)$ to the pseudometric  $d_A(g_1,g_2)$.

\begin{lemma}[From local to global]\label{lem: equal local to global}
Let $\chi: H \to \RR$ be as in Theorem~\ref{thm: equal fiberwise same length}.
For all $g_1, g_2 \in G$ with $\mu_{G}(g_1A \cap g_2A)>0$, there is a $\sigma$-compact $A'' \subseteq A$ with \[\mu_{G/H}(\pi A'') =\mu_{G/H} (\pi A)\]
such that for all $a \in A''H$, the following holds
\begin{enumerate}[\rm (i)]
    \item there are $\zeta_{g^{-1}_1a},  \zeta_{g^{-1}_2a} \in \TT$ such that for $i \in \{1,2\}$;
    \[ \mu_H\big((A\cap aH)\tri a\chi^{-1}(\zeta_{g_{i}^{-1}a}+I)\big)=0.\]
    \item  with any $\zeta_{g^{-1}_1a}, \zeta_{g^{-1}_2a} $ satisfying (i) and $\delta_{a, A}(g_1, g_2)= \zeta_{g^{-1}_2a}- \zeta_{g^{-1}_1a}$, if we have $\mu_{G}(g_1A \cap g_2A)>0$, then
$$ d_A(g_1, g_2) =  \mu_{G/H}(\pi A)|\delta_{a,A}(g_1,g_2)|.    $$
\end{enumerate}

\end{lemma}

\begin{proof}
Obtain $A', I, J, \zeta_a$ as in Theorem~\ref{thm: equal fiberwise same length}. Let $A'' \subseteq G$ be the $\sigma$-compact set
\[
\{ a \in A: g_1^{-1}a, g_2^{-1}a \in A'H      \}.
\]
By Corollary~\ref{lem:AH=gAH}, $\mu_{G/H}(\pi A'') =\mu_{G/H} (\pi A)$. Fix $a \in A''$.
 We then have 
\[  A \cap g_1^{-1}aH = g_1^{-1}a\chi^{-1}( \zeta_{g^{-1}_1a}+I)  \text{ and }  A \cap g_2^{-1}aH = g_2^{-1}a \chi^{-1}( \zeta_{g^{-1}_2a}+I). \]
Multiplying by $g_1$ and $g_2$ respectively, we get (i). 

As $\mu_{G}(g_1A \cap g_2A)>0$,  by Lemma~\ref{lem:iep}, $(g_1A\cap g_2A, B)$ is  minimally expanding. From Theorem~\ref{thm: equal fiberwise same length}(ii), for $\mu_{G/H}$-almost all $aH \in \pi(g_1A\cap g_2A) $, we have  
\[
\mu_H(g_1A\cap aH) =\frac{\mu_G(g_1A)}{\mu_{G/H}(\pi(g_1A))}\  \text{ and } \ \mu_H(g_1A\cap g_2A\cap aH) =\frac{\mu_G(g_1A\cap g_2A)}{\mu_{G/H}(\pi(g_1A\cap g_2A))}.
\]
Note that 
$\pi(g_1A\cap g_2A) \subseteq  \pi(g_1A) \cap \pi(g_2A).$
However, by Theorem~\ref{thm: equal fiberwise same length},
\[
\mu_{G/H}(\pi(g_1A\cap g_2A) ) = \mu_{H \backslash G}( \widetilde{\pi} B) =  \mu_{G/H}(\pi(g_1A) ) = \mu_{G/H}(\pi(g_2A)).
\]
Combining with Corollary~\ref{lem:AH=gAH}, we get 
\[ \mu_{G/H}( \pi(g_1A) \tri \pi A  ) =0 \ \text{ and } \ \mu_{G/H}( \pi(g_1A\cap g_2A) \tri \pi A  ) =0. \]
Hence, removing a set of measure $0$ from the above $A''$ if necessary, we can arrange that for all $a \in A''H$, 
\[
\mu_H(g_1A\cap aH) =\frac{\mu_G(g_1A)}{\mu_{G/H}(\pi A)}\  \text{ and } \ \mu_H(g_1A\cap g_2A\cap aH) =\frac{\mu_G(g_1A\cap g_2A)}{\mu_{G/H}(\pi(A))}.
\]
Finally, from (i), for all $a \in A''H$, we have
\[
\mu_H(g_1A)-\mu_H(g_1A \cap g_2A \cap aH ) = |\zeta_{g^{-1}_1a} - \zeta_{g^{-1}_1a}|.
\]
Recalling that $d_A(g_1,g_2)= \mu_G(g_1A) -\mu_G(g_1A \cap g_2 A)$, we learn that (ii) is satisfied.
\end{proof}

We now prove the pseudometric is linear as promised.

\begin{proposition}[Linearity of the pseudometric] \label{prop: equal pseudometric linear} For all, $g_1,g_2,g_3$ in $\Stab^{< \rho}_G(A)$, we have
\[ d_{A}(g_1,g_2)\in\{d_{A}(g_1,g_3)+ d_{A}(g_3,g_2),|d_{A}(g_1,g_3)- d_{A}(g_3,g_2)|\}.\]
\end{proposition}
\begin{proof}
For $i \in \{1, 2, 3\}$, $g_i$ is in $\Stab^{< \rho}_G(A)$ by assumption, so $d_A(\id, g_i)< \rho/2$. Hence, for $i, j \in \{1, 2, 3\}$ we have
\[
d_A(g_i, g_j)< \rho\  \text{ and } \ \mu_G(g_iA \cap g_jA)>0.
\]
Applying Lemma~\ref{lem: equal local to global}(ii), we get $\sigma$-compact $A''' \subseteq A$ with $\mu_G(A'''H) \neq \varnothing$ such that for each $a \in A'''H$, we have
$$  \mu_{G}(g_iA \cap g_jA) =  \mu_{G/H}(\pi A)|\zeta_{g^{-1}_ia} - \zeta_{g^{-1}_ja}| \text{ for } i,j \in \{1, 2, 3\}.    $$
The desired conclusion immediately follows.
\end{proof}

\subsection{Almost linear pseudometric from near minimal expansions}\label{subsec: almost linear fiber}

In this subsection, $G$ is always a  \emph{connected compact} Lie group.  Let $H$ be a closed subgroup of $G$, and $H$ is isomorphic to the one dimension torus $\TT$. Throughout the subsection, $A,B\subseteq G$ are $\sigma$-compact subsets such that 
\[
\kappa/2<\mu_G(A) < 2\kappa \text{ and } \mu_G(B) = \kappa,
\]
and $\dis_G(A,B)\leq \eta\kappa$ for some constant $\eta>0$, that is
\[
\mu_G(AB)\leq \mu_G(A)+\mu_G(B)+\eta\kappa.
\]
In this section, we assume $\eta\leq 10^{-12}$. We did not try to optimise $\eta$, so it is very likely that by a more careful computation, one can make $\eta$ much larger. But we believe this method does not allow $\eta$ to be very close to $1$.

 By Fact~\ref{fact:inverse theorem torus}, let $\tau=2$, and $c_B=c(\tau)$ be the constant obtained from the theorem. In this subsection, we consider the case when 
\[
\max\{\mu_H(A\cap aH), \mu_H(B\cap Hb)\}<c_B
\]
 for all $a,b\in G$. The proofs in this section is more involved compared to the equality case, and the main difficulty is to control the error term coming from the nearly minimally expanding pairs. For the readers who do not care the exact quantitative bound on the error terms, one can always view $\eta$ as an infinitesimal element, then one can use equalities to replace all the inequalities in the proofs by pretending to take the standard part, and apply the methods given in the previous section.

 Towards showing that sets $A$ and $B$ behave rigidly, our next theorem shows that most of the nonempty fibers in $A$ and $B$ have the similar lengths, and the majority of them behaves rigidly fiberwise. 

\begin{theorem}[Near rigidity fiberwise]\label{thm: fibers of same length 1newnew} There is a continuous surjective group homomorphism $\chi:H\to\TT$, two compact intervals $I, J\subseteq \TT$ with 
\[ \mu_\TT(I)=\frac{\mu_G(A)}{\mu_{G/H}(\pi A)}\  \text{ and  } \ \mu_\TT(J)=\frac{\mu_{H}(B)}{\mu_{H \backslash G}(\widetilde{\pi} B)}.
\]
$\sigma$-compact $A'\subseteq A$ and $B'\subseteq B$ with 
\[
\mu_{G/H}(\pi A')>99\mu_{G/H}(\pi A)/100\  \text{ and } \ \mu_{H \backslash G}(\widetilde{\pi} B')>99\mu_{H \backslash G}(\widetilde{\pi} B)/100,
\] 
and   a constant $\nu\leq 10^{-10}$ such that the following statements hold:
\begin{enumerate}[\rm (i)]
\item  we have 
\[
\frac{1}{1+\eta}\mu_{H \backslash G}(\widetilde{\pi}B) \leq  \mu_{G/H}(\pi A) \leq  (1 +\eta) \mu_{H \backslash G}(\widetilde{\pi}B) .
\]
   \item For every $a$ in $A'H$, 
    \[
    (1-\nu) \frac{\mu_G(A)}{\mu_{G/H}(\pi A)}\leq \mu_H(A\cap aH)\leq (1+\nu) \frac{\mu_G(A)}{\mu_{G/H}(\pi A)},
    \]
    and there is $\zeta_{a} \in \TT$ with 
    \[ 
       \mu_H\big((A\cap aH)\tri a\chi^{-1}(\zeta_{a}+I)\big)<\nu\min\Big\{\frac{\mu_G(A)}{\mu_{G/H}(\pi A)},\frac{\mu_G(B)}{\mu_{H\backslash G}(\widetilde{\pi} B )}\Big\}.
      \] 
       \item For every $b$ in $HB'$, 
    \[
    (1-\nu) \frac{\mu_G(B)}{\mu_{H\backslash G}(\widetilde{\pi} B )}\leq \mu_H(B\cap Hb)\leq (1+\nu) \frac{\mu_G(B)}{\mu_{H\backslash G}(\widetilde{\pi}B )},
    \]
    and there is $\widetilde{\zeta}_{b}\in \TT$ with
     \[  
       \mu_H\big((B\cap Hb)\tri \chi^{-1}(\widetilde{\zeta}_{b}(B)+J)b\big)<\nu\min\Big\{\frac{\mu_G(A)}{\mu_{G/H}(\pi A)},\frac{\mu_G(B)}{\mu_{H\backslash G}(\widetilde{\pi} )}\Big\}. 
    \]
\end{enumerate}
\end{theorem}
\begin{proof}
Without loss of generality, we assume that $\mu_{G/H}(\pi A)\geq\mu_{H \backslash G}(\widetilde{\pi}B)$. Let $\beta$ be a constant such that $\beta<(\kappa/800\mu_{H\backslash G}(\widetilde{\pi} B))$.  Obtain $b^* \in G$ such that 
$$\mu_H(B\cap Hb^*)\geq \sup_b\mu_H(B\cap Hb)-\beta  \text{ for all } b\in G,$$ 
and the fiber $B\cap Hb^*$ has at least the average length, that is
\begin{equation}\label{eq:choiceb_0}
    \mu_H(B\cap Hb^*)\geq\mathbb E_{b\in BH}\mu_H(B\cap Hb)=\frac{\mu_G(B)}{\mu_{H\backslash G}(\widetilde{\pi} B)}.
\end{equation}
Set $\delta=100\eta\kappa/\mu_{G/H}(\pi A)$. As $\eta\leq 10^{-12}$, we get $\nu\leq 10^{-10}$ such that 
\begin{equation}\label{eq: nu value}
    \delta<\nu\mu_G(A)/\mu_{G/H}(\pi A).
\end{equation} 
Set
\[  
N = \{ a\in AH : \dis_H(A\cap aH,B\cap Hb^*)>\delta  \}.
\]
Note that $N$ is measurable by Lemma~\ref{lem: mesurability}.
By Lemma~\ref{lem:int} we have
\begin{align*}
&\ \mu_G\big(A (B \cap Hb^*)\big)\\
=&\, \int_{N} \mu_H\big(( A \cap aH )(B \cap Hb^*)\big) \d\mu_G(a)+ \int_{G \setminus N} \mu_H\big(( A \cap aH )(B \cap Hb^*)\big) \d\mu_G(a).
\end{align*}
Since $A\cap aH$ is nonempty for every $a\in AH$, using Kemperman's inequality on $H$ we have that $\mu_G(A(B\cap Hb^*))$ is at least
\[
\int_{N} \big(\mu_H( A \cap aH )+\mu_H(B \cap Hb^*)+\delta\big) \d\mu_G(a)
+ \int_{G\setminus N} \big(\mu_H( A \cap aH )+\mu_H(B \cap Hb^*)\big) \d\mu_G(a).
\]
Suppose we have $\mu_G(N)>\mu_{G/H}(\pi A)/100$. Therefore, by the choice of $b^*$ we get
\begin{align}
\mu_G\big(A (B \cap Hb^*)\big)&>\mu_G(A)+\frac{\delta\mu_{G/H}(\pi A)}{100}+\mu_H(B\cap Hb^*)\mu_{G/H}(\pi A)\label{eq:b^*}\\
& \geq\mu_G(A)+\mu_G(B)\frac{\mu_{G/H}(\pi A)}{\mu_{H\backslash G}(\widetilde{\pi} B)}+\eta\kappa.\nonumber
\end{align}
Since $\mu_{G/H}(\pi A)\geq \mu_{H\backslash G}(\widetilde{\pi} B)$, and $A(B\cap Hb^*)\subseteq AB$,  we have
\begin{equation*}
  \mu_G(AB)>\mu_G(A)+\mu_G(B)+\eta\kappa. 
\end{equation*}
This contradicts the assumption that $\dis_G(A,B)$ is at most  $\eta\kappa$. Using the argument in equation (\ref{eq:b^*}) with trivial lower bound on $\mu_G(N)$ we also get 
\begin{equation}\label{eq:AHandHB}
  \mu_{H\backslash G}(\widetilde{\pi} B)\leq \mu_{G/H}(\pi A)\leq \big(1+\eta\big)\mu_{H\backslash G}(\widetilde{\pi} B),  
\end{equation}
which proves (i). 

From now on, we assume that $\mu_G(N)\leq \mu_{G/H}(\pi A)/100$. Since $\dis_G(A,B)$ is at most $\eta\kappa$, by  (\ref{eq:b^*}) again (using trivial lower bound on $\mu_G(N)$), we have
\[
\mu_H(B\cap Hb^*)\mu_{H\backslash G}(\widetilde{\pi} B)\leq \mu_G(B)+\eta\kappa,
\]
and this in particular implies that for every $b\in G$,  we have
\[
\mu_H(B\cap Hb)\leq\frac{\mu_G(B)}{\mu_{H\backslash G}(\widetilde{\pi} B)}+\frac{\eta\mu_G(B)}{\mu_{H\backslash G}(\widetilde{\pi} B)}+\eta<\big(1+\nu \big)\frac{\mu_G(B)}{\mu_{H\backslash G}(\widetilde{\pi} B)}.
\]
Thus there is $Y\subseteq B$ with $\mu_G(HY)<  \mu_{H\backslash G}(\widetilde{\pi} B)/100$  such that for every $b\in HY$, 
\[
\mu_H(B\cap Hb)\geq \frac{\mu_G(B)}{\mu_{H\backslash G}(\widetilde{\pi} B)}-100\frac{\eta\mu_G(B)}{\mu_{H\backslash G}(\widetilde{\pi} B)}-100\eta>\big(1-\nu\big)\frac{\mu_G(B)}{\mu_{H\backslash G}(\widetilde{\pi} B)}.
\]

Next, we apply the similar argument to $A$. Let $\alpha<(\mu_G(A)-2\eta\kappa)/200\mu_{G/H}(\pi A)$, and choose $a^*$ such that $\mu_H(A\cap a^*H)>\mu_H(A\cap aH)-\alpha$ for all $a\in AH$, and
\[
\mu_H(A\cap a^*H)\geq\mathbb E_{a\in AH}\mu_H(A\cap aH)=\frac{\mu_G(A)}{\mu_{G/H}(\pi A)}.
\] 
Let $N'\subseteq HB$ such that for every $b$ in $N'$, $\dis_H(A\cap a^*H,B\cap Hb) \geq \delta$. Hence we have
\begin{align}
&\ \mu_G\big((A\cap a^*H) B\big)\nonumber\\
=&\, \int_{N'} \mu_H\big(( A \cap a^*H )(B \cap Hb)\big) \d\mu_G(b)+ \int_{G \setminus N'} \mu_H\big(( A \cap a^*H )(B \cap Hb)\big) \d\mu_G(b)\nonumber\\
\geq&\, \mu_G(B)+\mu_H(A\cap a^*H)\mu_{H\backslash G}(\widetilde{\pi} B)+\delta\mu_G(N')\label{eq: lambda'}\\
\geq&\, \mu_G(A)+\mu_G(B)-\frac{\mu_G(A)\eta\kappa}{\mu_G(A)+\eta\kappa}+\delta\mu_G(N').\nonumber
\end{align}
By the fact that $\mu_G(AB)\geq\mu_G((A\cap a^*H)B)$ and $\dis_G(A,B)\leq \eta\kappa$, we have that
\[
\mu_G(N')\leq\frac{1}{200}\mu_{G/H}(\pi A)\leq\frac{1}{150}\mu_{H\backslash G}(\widetilde{\pi} B).
\]
Now, by equation (\ref{eq: lambda'}), and the choice of $a^*$, we have that for all $a\in AH$,
\[
\mu_H(A\cap aH)\leq\frac{\mu_G(A)}{\mu_{G/H}(\pi A)}+\frac{\eta\mu_G(A)}{\mu_{H\backslash G}(\widetilde{\pi} B)}+\alpha<\big(1+\nu\big)\frac{\mu_G(A)}{\mu_{G/H}(\pi A)}.
\]
Again by equation (\ref{eq: lambda'}), there is $X\subseteq A$ with $\mu_G(XH)\leq\mu_{G/H}(\pi A)/200$, such that for every $a\in X$,
\[
\mu_H(A\cap aH)\geq\frac{\mu_G(A)}{\mu_{G/H}(\pi A)}-200\frac{\eta\mu_G(A)}{\mu_{H\backslash G}(\widetilde{\pi} B)}-200\alpha\geq\big(1-\nu\big)\frac{\mu_G(A)}{\mu_{G/H}(\pi A)}.
\]

Let $A'=A\cap(AH\setminus(XH\cup N'))$, and let $B'=B\cap (HB\setminus(HY\cup N))$. Then 
\[
\mu_G(A')\geq \frac{99}{100}\mu_{G/H}(\pi A),\quad \mu_G(B')\geq \frac{99}{100}\mu_{H\backslash G}(\widetilde{\pi} B),
\]
Let $a$ be in $A'H$ and $b$ be in $B'H$. By our construction, the first parts of (ii) and (iii) are satisfied. Moreover, both $\dis_H(A\cap aH,B\cap Hb^*)$ and $\dis_H(A\cap a^*H,B\cap Hb)$ are at most $\delta$. By the way we construct $A'$ and $B'$, we have that $a^*\in A'$ and $b^*\in B'$. Recall that $\mu_H(A\cap aH),\mu_H(B\cap Hb)<\lambda$ for every $a,b\in G$. Therefore, by the inverse theorem on $\TT$ (Fact~\ref{fact:inverse theorem torus}), and Lemma~\ref{lem:fromsmalltobig}, there is a group homomorphism $\chi:H\to\TT$, and two compact intervals $I_A$, $I_B$ in $\TT$, with
\[
\mu_\TT(I_A)=\frac{\mu_G(A)}{\mu_{G/H}(\pi A)},\quad \mu_\TT(I_B)=\frac{\mu_G(B)}{\mu_{H\backslash G}(\widetilde{\pi} B)},
\]
such that for every $a\in A'$ and $b\in B'$, there are elements  $\zeta_{a},\widetilde{\zeta}_{b}(B)$ in $\TT$, and
\[
\mu_H(A\cap aH\tri a\chi^{-1}(\zeta_{a}+I_A))<\delta,\quad 
\mu_H(B\cap Hb\tri \chi^{-1}(\widetilde{\zeta}_{b}+I_B)b)<\delta,
\]
and the theorem follows from \eqref{eq: nu value}.
\end{proof}

  The next corollary gives us an important fact of the structure of the projection of $A$ on $G/H$.

\begin{corollary}[Global structure of $AH$]\label{cor:AH=1} Suppose $\mu_G(A)=\kappa$. Then for all $g \in \Stab^{\kappa/2}_G(A)$, we have
 $$\mu_G(AH\tri gAH)\leq 10\eta \mu_{G/H}(\pi A).$$
\end{corollary}

\begin{proof}
By Lemma~\ref{lem:iep}, $\dis_G(A\cup gA, B)\leq 2\eta\kappa$, and by Theorem~\ref{thm: fibers of same length 1newnew}, we have
\[
\frac{1}{1+2\eta}<\frac{\mu_G(AH\cup gAH)}{\mu_{H\backslash G}(\widetilde{\pi} B)}<1+2\eta.
\]
On the other hand, since $\dis_G(A,B)$ and $\dis_G(gA,B)$ are at most $\eta\kappa$, we have 
\[
\frac{1}{1+\eta}<\frac{\mu_{G}(AH)}{\mu_{H\backslash G}(\widetilde{\pi} B)}, \frac{\mu_G(gAH)}{\mu_{H\backslash G}(\widetilde{\pi} B)}<1+\eta
\]
Since $\eta< 10^{-12}$, by inclusion-exclusion principle, we get the desired conclusion. 
\end{proof}

Given $\tau<1/4$, we use $I(\tau)$ to denote the open interval $(-\tau,\tau)$ in $\TT$. As in Remark~\ref{rem}, one can similarly consider
\[
\delta_{a,A}(g_1,g_2)=\zeta_{g_1^{-1}a}-\zeta_{g_2^{-1}a}
\]
for a fix $a\in A$. Recall that
\[
d_A=\mu_G(g_1A\setminus g_2A)
\]
is a pseudometric (Proposition~\ref{prop: construct pseudo-metric}). The next lemma can be seen as an approximate version of Lemma~\ref{lem: equal local to global}, which gives the connection between $\delta_{a,A}(g_1,g_2)$ and $d_A(g_1,g_2)$.  

\begin{lemma}[From local to global] \label{lem: 2 elements near rigidity}
Suppose $\mu_G(A)=\kappa$, $g_1, g_2 \in \Stab^{\kappa/4}_G(A)$, and $\chi: H \to \TT$, $I \subseteq \TT$, $\nu$ are as in Theorem~\ref{thm: fibers of same length 1newnew}. Then there is a $\sigma$-compact $A'' \subseteq A$ with \[\mu_{G/H}(\pi A'') =96\mu_{G/H} (\pi A)/100\]
such that for all $a \in A''H$, the following holds
\begin{enumerate}[\rm (i)]
    \item there are $\zeta_{g^{-1}_1a},  \zeta_{g^{-1}_2a} \in \TT$ such that for $i \in \{1,2\}$;
    \[ \mu_H\big((A\cap aH)\tri a\chi^{-1}(\zeta_{g_{i}^{-1}a}+I)\big)<\frac{\nu\kappa}{\mu_{G/H}(\pi A)}.\]
    \item with any $\zeta_{g^{-1}_1a}, \zeta_{g^{-1}_2a} $ satisfying (i) and $\delta_{a, A}(g_1, g_2)= \zeta_{g^{-1}_2a}- \zeta_{g^{-1}_1a}$, we have
$$ d_A(g_1, g_2) \in  \mu_{G/H}(\pi A)|\delta_{a, A}(g_1, g_2)| + I\big(20\nu\kappa\big).    $$
\end{enumerate}
\end{lemma}

\begin{proof}
Obtain $A', I, J$ as in Theorem~\ref{thm: fibers of same length 1newnew}. Let $A'_1 \subseteq G$ be the $\sigma$-compact set
\[
\{ a \in A: g_1^{-1}a, g_2^{-1}a \in A'H      \}.
\]
It is easy to see that $\mu_{G/H}(\pi A'_1) > 98/100\mu_{G/H} (\pi A)$. Fix $a \in A'_1$, by Theorem~\ref{thm: fibers of same length 1newnew} again
 we then have 
 \[
  \mu_H\big((A\cap g_1^{-1}aH)\tri g_1^{-1}a\chi^{-1}(\zeta_{g_1^{-1}a}+I)\big)<\nu\frac{\kappa}{\mu_{G/H}(\pi A)},
 \]
 and
  \[
  \mu_H\big((A\cap g_2^{-1}aH)\tri g_2^{-1}a\chi^{-1}(\zeta_{g_2^{-1}a}+I)\big)<\nu\frac{\kappa}{\mu_{G/H}(\pi A)}.
 \]
Multiplying by $g_1$ and $g_2$ respectively, we get (i) when $A''\subseteq A'_1$. 

As $g_1, g_2 \in \Stab^{\kappa/4}_G(A)$, we have $\mu_{G}(g_1A \cap g_2A)>0$. Note that $\dis_G(g_1A\cap g_2A, B)\leq 2\eta\kappa$ by Lemma~\ref{lem:iep}. By Theorem~\ref{thm: fibers of same length 1newnew}(i), we have
\begin{equation}\label{eq: compute intersection}
\mu_G((g_1A\cap g_2A)H)\geq \frac{1}{1+2\eta}\mu_G(HB) \geq \frac{1}{(1+2\eta)(1+\eta)}\mu_{G}(AH). 
\end{equation}
Also Theorem~\ref{thm: fibers of same length 1newnew}(ii) implies that there is $A'_2\subseteq g_1A\cap g_2A$ with $\mu_{G/H}(\pi A'_2)>99\mu_{G/H}(\pi (g_1A\cap g_2A))/100$, such that for all $a\in A'_2H$, we have  
\begin{equation}\label{eq: pseu first}
\big(1-\nu\big)  \frac{\mu_G(g_1A)}{\mu_{G/H}(\pi(g_1A))} \leq  \mu_H(g_1A\cap aH) \leq (1+\nu)\frac{\mu_G(g_1A)}{\mu_{G/H}(\pi(g_1A))}\ 
\end{equation}
and
\[ (1-2\nu)\frac{\mu_G(g_1A\cap g_2A)}{\mu_{G/H}(\pi(g_1A\cap g_2A))} \leq \mu_H(g_1A\cap g_2A\cap aH) \leq (1+2\nu)\frac{\mu_G(g_1A\cap g_2A)}{\mu_{G/H}(\pi(g_1A\cap g_2A))}.
\]

Note that 
$\pi(g_1A\cap g_2A) \subseteq  \pi(g_1A) \cap \pi(g_2A).$
However, \eqref{eq: compute intersection} together with Corollary~\ref{cor:AH=1} give us
\begin{align*}
&\mu_{G/H}( \pi(g_1A) \tri \pi A  ) \leq 10\eta \mu_{G/H}(\pi A),\\
&\mu_{G/H}( \pi(g_1A\cap g_2A) \tri \pi A  ) \leq 14\eta \mu_{G/H}(\pi A). 
\end{align*}
Hence, Let $A''=A'_1H\cap A'_2H\cap A$. Note that $\mu_{G/H}(\pi A'')\geq 96\mu_{G/H}(\pi A)/100$, and for all $a \in A''H$, by \eqref{eq: compute intersection}
\begin{equation}\label{eq: pseu second}
(1-5\nu) \frac{\mu_G(g_1A\cap g_2A)}{\mu_{G/H}(\pi(A))} \leq \mu_H(g_1A\cap g_2A\cap aH) \leq (1+5\nu) \frac{\mu_G(g_1A\cap g_2A)}{\mu_{G/H}(\pi(A))}.
\end{equation}
Finally, recall that $d_A(g_1,g_2) =\mu_G(g_1A)-\mu_G(g_1A \cap g_2A)$. Note that
\begin{align*}
\mu_{G/H}(\pi A)|\delta_{a,A}(g_1,g_2)|&=\mu_{G/H}(\pi A)|\zeta_{g_1^{-1}a}-\zeta_{g_2^{-1}a}|\\
&\in \mu_H(g_1A\cap aH)-\mu_H(g_1A\cap g_2A\cap aH)+I(2\nu\kappa).
\end{align*}
Hence, by \eqref{eq: pseu first} and \eqref{eq: pseu second},
\begin{align*}
d_A(g_1,g_2)&\geq \mu_{G/H}(\pi A)\left(\frac{\mu_H(g_1A\cap aH)}{1+\nu}-\frac{\mu_H(g_1A\cap g_2A\cap aH)}{1-5\nu}\right)\\
&\geq \mu_{G/H}(\pi A)\big(\mu_H(g_1A\cap aH)-\mu_H(g_1A\cap g_2A\cap aH)\big)-18\nu\kappa\\
&\geq \mu_{G/H}(\pi A)|\delta_{a,A}(g_1,g_2)|-20\nu\kappa. 
\end{align*}
The upper bound on $d_A(g_1,g_2)$ can be computed using a similar method, and this finishes the proof.
\end{proof}


We now deduce properties of the pseudometric $d_A$. Besides the almost linearity, we also need the path monotonicity of the pseudometric to control the errors in the almost homomorphism obtained in Section~\ref{section:pseudometric}. 

\begin{proposition}[Almost linearity and path monotonicity of the pseudometric]\label{prop: almost linear metric from local}
Assume that $\mu_G(A)=\kappa$, and let $\nu$ be as in Theorem~\ref{thm: fibers of same length 1newnew}. Then we have the following:
\begin{enumerate}[\rm (i)]
    \item For all $g_1,g_2,g_3$ in $\Stab^{\kappa/2}_G(  A )$, we have 
    \[
    d_A(g_1,g_2)\in |\pm d_A(g_1,g_3)\pm d_A(g_2,g_3)|+I\big(60\nu\kappa\big),
    \]
       \item Let $\mathfrak{g}$ be the Lie algebra of $G$, and let $\exp: \mathfrak{g}\to G$ be the exponential map. For every $X\in \mathfrak{g}$, either 
       $$d_A(\exp(Xt), \id)<\kappa/4 \text{ for all } t\in\RR$$ 
       or  there is $t_0>0$ with $d_A(\exp(Xt_0), \id)\geq \kappa/4$  such that for every $t\in[0,t_0]$, 
    \begin{align*}
   &\, d_A(\exp(X(t+t_0)),\id)\\
    \in&\, d_A(\exp(X(t+t_0)),\exp(Xt_0))+d_A(\exp(Xt_0),\id)+I\big(180\nu \kappa\big).
    \end{align*}
\end{enumerate}
\end{proposition}
\begin{proof}
We first prove (i). Let $\chi$ and $I$ be as in Theorem~\ref{thm: fibers of same length 1newnew}. Applying Lemma~\ref{lem: 2 elements near rigidity}, we get $a \in  AH$ and  $\zeta_{g^{-1}_1a},  \zeta_{g^{-1}_2a}, \zeta_{g^{-1}_1a} \in \TT$ such that for $i \in \{1, 2, 3\}$, we have
\[ \mu_H\big((A\cap aH)\tri a\chi^{-1}(\zeta_{g_{i}^{-1}a}+I)\big)<\frac{\nu\kappa}{\mu_{G/H}(\pi A)},\]
and for $i, j \in \{1, 2, 3\}$, we have
$$ d_A(g_i, g_j) \in  \mu_{G/H}(\pi A)|\delta_{a, A}(g_i, g_j)| + I\big(20\nu\kappa\big).    $$
with $\delta_{a, A}(g_i, g_j)= \zeta_{g^{-1}_ja}- \zeta_{g^{-1}_ia}$.
As $\delta_{a, A}(g_1, g_2)= \delta_{a, A}(g_1, g_3)+ \delta_{a, A}(g_3, g_2)$, we get the desired conclusion.

Next, we prove (ii). 
Let $X\in \mathfrak{g}$, and suppose there is $t>0$ such that 
$$d_A(\exp(Xt),\id)\geq \kappa.$$
Using the continuity of $g \mapsto \mu_G(A\setminus gA)$ (Fact~\ref{fact: Haarmeasurenew}(vii)), we obtain $t_0>0$ such that $t_0$ the smallest positive real number with $d_A(\id, \exp(Xt_0))\geq \kappa/10$. Fix $t\in [0,t_0]$, and set 
$$g_0= \exp(Xt_0) \text{ and } g = \exp(Xt). $$
Note that $gg_0=g_0g$ as $g_0$ and $g$ are on the same one parameter subgroup of $G$. 
One can easily check that $g_0$, $g$, $g_0g$ are in $\Stab^{\kappa/2}_G(A)$. Again, let $\chi$ and $I$ be as in Theorem~\ref{thm: fibers of same length 1newnew} and apply Lemma~\ref{lem: 2 elements near rigidity} to get $a \in  AH$ and  $\zeta_{g^{-1}_ia} \in \TT$ for $g_i \in \{ \id, g, g_0, gg_0 \}$ such that   
\begin{equation} \label{8.18.1}
 \mu_H\big((A\cap aH)\tri (a\chi^{-1}(\zeta_{g_{i}^{-1}a}+I)\big)<\frac{\nu\kappa}{\mu_{G/H}(\pi A)},
\end{equation}
and  for  $g_i, g_j \in \{ \id, g, g_0, gg_0 \}$, we have
\begin{equation} \label{8.18.2}
 d_A(g_i, g_j) \in  \mu_{G/H}(\pi A)|\delta_{a, A}(g_i, g_j)| + I\big(20\nu\kappa\big)   
\end{equation} 

As $gg_0= g_0g$,  we have
\begin{equation} \label{18.8.3}
   \delta_{a, A}(  \id, g)+\delta_{a,A}(g,gg_0) = \delta_{a, A}(  \id, g_0g) = \delta_{a, A}(  \id, g_0) + \delta_{a, A}(  g_0, g_0g)   
\end{equation}
Using \eqref{8.18.1}, \eqref{8.18.2}, and the fact that $d_A(\id, g_0) = d_A(g, gg_0) $, we get
\begin{align*}
\delta_{a,A}(g, gg_0)\in \pm \delta_{a,A}(\id, g_0)+I\big(60\nu\kappa\big). 
\end{align*}
By a similar argument, $
\delta_{a,A}(g_0, g_0g)\in \pm \delta_{a,A}(\id, g)+I\big(60\nu\kappa\big)$.
Combining with \eqref{18.8.3}, we get that
 $\delta_{a,A}(\id,gg_i)$ is in both
\[
\delta_{a,A}(\id, g_0)\pm \delta_{a,A}(\id,g)+I\big(60\nu\kappa\big)
\]
and
\[
\delta_{a,A}(\id, g)\pm \delta_{a,A}(\id,g_0)+I\big(60\nu\kappa\big).
\]
Using the fact that $\nu<10^{-6}$, and considering all the four possibilities, we deduce 
\[
\delta_{a,A}(\id, gg_0) =\delta_{a,A}(\id, g_0)+ \delta_{a,A}(\id,g)+I\big(120\nu\kappa\big).
\]
Applying \eqref{8.18.1} and \eqref{8.18.2} again, we get the desired conclusion.
\end{proof}

\section{Proof of the main theorems}\label{sec: 9}
 In this section, we put everything together to prove   some slight generalizations of Theorem~\ref{thm:mainequal} and  Theorem~\ref{thm:mainapproximate} as well as Theorem~\ref{thm: maingap}.
 
\subsection{Minimal expansions in noncompact groups}\label{sec: 9.1}
   The next theorem is a restatement of Theorem~\ref{thm:mainequal}(vi), which is the main result in this subsection. 

\begin{theorem}[Main theorem for noncompact groups] Suppose $G$ is a connected unimodular noncompact group, $\widetilde{\mu}_G$ is the inner measure corresponding to a Haar measure of $G$, the sets $A, B \subseteq G$ have $0< \widetilde{\mu}_G(A), \widetilde{\mu}_G(B)< \infty$, and 
\[
\widetilde{\mu}_G(AB)=\widetilde{\mu}_G(A)+\widetilde{\mu}_G(B).
\]
Then there is a continuous surjective group homomorphism $\chi:G\to\RR$ with compact kernel, and compact intervals $I, J \subseteq \RR$ with $\mu_\RR(I) = \widetilde{\mu}_G(A)$ and $\mu_\RR(J) =\widetilde{\mu}_G(B)$, such that
\[
A\subseteq\chi^{-1}(I),\quad\text{and}\quad B\subseteq\chi^{-1}(J). 
\]
If $A$ and $B$ are, moreover, compact, then we also have $A=\chi^{-1}(I)$ and $B=\chi^{-1}(J)$. 
\end{theorem}
\begin{proof}
We first treat the case where both $A$ and $B$ are $\sigma$-compact. In this case, we can simply use the Haar measure $\mu_G$ instead of the inner measure $\widetilde{\mu}(G)$. It is easy to see that the last assertion of this theorem is an immediate consequence of the conclusion for this special case.

By the Gleason--Yamabe Theorem (Fact~\ref{fact:Yamabe}), there
is a connected compact normal subgroup $H$ of $G$ such that $L=G/H$ is a connected Lie group. By Fact~\ref{fact: unimodular}, $L$ is unimodular. Let $\pi: G\to L$ be the quotient map. Using Corollary~\ref{cor:transfer to Lie noncompact}, there are $\sigma$-compact subsets $A',B'$ of $L$ such that
\[
\mu_G(A\tri \pi^{-1}(A'))=0\quad\text{and}\quad \mu_G(B\tri \pi^{-1}(B'))=0,
\]
and we still have $\mu_L(A'B')=\mu_L(A')+\mu_L(B')$. 

When $L$ is a simple Lie group, by the Iwasawa decomposition, $L=KAN$ where $AN$ is a simply connected closed nilpotent group. Thus $AN$ contains $\RR$ as a closed subgroup, and so does $L$. When $L$ is not simple, then $L$ contains a connected closed normal subgroup $H$, and by Fact~\ref{fact: unimodular}, $H$ is unimodular, and of smaller dimension. Applying induction on dimension, we may assume $H$ satisfies the statement of the theorem. For $g_1,g_2\in L$, let 
\[
d_A(g_1,g_2)=\mu_L(A)-\mu_L(g_1A\cap g_2A).
\]
Then by Proposition~\ref{prop: construct pseudo-metric}, $d_A$ is a pseudometric on $L$, with radius $\mu_G(A)$. By Proposition~\ref{prop: equal pseudometric linear}, $d_A$ is locally linear. Using Proposition~\ref{prop: strong linear}, we have $\ker d_A$ is a compact normal subgroup of $L$, and $L/\ker d_A$ is isomorphic to $\RR$ as topological groups. Note that $\ker d_A$ is the same as its identity component $(\ker d_A)_0$, because $L/(\ker d_A)_0$ is a cover group of $\RR$, which must be $\RR$ because it is already simply connected. By the third isomorphism theorem (Fact~\ref{fact: first iso thm}), $\RR$ is a quotient group of $G$, and the corresponding quotient map $\chi$ has a connected compact kernel. Applying Corollary~\ref{cor:transfer to Lie noncompact} again, as well as the inverse theorem on $\RR$ (Fact~\ref{fact:Brunnd=1 R}),  there are $I,J$ compact intervals of $\RR$ such that
\[
\mu_G(A\tri \chi^{-1}(I))=0\quad\text{and}\quad \mu_G(B\tri \chi^{-1}(J))=0.
\]
This also implies that $\mu_G(A)=\mu_\RR(I)$ and $\mu_G(B)=\mu_\RR(J)$. 

Suppose $g\in A$ and $g\notin \chi^{-1}(I)$. Since $I$ is compact, $\chi(g)\notin I$, and there is $\alpha>0$ such that the distance between $\chi(g)$ and the nearest element in $I$ is at least $\alpha$. Thus
\[
\mu_\RR(\chi(g)\chi(B)\setminus \chi(A)\chi(B))\geq \alpha. 
\]
and this implies that $\mu_G(gB\setminus \chi^{-1}(I)\chi^{-1}(J))\geq\alpha$. Therefore,
\[
\mu_G(AB)\geq \mu_G(\chi^{-1}(I)\chi^{-1}(J))+\mu_G(gB\setminus \chi^{-1}(I)\chi^{-1}(J))\geq \mu_G(A)+\mu_G(B)+\alpha,
\]
and this contradicts the fact that $(A,B)$ is minimally expanding. Hence, we have $A\subseteq \chi^{-1}(I)$ and $B\subseteq \chi^{-1}(J)$ as desired. 

Now we treat the general case without the assumption that $A$ and $B$ are $\sigma$-compact. By inner regularity, we obtain $\sigma$-compact $\widetilde{A} \subseteq A$ and $\widetilde{B} \subseteq B$ such that 
\[
\mu_G(\widetilde{A})=\widetilde{\mu}_G(A)\  \text{ and }\ \mu_G(\widetilde{B})=\widetilde{\mu}_G(B).
\]
Then $\widetilde{A}\widetilde{B}$ is a $\sigma$-compact subset of $AB$. From $\widetilde{\mu}_G(AB)=\widetilde{\mu}_G(A)+\widetilde{\mu}_G(B)$, we obtain
\[  \mu_G(\widetilde{A}\widetilde{B})=\mu_G(\widetilde{A})+ \mu_G(\widetilde{B}).\]
Obtain $\chi$, $I$, and $J$ for $\widetilde{A}$ and $\widetilde{B}$ as in our earlier proven special case. Argue as in the preceding paragraph, we get $A \subseteq \chi^{-1}(I)$ and $B\subseteq \chi^{-1}(I)$ as desired.
\end{proof}

  We remark that the same argument almost works for compact groups when $(A,B)$ is a minimal expansion pair, except that when we choose the closed subgroup $H$, we need to choose one such that we are in the toric transversal scenario. This can be done by using Theorem~\ref{thm: Torictransversal} (See Section~\ref{sec: main compact}).

\subsection{Nearly minimal expansions in compact groups}\label{sec: main compact}

In this subsection, we prove the main theorems for compact groups. We first prove Theorem~\ref{thm: maingap}. 
\begin{proof}[Proof of Theorem~\ref{thm: maingap}]
Let $d>0$ be an integer, let $c_B$ be the real number fixed at the beginning of Section~\ref{subsec: almost linear fiber}, let $\nu$ be in Theorem~\ref{thm: fibers of same length 1newnew}. Let $G$ be a connected compact simple Lie group of dimension at most $d$, and $A$ is a compact subset of $G$ of measure at most $S$, and $\mu_G(A^2)<(2+\eta)\mu_G(A)$, where $S=S(d)$ is the constant in Theorem~\ref{thm: Torictransversal} (with $M=3$), and $\eta$ is the constant fixed in Section~\ref{subsec: almost linear fiber}. 

By Theorem~\ref{thm: Torictransversal}, when $A$ satisfies $\mu_G(A)<S$, there is a closed subgroup $H$ which is isomorphic to $\TT$, such that for all $g\in G$, $\mu_H(gA\cap H)<c_B$. We fix such a closed subgroup $H$ of $G$.  Let
\[
d_A(g_1,g_2)=\mu_G(A)-\mu_G(g_1A\cap g_2A).
\]
By Proposition~\ref{prop: construct pseudo-metric}, $d_A$ is a pseudometric. Since $\mu_G(A^2)<(2+\eta)\mu_G(A)$, Proposition~\ref{prop: almost linear metric from local} shows that $d_A$ is a $60\nu\mu_G(A)$-linear pseudometric, and it is $180\nu\mu_G(A)$-path-monotone. By Proposition~\ref{prop: localmonotoneimplyglobalmonotone}, $d_A$ is globally $1620\nu\mu_G(A)$-monotone. Let $\gamma=1620\nu\mu_G(A).$ Then $d_A$ is $\gamma$-monotone $\gamma$-linear, and of radius $\rho=\mu_G(A)/2$. As $\nu\leq 10^{-10}$, $10^6\gamma<\rho$, and Theorem~\ref{thm: homfrommeasurecompact} thus implies that there is a continuous surjective group homomorphism mapping $G$ to $\TT$, and this contradicts the fact that $G$ is simple. 
\end{proof}

  Now we are going to prove the inverse theorem. In Proposition~\ref{prop: struc on AB}, given a continuous surjective group homomorphism from $G$ to $\TT$, we obtain a structural characterization with further assumption that the images of both $A$ and $B$ are small. The next lemma says that, with the homomorphism obtained from almost linear pseudometric in Sections 7 and 8, both $A$ and $B$  should have small image. 

\begin{lemma}\label{lem: making projection small}
Suppose $G$ is a connected compact groups,  $\chi: G \to \TT$ is a continuous and surjective group homomorphism with connected kernel, $A, B \subseteq G$ are nonempty and $\sigma$-compact with 
$$ \dis_G(A,B)< \mu_G(A)+\mu_G(B) < \max\{\mu_G(A),\mu_G(B)\}<1/250.   $$
Suppose for every $g\in \ker(\chi)$ with $\mu_G(A\setminus gA)<\mu_G(A)/36$, we further have 
$$\mu_G(A\setminus gA)<\mu_G(A)/72.$$ Then $\mu_\TT(\chi(A))+\mu_\TT(\chi(B))<1/5$. 
\end{lemma}
\begin{proof}
Set $H=\ker(\chi)$. We first show that   $\sup_g\mu_H(A\cap gH)> 1/2$. Suppose to the contrary that $\sup_g\mu_H(A\cap gH)\leq 1/2$. Note that
\[
\frac{\mu_G(AH)}{\mu_G(A)}>\frac{1}{\sup_g \mu_H(A\cap gH)}. 
\]
Hence by Lemma~\ref{lem: trans along H}, for every $\ell>1/2$, there is $h\in H$ such that $\mu_G(A\cap hA)=\ell \mu_G(A)$, and in particular, there is $h\in H$ with
\[
\frac{35\mu_G(A)}{36}<\mu_G(A\cap hA)<\frac{71\mu_G(A)}{72},
\]
which contradicts the assumption. 

Now apply Lemma~\ref{lem: A=AH toric exp prepare}, we get
$
\mu_{\TT}(\chi A) +\mu_{\TT}(\chi B) \leq 50 ( \mu_G(A)+\mu_G(B) )=1/5
$
as desired.
\end{proof}

  With all tools in hand, we are going to prove the following theorem, which is a restatement of Theorem~\ref{thm:mainapproximate}
 and Theorem~\ref{thm:mainequal}(v) for compact groups.

\begin{theorem}[Main theorem for compact groups]\label{thm:section7main}
  Let $G$ be a connected compact group, and $A,B$ be compact subsets of $G$ with  $$0< \lambda=\min\{\mu_G(A),\mu_G(B),1-\mu_G(A)-\mu_G(B)\}.$$
  There is a constant $K=K(\lambda )$, not depending on $G$, such that for any $0\leq \varepsilon<1$, whenever we have $\delta\leq K\varepsilon$ and
  \[
  \widetilde{\mu}_G(AB)<\widetilde{\mu}_G(A)+\widetilde{\mu}_G(B)+\delta\min\{\widetilde{\mu}_G(A),\widetilde{\mu}_G(B)\},
  \]
there is a surjective continuous group homomorphism $\chi: G \to \TT$ together with two compact intervals $I,J\in \mathbb T$ with
 \[
 \mu_\TT(I)<(1+\varepsilon)\mu_G(A),\quad \mu_\TT(J)<(1+\varepsilon)\mu_G(B),
 \]
 and $A\subseteq \chi^{-1}(I)$, $B\subseteq\chi^{-1}(J)$.
 
If $A, B$ are, moreover, compact with
 $\mu_G(AB)=\mu_G(A)+\mu_G(B)$, then we also have $$A=\chi^{-1}(I) \text{  and } B=\chi^{-1}(J).$$
\end{theorem}
\begin{proof}
The latter part of the theorem follows immediately from the former part by the fact that open subsets of $G$ has positive $\mu_G$-measure.

We now prove the former part of the theorem. Consider first the case  $\varepsilon>0$. Let $\delta>0$ to be determined later. 
Suppose $\mu_G(A)\leq \mu_G(B)$, and 
$$\widetilde{\dis}_G(A,B)= \widetilde{\mu}_G(AB)-\widetilde{\mu}_G(A)-\widetilde{\mu}_G(A)= \delta\widetilde{\mu}_G(A).$$
By the inner regularity of $\widetilde{\mu}_G$, we obtain $\sigma$-compact $A_0 \subseteq A$ and $B_0 \subseteq B$ such that 
\[
\mu_G(A_0) = \widetilde{\mu}_G(A) \ \text{ and } \ \mu_G(B_0) = \widetilde{\mu}_G(B).\] As $A_0B_0$ is a $\sigma$-compact subset of $AB$, we have $\mu_G(A_0B_0) \leq \widetilde{\mu}_G(AB)$ and
\[ \dis_G(A_0, B_0)\leq \widetilde{\dis}_G(A,B) =\delta\mu_G(A_0). \]
Let $\tau$ be the constant from Proposition~\ref{prop: coarsetheorem}.
Using Lemma~\ref{lem:sml}, we obtain a constant $K_1=O_\lambda(1)$ and $\sigma$-compact $A_1,B_1 \subseteq G$ such that $\mu_G(A_1)=\mu_G(B_1)<\tau$, 
\[
\dis_G(A_1, B_1)<K_1\delta\mu_G(A_0),
\]
and both $\dis_G(A_1,B)$ and $\dis_G(A,B_1)$ are at most $K_1\delta\mu_G(A)$. Applying Proposition~\ref{prop: coarsetheorem}, we get a connected compact subgroup $H$ of $G$ with $G'=G/H$ a Lie group of dimension $O(1)$ and $\sigma$-compact  $A'_1,B'_1 \subseteq G'$, such that $\mu_L(A'_1)\leq\mu_L(B'_1)$,
\[
\dis_{G'}(A'_1, B'_1)<7K_1\delta\mu_G(A_0),
\]
and with $\pi: G\to G'$ the quotient map, we have
\begin{equation}\label{eq: approx main compact}
\max\{\mu_{G'}(A_1\tri\pi^{-1}(A'_1)),\mu_G(B_1\tri\pi^{-1}(B'_1))\}<3K_1\delta\mu_G(A_0).
\end{equation}

Let $S$ be the constant from Theorem~\ref{thm: Torictransversal} with $M=3$. We apply Lemma~\ref{lem:sml} again to get a constant $K_2$ and $\sigma$-compact $A'_2,B'_2 \subseteq G'$, such that 
$$\mu_{G'}(A'_2)= \mu_{G'}(B'_2)=S, $$ and all of $\dis_{G'}(A'_2,B'_2)$, $\dis_{G'}(A'_1,B'_2)$, and $\dis_{G'}(A'_2,B'_1)$ are at most $K_1K_2\delta\mu_G(A_0)$.

Let $c_B$ be the constant fixed in the beginning of Section~\ref{subsec: almost linear fiber}. By Theorem~\ref{thm: Torictransversal}, we obtain a one-dimensional torus subgroup $H'$ of $G'$, such that for every $g'
\in G'$
\[
\max\{\mu_{H'}(A'_2\cap g'H'), \mu_T(B'_2\cap H'g')\}<c_B. 
\]
 Similarly as what we did in the proof of Theorem~\ref{thm: maingap}, we define the pseudometric $d: G'\times G' \to \RR$ by
 $$ d(g'_1, g_2') = \min\{ S/2,  \mu_{G'}(A_2') - \mu_{G'}(g_1A_2 \cap g_2A_2)\}, $$
 this is indeed a pseudometric by Proposition~\ref{prop: construct pseudo-metric}.

 Suppose $K_1K_2\delta\mu_G(A)<\eta S$ where $\eta$ is the fix constant from Section~\ref{subsec: almost linear fiber}.  By Proposition~\ref{prop: almost linear metric from local} and Proposition~\ref{prop: localmonotoneimplyglobalmonotone},  we obtain a $\gamma$-linear $\gamma$-monotone pseudometric, where $\gamma=1620\nu S$, and $\nu$ is  from Theorem~\ref{thm: fibers of same length 1newnew}. Therefore, Theorem~\ref{thm: homfrommeasurecompact} gives us a surjective continuous group homomorphism $\phi:G'\to\TT$, such that for every $g'\in\ker\phi\cap N(\lambda)$ with $\lambda=\rho/36$, we have $\mu_{G'}(A'_2\setminus gA'_2)<S/72$.  
 By replacing $\phi$ with the quotient map $G' \to G'/(\ker(\phi)_0)$ if necessary, where $(\ker(\phi)_0)$ is the identity component of $\ker(\phi)$,  we can arrange that $\phi$ has connected kernel.

Let $\chi=\pi\circ\phi$. We now determine the structure of $A$ and $B$. By Lemma~\ref{lem: making projection small},  we have $\mu_\TT(\phi(A'_2))+\mu_\TT(\phi(B'_2))<1/5$. Then by Proposition~\ref{prop: struc on AB}, there are compact intervals $I'_2, J'_2$ in $\TT$, such that
\[
\mu_\TT(I'_2)=\mu_L(A'_2)\quad\text{and}\quad\mu_\TT(J_2)=\mu_L(B'_2), 
\]
and 
\[
\max\{ \mu_{G'}(A'_2\tri\phi^{-1}(I'_2)), \mu_L(B'_2\tri\phi^{-1}(J'_2)) \} <K_0K_1K_2\delta\mu_G(A),
\]
where $K_0$ is the constant in Proposition~\ref{prop: struc on AB}.
Let $c_G$ be the constant in Lemma~\ref{lem:fromsmalltobig}. Choose $\delta$ such that $K_0K_1K_2\delta\mu_G(A)<c_G/700$. By Lemma~\ref{lem:fromsmalltobig}, there are compact intervals $I'_1,J'_1$ in $\TT$ with 
\[
\mu_\TT(I'_1)=\mu_L(A'_1)\quad\text{and}\quad\mu_\TT(J'_1)=\mu_L(B'_1), 
\]
and 
\[
 \max\{\mu_L(A'_1\tri\phi^{-1}(I'_1)), \mu_L(B'_1\tri\phi^{-1}(J'_1))\}<10K_0K_1K_2\delta\mu_G(A).   
\]
  By \eqref{eq: approx main compact}, and slighltly modifying $I'_1, J'_1$, we get compact intervals $I_1, J_1\subseteq \TT$ such that 
 \[
\mu_\TT(I_1)=\mu_G(A_1)\quad\text{and}\quad\mu_\TT(J_1)=\mu_G(B_1), 
\]
and
\[
 \max\{\mu_G(A_1\tri\chi^{-1}(I_1)), \mu_G(B_1\tri\chi^{-1}(J_1))\}<101K_0K_1K_2\delta\mu_G(A).   
\]
  By Lemma~\ref{lem:fromsmalltobig} again, there are intervals $I_0$ and $J_0$ in $\TT$, such that
\[
\mu_\TT(I_0)=\mu_G(A_0)\quad\text{and}\quad\mu_\TT(J_0)=\mu_G(B_0), 
\]
and 
\[
 \max\{\mu_G(A_0\tri\chi^{-1}(I_0)), \mu_G(B_0\tri\chi^{-1}(J_0))\} <1010K_0K_1K_2\delta\mu_G(A_0).   
\]

Next, by Lemma~\ref{lem: from sym dif to subset}, there are compact intervals $I,J\subseteq \TT$, with 
\begin{align*}
\mu_\TT(I)-\mu_G(A)&<10100K_0K_1K_2\delta\mu_G(A_0),\\
\mu_\TT(J)-\mu_G(B)&<10100K_0K_1K_2\delta\mu_G(A_0)\leq 10100K_0K_1K_2\delta\mu_G(B_0),
\end{align*}
and $A\subseteq \chi^{-1}(I)$, $B\subseteq\chi^{-1}(J)$. 
Note that all $K_0$, $K_1$, and $K_2$ only depend on $\lambda$, then one can take
\[
\delta\leq \min\Big\{\frac{1}{10100K_0K_1K_2}, \frac{\eta S}{K_1K_2},\frac{c_G}{700K_0K_1K_2}\Big\}\varepsilon,
\]
this finishes the proof. 

Finally, we consider the case when $\varepsilon=0$, that is, $\mu_G(AB)=\mu_G(A)+\mu_G(B)$. The proof follows the same argument, by replacing Proposition~\ref{prop: almost linear metric from local} with Proposition~\ref{prop: equal pseudometric linear} to construct the locally linear pseudometric, and by replacing Proposition~\ref{prop: almost linear metric from local} and Theorem~\ref{thm: homfrommeasurecompact} by Proposition~\ref{prop: strong linear}. 
\end{proof}

\section*{Acknowledgements}
  The authors would like to thank Lou van den Dries, Arturo Rodriguez Fanlo, Kyle Gannon, John Griesmer, Daniel Hoffmann, Ehud Hrushovski, Anand Pillay, Pierre Perruchaud, Daniel Studenmund, Jun Su, Jinhe Ye, and Ruixiang Zhang for valuable discussions. Part of the work was carried out while the first author was visiting the second author at the Department of Mathematics of University of Notre Dame, and he would like to thank the department for the hospitality.

\bibliographystyle{amsalpha}
\bibliography{ref}

\end{document}